\documentclass[10pt, reqno]{amsart}
\usepackage{geometry}
\usepackage{graphicx}%
\usepackage{multirow}%
\usepackage{amsmath,amssymb,amsfonts}%
\usepackage{amsthm}%
\usepackage{mathrsfs}%
\usepackage[title]{appendix}%
\usepackage{xcolor}%
\usepackage{textcomp}%
\usepackage{manyfoot}%
\usepackage{booktabs}%
\usepackage{algorithm}%
\usepackage{algorithmicx}%
\usepackage{algpseudocode}%
\usepackage{listings}%
\usepackage{enumitem}
\usepackage{bbm,bm,mathrsfs,array,esint,cases,mathtools}
\usepackage{graphics,color,eso-pic,graphicx,bbding,rotating,verbatim,newtxtext,newtxmath}
\usepackage{tikz,makecell}
\usepackage{hyperref}
\hypersetup{colorlinks=true, pdfstartview=FitV, linkcolor=blue, citecolor=red, urlcolor=red}
\usepackage{float,tikz}
\usetikzlibrary{arrows.meta, positioning}

\theoremstyle{thmstyleone}%
\newtheorem{theorem}{Theorem}[section]

\newtheorem{prop}[theorem]{Proposition}%
\newtheorem{lemma}[theorem]{Lemma}

\theoremstyle{thmstyletwo}%
\newtheorem{example}{Example}%
\newtheorem{remark}{Remark}[section]%

\theoremstyle{thmstylethree}%
\newtheorem{definition}{Definition}[section]%

\usepackage{mdwlist}
\numberwithin{equation}{section}
\allowdisplaybreaks[3]

\DeclareMathOperator*{\esssup}{ess\,sup}
\usetikzlibrary{positioning, calc}

\begin{document}

\title[A class of \(\Psi\)-bounded oscillation operators on abstract Orlicz-Morrey spaces ...]{A class of \(\Psi\)-bounded oscillation operators on abstract Orlicz-Morrey spaces endowed with a ball-basis and their applications}

\author[F. Shan]{Fate Shan}
\address{Fate Shan:\\
	School of Mathematics and System Sciences\\
	Xinjiang University\\
	Urumqi 830046 \\
	People's Repunlic of China}\email{107552400558@stu.xju.edu.cn}

\author[J. Zhou]{Jiang Zhou*} 
\address{Jiang Zho:\\
	School of Mathematics and System Sciences\\
	Xinjiang University\\
	Urumqi 830046 \\
	People's Repunlic of China}\email{zhoujiang@xju.edu.cn}

\date{\today}

\begin{abstract}
This work introduces a class of abstract Orlicz-Morrey spaces endowed with a ball-basis on general measure spaces and defines the associated concept of \(\Psi\)-bounded oscillation operators. Within this framework, we establish pointwise estimates and norm inequalities for these operators via sparse domination techniques. As applications, we verify that this class of \(\Psi\)-bounded oscillation operators includes maximal operators and Carleson-type operators on general measure spaces, as well as \(\omega\)-Calder\'on-Zygmund operators and intrinsic square operators on \(\mathbb{R}^n\), thus providing a generalization of certain classical Orlicz-Morrey spaces and their associated operator theory.
\end{abstract}

\thanks{The first named author was supported by the Research Innovation Program for Postgraduates of Xinjiang Uygur Autonomous Region (No. XJ2026G070). The second author gratefully acknowledges support from the National Natural Science Foundation of China (No.12461021).\\
	$^{*}$ Corresponding author, Email: zhoujiang@xju.edu.cn. }

\maketitle

\tableofcontents

\baselineskip 15pt


\section{Introduction}\label{Section 1}

To unify the study of diverse operators---including maximal operators, \(\omega\)-Calder\'on-Zygmund operators, intrinsic square operators, and Carleson-type operators---within a single framework, we develop the theory of abstract Orlicz-Morrey spaces endowed with a ball-basis (cf. Definition \ref{Orlicz-Morrey space}) and introduce \(\Psi\)-bounded oscillation operators (\(\Psi\)-BOOs, cf. Definition \ref{BOO}).

\subsection{Background and motivation}\label{Section 1.1}

Based on the content studied in this paper, the background in this section will be elaborated from the following three aspects.

{\bf (1) The theory of Orlicz-Morrey function spaces}

To our knowledge, multiple definitions exist for Orlicz-Morrey spaces on the underlying space \((\mathbb{R}^n,\mathcal{L}^n)\)---\(\mathcal{L}^n\) being the standard Lebesgue measure---as seen in \cite{OM-KK-1991,OM-Nakai-2004,OM-SST-2012,OM-DGS-2014,OM-H-2023}. A foundational contribution was made by Kokilashvili and Krbec \cite{OM-KK-1991}, who introduced the Orlicz-Morrey class in 1991. Subsequent developments on this class appear in \cite{OM-KK-1997,OM-KK-2004,OM-KK-2019}. In 2004, Nakai \cite{OM-Nakai-2004} formulated the Orlicz-Morrey space and showed that it generalizes \(L^p(\mathbb{R}^n,\mathcal{L}^n)\) spaces, Morrey spaces, and Orlicz spaces. Furthermore, he established the boundedness of the Hardy-Littlewood maximal operator on this space. Further work by Nakai \cite{OM-Nakai-2006} proved the boundedness of Calder\'on-Zygmund operators in Orlicz-Morrey spaces and derived modular inequalities for Orlicz spaces. Detailed structural properties were later provided in \cite{OM-Nakai2008,OM-Nakai-2011}, where Nakai verified that Orlicz-Morrey spaces contain \(L^p(\mathbb{R}^n,\mathcal{L}^n)\), \(L^\infty(\mathbb{R}^n,\mathcal{L}^n)\), (generalized) Morrey spaces, and Orlicz spaces; defined weak Orlicz-Morrey spaces; established H\"older's inequality; characterized their predual spaces; and clarified inclusion relations among different Orlicz-Morrey spaces. Additional relevant results in this theory can be found in \cite{OM-Nakai-2019,OM-Nakai-2023,OM-Nakai-2021-1,OM-Nakai-2021-2,OM-Nakai-2022,OM-Nakai-2024}.

The theory of Orlicz-Morrey spaces, when developed over a variety of underlying spaces, has given rise to numerous variants. These include central Orlicz-Morrey spaces \cite{COM-2015}, vanishing generalized Orlicz-Morrey spaces \cite{VOM-2014,VOM-2016}, Orlicz-Morrey spaces on spaces of homogeneous type \cite{OM-DGS-2019,TSP2020,OM-DGS-2021}, weighted Orlicz-Morrey spaces \cite{OM-Nakai-2021-2}, and local Orlicz-Morrey spaces \cite{LOM,OM-H-2022}, among others. For a broader overview of contributions to this theory, we refer to \cite{OM-SST-2015,OM-SST-2016,OM-SST-2021,OM-SST-2025,OM-DGS-2015-2,OM-DGS-2018,TSP2020,OM-DGS-2021,OM-DGS-2016,OM-SST-2014,OM-DGS-2015}.

{\bf (2) Abstract function spaces endowed with a ball-basis and the theory of bounded oscillation operators}

Regarding the bounded oscillation operators \(T_K\) defined on a measure space \((X,\,\mathfrak{M},\, \mu)\) endowed with a ball-basis, Karagulyan \cite{Abstract2019} established weighted estimates for \(T_K\). This work introduced, for the first time, the concept of a ball-basis for abstract Lebesgue spaces and defined the associated class of operators \(T_K\).
\begin{align}\label{Theorem 6.2}
	\|T_K\|_{L^p(\omega) \to L^p(\omega)} \lesssim c(p,\mathscr{C}_0) \left( \mathscr{C}_1(T_K) + \mathscr{C}_2(T_K) +\|T_K\|_{L^1 \to L^{1,\infty}}\right) [\omega]_{A_p}^{\max\left\{\frac{p+2}{p(p-1)},\frac{3p-2}{p}\right\}},
\end{align}
where \(\mathscr{C}_0\) is the constant from Definition \ref{ball-basis}, and \(\mathscr{C}_1(T_K),\,\mathscr{C}_2(T_K)\) are the constants appearing in Definition \ref{BOO}. We refer to this inequality as the {\sffamily K-type weighted estimate}. This estimate unifies and extends a number of classical results, including:

\begin{itemize*}
	
	\item The Hardy-Littlewood maximal operator admits a bounded extension on weighted \(L^p(\mathbb{R}^n, w\,d\mathcal{L}^n)\) spaces, a fact following the seminal work of Muckenhoupt \cite{Muckenhoupt1972} and the subsequent sharp norm estimate by Buckley \cite{Buckley1993}.

	\item Research by Coifman \cite{Coifmany1974}, followed by Dragi\v cevi\'c \cite{Dragicevic2005}, Hyt\"onen \cite{Hytonen2012}, and Lacey \cite{Lacey2017}, established the weighted \(L^p(\mathbb{R}^n, w\,d\mathcal{L}^n)\)-boundedness of Calder\'on-Zygmund type operators on \(\mathbb{R}^n\).
	
	\item Weighted bounds for the martingale transform under different measures were obtained by Lacey \cite{Lacey2017}, Thiele \cite{Thiele2015}, and others.
	
	\item Grafakos \cite{Grafakos2005}, Di Plinio \cite{Di Plinio2014}, and others established weighted boundedness results for Carleson-type operators on \(L^p(\mathbb{R}^n,\mathcal{L}^n)\) spaces.
	
\end{itemize*}

In 2023, Cao \cite{Caomingming2023} refined the definition of a ball-basis introduced in \cite{Abstract2019} and established a class of Banach-valued bounded oscillation operators on abstract Lebesgue spaces. It was verified that Calder\'on-Zygmund type operators---such as the Hardy-Littlewood maximal operator and \(\omega\)-Calder\'on-Zygmund operators---as well as operators beyond the classical Calder\'on-Zygmund theory---such as intrinsic square operators and Carleson-type operators---belong to this class. Moreover, K-type weighted estimates were also obtained. Building on this work, Zhou \cite{Zhou2025} introduced the abstract generalized Orlicz space endowed with a ball-basis in 2025 and established the boundedness of the Hardy-Littlewood maximal operator on this space. Subsequently, in 2026, we \cite{Shan2026} constructed abstract Morrey spaces endowed with a ball-basis (cf. Definition \ref{Morrey space}) and bounded oscillation operators, and derived norm estimates for such operators on these abstract ball-basis Morrey spaces.

{\bf (3) Sparse domination techniques}

Sparse domination techniques trace their origins to Lerner's foundational work in \((\mathbb{R}^n,\mathcal{L}^n)\) \cite{Lerner2013A, Lerner2016}, which introduced the key concepts of sparse sets and sparse operators. By means of these sparse operators, he achieved pointwise estimates for operators of weak type \cite[Theorem 4.2]{Lerner2016} as well as weighted norm estimates \cite[Corollary 4.6]{Lerner2016}.

Karagulyan \cite{Abstract2019} introduced two pairwise disjoint \(\beta\)-sparse families (cf. Definition \ref{sparse}) on an abstract measure space \((X,\,\mathfrak{M},\,\mu)\) endowed with a ball-basis, and thereby defined the sparse operator \(\mathcal{A}_{\mathcal{S},r}\) for \(r \geq 1\), where \(\mathcal{S}\) denotes the union of these two sparse families. This construction led to the sparse domination of the bounded oscillation operator \(T_K\).
\begin{align}\label{Theorem 1.1}
	\left| T_K f(x) \right| \lesssim \left[ \mathscr{C}_1(T_K) + \mathscr{C}_2(T_K) +\|T_K\|_{L^r \to L^{r,\infty}} \right] \mathcal{A}_{\mathcal{S}, r}f(x)\quad a.e.\;\; x \in B.
\end{align}
We refer to (\ref{Theorem 1.1}) as the {\sffamily K-type sparse domination}. The K-type sparse domination generalizes the results in \cite[Theorem 3.1, Theorem 4.2]{Lerner2016}. The development of sparse domination theory has seen several key refinements. A foundational advancement was made by Lerner \cite[Theorem 1.1]{Lerner2020}, who identified near-minimal sufficient conditions for the pointwise control of oscillation operators by sparse operators. Building on this, Cao \cite[Theorem 1.5]{Caomingming2023} later achieved a refinement of the parameter \(\beta\) within the so-called K-type sparse domination. This refined framework was subsequently applied by the authors to the setting of Morrey spaces in \cite[Theorem 1]{Shan2026}. For further developments in sparse domination, see \cite{sparse2026,Lerner2024,sparse2021,Lerner2013B,Lerner2020A}.

\vspace{10pt}

Based on the background discussion above related to the present work, the primary motivation for our research stems from the following three aspects:
\begin{itemize*}
	
	\item Recent developments in singular integral operator theory, particularly the trend toward abstraction in its formulation, continue to drive the evolution of the field. Building upon the pioneering contribution of Karagulyan \cite{Abstract2019}, the theoretical framework was subsequently advanced by Cao et al. \cite{Caomingming2023}. Their work provides a robust foundation for studying other theories on abstract Lebesgue spaces endowed with a ball-basis \cite{Abstract2021, Abstract2023} and establishes a theoretical groundwork for other types of abstract function spaces possessing a ball-basis \cite{Abstract2026, Abstract2026BMO, Shan2026, Zhou2025}. Inspired by this question of extendibility, our work aims to generalize these advances to broader classes of function spaces.

	\item Both historically emerging from distinct mathematical needs, Orlicz and Morrey spaces were developed to extend and complement the classical Lebesgue framework. The inception of Morrey spaces by Morrey \cite{Morrey1938} was driven by problems in elliptic partial differential equations, specifically to quantify the local regularity of their solutions. Independently, Orlicz spaces found their utility in analyzing the limiting behavior of the Hardy-Littlewood maximal operator near \(L^1(\mathbb{R}^n,\mathcal{L}^n)\) \cite{Kita1996,Kita1997,Cianchi1999}. The subsequent unification into Orlicz-Morrey spaces represents a synthesis of these ideas, yielding a theory of considerable breadth and depth, as evidenced by the extensive body of subsequent research \cite{Morrey1964,Wmq2024,Peetre1966,Peetre1969,Nakai2000,OM-Nakai2008,Xue2025,MorreyBook1,MorreyBook2,M_Phi1995,OM-KK-2019,OM-Nakai-2011}. A parallel and equally important aim is the construction of a rigorous, ball-basis-dependent framework for abstract Orlicz-Morrey spaces, which we believe would greatly broaden the reach of the classical theory.
	
	\item For Calder\'on-Zygmund type operators, pointwise bounds and norm estimates are commonly used to describe their behavior. Fundamentally, this is because the oscillatory nature of such operators plays a crucial role. In the Euclidean setting \((\mathbb{R}^n, \mathcal{L}^n)\), Lerner's foundational work revealed that the technique of local mean oscillation \cite{Lerner2010} and the method of sparse domination \cite{Lerner2013A,Lerner2016,Lerner2020} provide powerful tools for controlling Calder\'on-Zygmund type operators. We note that \cite[Theorem 4.2]{Lerner2016} and various forms of \(K\)-type sparse domination imply that oscillation-type operators must satisfy a weak boundedness condition in order to be pointwise dominated by sparse operators. Following this technical line, Lerner established in \cite[Theorem 1.1]{Lerner2020} nearly minimal assumptions under which sparse domination can be achieved. Building on these transformative advances that continue to define the frontier of quantitative estimation and operator theory \cite{Lerner2020A,LKW2018,Lerner2024,sparse2026,OM-sparse2025,CYPsparse2025,XQYsparse2025,Lerner2019,Xue2025}, a principal aim of this work is to adapt this powerful framework to the setting of abstract Orlicz-Morrey spaces endowed with a ball-basis. This extension will facilitate the study of both oscillation operators and sparse operators within this generalized context.
	
\end{itemize*}

\subsection{Main results and related contributions}\label{Section 1.2}

This subsection presents the principal findings and contributions of this work. Our investigation commences by establishing sparse domination within the framework of abstract Orlicz-Morrey spaces endowed with a ball-basis.

\begin{theorem}\label{main}
	Let \((X,\,\mathfrak{M},\,\mu)\) be a measure space endowed with a ball-basis \(\mathfrak{B}\) and let \(\lambda > 3\mathscr{C}_0^6\). Suppose that \((\Phi,\Psi,\phi,\psi)\in\mathscr{Y} \uplus \mathscr{G}\) and \((\Phi, \phi) \in E_{mb}\). Assume that \(T\) is a \(\mathbb{B}\)-valued \(\Psi\)-BOO with \(T \in \mathbb{W}_{\Psi, \psi, \lambda}\), and let \(\mathscr{T}\) be a \(\mathbb{B}\)-valued linear operator such that \(Tf(x) = \|\mathscr{T}f(x)\|_\mathbb{B}\). Then there exists a family \(\mathcal{S} \subset \mathfrak{B}\) which is the union of two \(\frac{1}{2 \mathscr{C}_0^3}\)-sparse families, such that for every \(f \in L^{(\Psi,\psi)}(X) \subset L^{(\Phi,\phi)}(X)\) and every \(B \in \mathfrak{B}\),
	\begin{align}\label{ine main}
		\|\mathscr{T}f(x)\|_\mathbb{B} \lesssim \mathscr{C}(T) \cdot \mathcal{A}_{\mathcal{S}, \Phi,\phi} f(x), \quad a.e. \; x \in B,
	\end{align}
	where \(\mathscr{C}(T) := \mathscr{C}_1(T) + \mathscr{C}_2(T) + \|T\|_{L^{(\Psi, \psi)}(X) \to wL^{(\Psi,\psi)}(X)}\). Here, \(\mathscr{C}_0\) is the constant appearing in Definition \ref{ball-basis}, while \(\mathscr{C}_1(T)\) and \(\mathscr{C}_2(T)\) are the constants introduced in Definition \ref{BOO}.
\end{theorem}

Building upon Theorem \ref{main}, we proceed to establish the stated norm estimates. The argument unfolds in three natural steps. First, a direct application of Definition \ref{B-condition} to a ball-basis \(\mathfrak{B}\) satisfying the Besicovitch \(\mathfrak{N}\)-condition yields the pivotal pointwise control:
\begin{align}\label{ine main2}
	\mathcal{A}_{\mathcal{S}, \Phi,\phi}f(x) \lesssim \mathfrak{N} \cdot \mathcal{M}_{\mathfrak{B},\Phi,\phi}f(x), \quad a.e. \; x \in X.
\end{align}
Next, we employ the monotonicity inherent to the Orlicz-Morrey norm, which gives that
\begin{align}\label{ine main3}
	\text{if } f(x) \leq g(x), \text{ then } \|f\|_{L^{(\Psi, \psi)}(X)} \leq \|g\|_{L^{(\Psi,\psi)}(X)}, \quad a.e. \; x \in X.
\end{align}
Finally, synthesizing Theorem \ref{main} with the two preceding inequalities, (\ref{ine main2}) and (\ref{ine main3}), leads us to the desired norm estimate, encapsulated in the following theorem.
\begin{theorem}\label{main2}
	Let \((X,\,\mathfrak{M},\,\mu)\) be a measure space endowed with a ball-basis \(\mathfrak{B}\) and let \(\lambda > 3\mathscr{C}_0^6\). Assume that \(\mathfrak{B}\) satisfies the Besicovitch \(\mathfrak{N}\)-condition, \((\Phi,\Psi,\phi,\psi)\in\mathscr{Y} \uplus \mathscr{G}\), and \((\Phi, \phi) \in E_{mb}\). If \(T\) is a \(\mathbb{B}\)-valued \(\Psi\)-BOO and \(T \in \mathbb{W}_{\Psi, \psi, \lambda}\), then
	\begin{align*}
		\|T\|_{L^{(\Psi, \psi)}(X) \to L^{(\Psi,\psi)}(X)} \lesssim \mathscr{C}(T) \cdot \mathscr{K} \cdot \mathfrak{N}.
	\end{align*}
	Here, \(\mathscr{K}\) is the constant appearing in (\ref{K}).
\end{theorem}

The principal contributions of this study are summarized as follows:

\begin{itemize*}
	
	\item Our framework of Banach-valued \(\Psi\)-BOOs unifies the Orlicz-Morrey maximal operator, \(\omega\)-Calder\'on-Zygmund operators, as well as operators beyond Calder\'on-Zygmund theory. Notably, this framework incorporates the Orlicz-Morrey maximal operator into the theory of singular integrals, while categorizing the intrinsic square operators and Carleson-type operators as examples that extend beyond classical Calder\'on-Zygmund theory.

	\item We develop a framework of abstract Orlicz-Morrey spaces endowed with a ball-basis that is both broad and unifying. This framework admits as special cases a range of abstract spaces (for example, abstract Orlicz spaces and abstract generalized Morrey spaces; cf. Appendix \ref{Section A2}) and also recovers key classical examples. The latter include central Orlicz-Morrey spaces \cite{COM-2015}, their counterparts on spaces of homogeneous type \cite{OM-DGS-2019}, and local Orlicz-Morrey spaces \cite{LOM}. The generality of the Orlicz-Morrey setting ensures that other families, such as Orlicz-type spaces and Morrey-type spaces, are also contained within this framework.

	\item Our Theorem \ref{main} establishes a pointwise estimate (a local characterization) for the \(\Psi\)-BOO in abstract Orlicz-Morrey spaces endowed with a ball-basis. This result is then used to derive Theorem \ref{main2}. The methodology here differs from classical proofs \cite{OM-Nakai-2004,OM-Nakai-2006,OM-Nakai2008,OM-DGS-2015}, as the principal difficulty stems from the lack of geometric structure, such as the one inherent in \((\mathbb{R}^n,\mathcal{L}^n)\), in general measure spaces. Consequently, the operator boundedness results obtained differ from their classical counterparts. Nevertheless, the required inequalities are successfully established via Theorems \ref{main} and \ref{main2}.

	\item As an application, we provide several classical examples within Orlicz-Morrey spaces in Sect. \ref{Section 2}: the Orlicz-Morrey maximal operator, \(\omega\)-Calder\'on-Zygmund operators, intrinsic square operators, and Carleson-type operators. Utilizing the results from Theorem \ref{main} and Theorem \ref{main2}, our framework recovers the results presented in \cite{OM-Nakai-2004,OM-Nakai-2006,OM-Nakai2008,OM-DGS-2015}.
	
\end{itemize*}

\subsection{Basic assumptions and related concepts}\label{Section 1.3}

Our approach is to transcend the limitations of Euclidean geometry. We employ a generalized metric structure and replace the conventional geometric building blocks with a flexible "ball-basis"—a family of measurable sets. This leads to a unified analytic framework designed for broad application.

\begin{definition}{\cite{Caomingming2023}}\label{ball-basis}
Let \((X,\,\mathfrak{M},\,\mu)\) be a measure space. A family of sets \(\mathfrak{B} \subset \mathfrak{M}\) is called a {\sffamily ball-basis} if it satisfies the following properties:  
	\begin{itemize*}
		\item[\rm B1]~~(Basis property): Every \(B \in \mathfrak{B}\) satisfies \(0 < \mu(B) < \infty\).  

		\item[\rm B2]~~(Covering property): For any two points \(x,\,y \in X\), there exists some \(B \in \mathfrak{B}\) containing both \(x\) and \(y\).  

		\item[\rm B3]~~(Approximation property): For every measurable set \(E \in \mathfrak{M}\) and every \(\varepsilon > 0\), one can find a sequence of balls \(\{B_k\} \subset \mathfrak{B}\) (finite or infinite) such that \(\mu(E\Delta\bigcup_kB_k)\textless\varepsilon\).

		\item[\rm B4]~~(Ball-hull property): For each \(B \in \mathfrak{B}\) there exists a ball \(B^\P \in \mathfrak{B}\), called the {\sffamily ball-hull} of \(B\), satisfying
		\[
		\bigcup_{\substack{B' \in \mathfrak{B}: \, B' \cap B \neq \varnothing \\ \mu(B') \leq 2\mu(B)}} B' \subset B^\P,
		\]
		together with the implication  
		\[
		A \subset B \;\Longrightarrow\; A^\P \subset B^\P,
		\]
		and the measure estimate  
		\[
		\mu(B^\P) \leq \mathscr{C}_0 \mu(B),
		\]
		where \(\mathscr{C}_0 \geq 1\) is a fixed constant.  
	\end{itemize*}
\end{definition}

In this case, we say that \((X,\,\mathfrak{M},\,\mu)\) is a measure space endowed with a ball-basis \(\mathfrak{B}\).

\begin{remark}\label{example ball-basis}
	The following examples are some common geometric structures, all of which are ball-basis.
	\begin{itemize}[leftmargin=1.5em]
		
		\item \textbf{Euclidean ball}: In \(\mathbb{R}^n\), consider the family of balls \(B(x_k,\,r) = \left\{x \in \mathbb{R}^n : \|x_k - x\|_{l^2} \leq r,\; r>0 \right\}\), denoted by \[\mathfrak{B}_E = \left\{B(x_k,\,r): x_k \in \mathbb{R}^n,\; k=1,2,\dots\right\}.\]
		Then \(\mathfrak{B}_E\) is a ball-basis in \((\mathbb{R}^n,\mathcal{L}^n)\). Here \(\|x - y\|_{l^2} = \left( \sum_{i=1}^n (x_i - y_i)^2 \right)^{1/2}\), with \(x = (x_1,\dots, x_n),\; y = (y_1,\dots, y_n) \in \mathbb{R}^n\). In fact, conditions B1, B2 and B3 are obvious. For \(\ell \geq 1+2^{1+1/n}\), take \(\mathscr{C}_0 = \ell^n\). Then for every \(B \in \mathfrak{B}_E\), we have \(B^{\P} = \ell B\), so condition B4 is satisfied.

		\item \textbf{Dyadic lattice}: In the space \((\mathbb{R}^n,\mathcal{L}^n)\), a dyadic lattice \(\mathfrak{D}\) is a ball-basis. A dyadic lattice \(\mathfrak{D}\) in \(\mathbb{R}^n\) is a collection of cubes \(Q\) satisfying the following properties:
		\begin{itemize}
			\item[\(-\)] {\it Closure property}: Whenever a cube \(Q\) belongs to \(\mathfrak{D}\), so do all cubes obtained by successively halving its sides (its dyadic descendants). Denoting the descendant set by \(D(Q)\), this means \(D(Q) \subset \mathfrak{D}\).
			
			\item[\(-\)] {\it Existence property}: Any two cubes \(Q_1, \, Q_2 \in \mathfrak{D}\) admit a cube \(Q \in \mathfrak{D}\) of which both are dyadic descendants, meaning that \(Q_1, \, Q_2 \in D(Q)\).
			
			\item[\(-\)] {\it Covering property}: For each compact set \(K \subset \mathbb{R}^n\), one can find a cube \(Q \in \mathfrak{D}\) that contains \(K\) (so \(K \subset Q\)).
		\end{itemize}

	\end{itemize}
	
	However, it should be noted that the following example does not meet the requirements of a ball-basis.
			
	\vspace{4pt}
			
	\begin{itemize}[leftmargin=1.5em]
					
			\item \textbf{Zygmund rectangles}: Investigate the family \(\mathfrak{B}\) comprising all Zygmund rectangles in \(\mathbb{R}^3\), defined to be rectangles with sides parallel to the coordinate axes and side lengths proportional to \(s\), \(t\), and \(st\) for some \(s,\,t > 0\). This family does not constitute a ball-basis for \((\mathbb{R}^3,\,\mathcal{L}^3)\). For a counterexample, one may take the cube
			\[
			B = [0,1)^3,~~~\text{and}~~~B_k = [0, 2^{k+1/2}) \times [0, 2^{-k}) \times [0, 2^{1/2}),\,k = 0,1,2,....
			\]
			Each \(B_k\) belongs to \(\mathfrak{B}\) and satisfies \(B_k \cap B \neq \varnothing\) and \(|B_k| = 2|B|\). However, no Zygmund rectangle in \(\mathfrak{B}\) can contain the union \(\bigcup_{k=0}^\infty B_k\). Hence condition B4 of a ball-basis is violated.

			\item \textbf{Centrally shrinking cube sequence}: Consider the subset of \(\mathbb{R}^n\) given by the cube \[Q = [a, b]^n = \{(x_1, \dots, x_n) : x_i \in [a, b],\, i=1,2,\dots, n \},\] where \(a < b\), and the corresponding space \((Q,\mathcal{L}^n)\). The centrally shrinking cube sequences are defined as
			\[Q_k=\left[a+\sum_{l=1}^k\frac{b-a}{2^{l+1}},b-\sum_{l=1}^k\frac{b-a}{2^{l+1}}\right]^n,\,\,\widetilde{Q}_k=\left[a+\frac{b-a}{2^{k+1}},b-\frac{b-a}{2^{k+1}}\right]^n,\quad k=1,\,2,\,....\]
			The family \(\mathfrak{Q}=\{Q_k,\,\widetilde{Q}_k:\,k=1,\,2,\,...\}\) does not satisfy condition B3 of a ball-basis. Therefore, \(\mathfrak{Q}\) is not a ball-basis for \((Q,\mathcal{L}^n)\).
					
		\end{itemize}
\end{remark}

In order to introduce a new class of bounded oscillation operators, the necessary notation is first established. Let \((X,\,\mathfrak{M},\,\mu)\) be a measure space endowed with a ball-basis \(\mathfrak{B}\), and let \(\Phi \in \mathscr{Y}\) and \(\phi \in \mathscr{G}\) be given as in (\ref{Y function}) and (\ref{G function}), respectively. We denote
\begin{align*}
	\langle \|f\|\rangle_{\Phi,\phi,B}=\sup_{B'\in\mathfrak{B}:B' \supset B}\|f\|_{\Phi,\phi,B'}, 
\end{align*}
where \(\|f\|_{\Phi,\phi,B}\) is the generalized Luxemburg norm as defined in Appendix \ref{Section A2}.

\begin{remark}
	Below we introduce simplified notations for specific choices of \(\Phi\) and \(\phi\).
	\begin{itemize*}
		\item When \(\Phi(t)=t^r\), we denote \(\langle \|f\|\rangle_{\Phi,\phi,B}:=\langle \|f\|\rangle_{r,\phi,B}\) and \(\|f\|_{\Phi,\phi,B}=\|f\|_{r,\phi,B}\).
		\begin{itemize*}
			\item If \(\phi(t)=1\), we further write \(\langle \|f\|\rangle_{r,\phi,B}:=\langle \|f\|\rangle_{r,B}\) and \(\|f\|_{r,\phi,B}=\|f\|_{r,B}\).
		\end{itemize*}
		
		\item When \(\Phi(t)=t\), we denote \(\langle \|f\|\rangle_{\Phi,\phi,B}:=\langle \|f\|\rangle_{\phi,B}\) and \(\|f\|_{\Phi,\phi,B}=\|f\|_{\phi,B}\).
		\begin{itemize*}
			\item If \(\phi(t)=1\), we further write \(\langle \|f\|\rangle_{\phi,B}:=\langle \|f\|\rangle_{B}\) and \(\|f\|_{\phi,B}=\|f\|_{B}\).
		\end{itemize*}
	\end{itemize*}
	Here \(\|f\|_{r,\phi,B}\), \(\|f\|_{r,B}\), \(\|f\|_{\phi,B}\), and \(\|f\|_{B}\) are respectively defined as:
	\begin{figure}[H]
		\centering
		\begin{tikzpicture}[
			node distance=2cm,
			box/.style={draw, rectangle, minimum width=1.5cm, minimum height=1cm},
			arrow/.style={-Stealth, thick}
			]
			\node (1) {\(\displaystyle\|f\|_{r,\phi,B} = \left(\frac{1}{\mu(B)\phi(\mu(B))} \int_{B} |f(x)|^r d\mu\right)^{\frac{1}{r}}\)};
			\node[right=2cm of 1] (2) {\(\displaystyle\|f\|_{r,B} = \left(\frac{1}{\mu(B)} \int_{B} |f(x)|^r d\mu\right)^{\frac{1}{r}}\)};
			\node[below=1.5cm of 1] (3) {\(\displaystyle\|f\|_{\phi,B} = \frac{1}{\mu(B)\phi(\mu(B))} \int_{B} |f(x)| d\mu\)};
			\node at (2 |- 3) (4) {\(\displaystyle\|f\|_{B} = \frac{1}{\mu(B)} \int_{B} |f(x)| d\mu\)};

			\draw[arrow] (1) -- node[above] {\(\phi(t)=1\)} (2);
			\draw[arrow] (3) -- node[above] {\(\phi(t)=1\)} (4);
			\draw[arrow] (1) -- node[left] {\(r=1\)} (3);
			\draw[arrow] (2) -- node[right] {\(r=1\)} (4);
			
		\end{tikzpicture}
		\caption{The relationships between \(\|f\|_{r,\phi,B}\)}
	\end{figure}
\end{remark}

\begin{definition}\cite{CFA}
	Let \((X,\,\mathfrak{M},\,\mu)\) be a measure space. Denote by \(\mathscr{L}_0(X,\,\mathfrak{M},\,\mu)\) the set of all real-valued measurable functions on \(X\), and let \(\mathscr{U}(X,\,\mathfrak{M},\,\mu)\) be a suitable linear subspace of \(\mathscr{L}_0(X,\,\mathfrak{M},\,\mu)\). Consider an operator \(T: \mathscr{U}(X,\,\mathfrak{M},\,\mu) \to \mathscr{L}_0(X,\,\mathfrak{M},\,\mu)\).
	\begin{itemize*}
		\item \(T_l\) is said to be {\sffamily linear} if for all \(f,\,g\in \mathscr{U}(X,\,\mathfrak{M},\,\mu)\) and all \(\alpha \in \mathbb{R}\),  
		\begin{align*}
			T_l(\alpha \; f)(x) = \alpha \; T_lf(x),\quad \text{and}\quad T_l(f + g) = T_l(f) + T_l(g).
		\end{align*}
		
		\item \(T_{sl}\) is called {\sffamily sublinear} if for all \(f,\,g\in \mathscr{U}(X,\,\mathfrak{M},\,\mu)\) and all \(\alpha \in \mathbb{R}\),
		\begin{align*}
			|T_{sl}(\alpha \; f)(x)| = |\alpha| \; |T_{sl}f(x)|,\quad \text{and}\quad |T_{sl}(f + g)(x)| \leq |T_{sl}f(x)| + |T_{sl}g(x)|.
		\end{align*}
		
		\item A sublinear operator \(T_{sl}\) is said to be {\sffamily linearizable} if there exist a Banach space \(\mathbb{B}\) and a \(\mathbb{B}\)-valued linear operator \(\mathscr{T}_\mathbb{B}: \mathscr{U}(X, \mathfrak{M}, \mu) \to \mathscr{L}_0(X,\,\mathfrak{M},\,\mu;\, \mathbb{B})\) such that for every \(f \in \mathscr{U}(X,\,\mathfrak{M},\,\mu)\),
		\begin{align*}
			T_{sl}f(x) = \left\|\mathscr{T}_\mathbb{B}f(x)\right\|_{\mathbb{B}} \quad \text{for } \mu\text{-almost every } x \in X.
		\end{align*}
	\end{itemize*}
\end{definition}

The following property was established in \cite[Proposition 3.2]{OM-Nakai2008}: Let \((X,\,\mathfrak{M},\,\mu)\) be a measure space. For Young functions \(\Phi,\,\Psi\in\mathscr{Y}\) and \(\phi,\,\psi\in\mathscr{G}\), under the assumptions  
\begin{align}\label{Y-G condition}
	\Phi(t) \leq \Psi(C(\Phi,\Psi)t) \quad \text{and} \quad \psi(t) \leq C(\psi,\phi)\phi(t),
\end{align}
the following hold:
\begin{itemize}
	\item \textbf{Inclusion relation:} \[L^{(\Phi,\phi)} \; \supset L^{(\Psi,\phi)} \; \supset L^{(\Psi,\psi)}\].
	
	\item \textbf{Norm estimate:} \[\|f\|_{L^{(\Phi,\phi)}} \; \leq C(\Phi,\Psi) \, \|f\|_{L^{(\Psi,\phi)}} \;\leq C(\Phi,\Psi) \max\{1, \, C(\psi,\phi)\} \, \|f\|_{L^{(\Psi,\psi)}}.\]
\end{itemize}
Setting
\begin{align}\label{K}
	\mathscr{K} :=C(\Phi,\Psi)\max\{1,\,C(\psi,\phi)\}>0,
\end{align}
we obtain the bound \(\|f\|_{L^{(\Phi,\phi)}} \leq \mathscr{K}  \|f\|_{L^{(\Psi,\psi)}}\).

\begin{definition}\label{assumption1}
	A quadruple \((\Phi,\Psi,\phi,\psi)\) is said to satisfy the \(\mathscr{Y} \uplus \mathscr{G}\)-condition, written \((\Phi,\Psi,\phi,\psi)\in\mathscr{Y} \uplus \mathscr{G}\), provided that relation (\ref{Y-G condition}) is fulfilled, where \(\Phi,\,\Psi\in\mathscr{Y}\) and \(\phi,\,\psi\in\mathscr{G}\).
\end{definition}

\begin{definition}\label{BOO}
	Let \((X,\,\mathfrak{M},\,\mu)\) be a measure space endowed with a ball-basis \(\mathfrak{B}\), and let \((\Phi,\Psi,\phi,\psi) \in \mathscr{Y} \uplus \mathscr{G}\). Given a Banach space \(\mathbb{B}\), we say that an operator \(T\) is a {\sffamily \(\mathbb{B}\)-valued \(\Psi\)-bounded oscillation operator} {\rm (\(\Psi\)-BOO)} with respect to \(\mathfrak{B}\) (and \((\Phi,\Psi,\phi,\psi)\in\mathscr{Y} \uplus \mathscr{G}\)) if there exist constants \(\mathscr{C}_1(T), \mathscr{C}_2(T) \in (0, \infty)\) and a \(\mathbb{B}\)-valued linear operator \(\mathscr{T}\) satisfying \[Tf(x) = \|\mathscr{T}f(x)\|_\mathbb{B} \quad (\forall x \in X),\] or, in the case \(\mathbb{B} = \mathbb{R}\), a single real-valued (sub)linear operator \(\mathscr{T} := \{T\}\), such that for every \(f \in L^{(\Psi,\psi)}(X)\subset L^{(\Phi,\phi)}(X)\) the following two conditions hold:
	
	\begin{enumerate}[label=(\roman*), leftmargin=*, align=left, itemsep=0.5\baselineskip]
		\item[{\rm (\(\Psi\)-BOO-I)}] For every \(B_0 \in \mathfrak{B}\) with \(B_0^\P \subsetneq X\), there exists a ball \(B \in \mathfrak{B}\) with \(B \supsetneq B_0\) such that
		\begin{align*}
			\sup_{x \in B_0} \left\| \mathscr{T}(f\,\mathbb{I}_{B^\P})(x) - \mathscr{T}(f\,\mathbb{I}_{B_0^\P})(x) \right\|_{\mathbb{B}}
			\leq \mathscr{C}_1(T) \,\|f\|_{\Phi,\phi,B^\P}.
		\end{align*}
		
		\item[{\rm (\(\Psi\)-BOO-II)}] For every ball \(B \in \mathfrak{B}\),
		\begin{align*}
			\sup_{x,x' \in B} \left\| \left(\mathscr{T}f - \mathscr{T}(f \,\mathbb{I}_{B^\P})\right)(x) - \left(\mathscr{T}f - \mathscr{T}(f\, \mathbb{I}_{B^\P})\right)(x') \right\|_{ \mathbb{B}} \leq \mathscr{C}_2(T) \, \langle \|f\|\rangle_{\Phi,\phi,B}.
		\end{align*}
	\end{enumerate}
	When \(T\) is real-valued, which is equivalent to \(\mathbb{B} = \mathbb{R}\), we will omit the reference to \(\mathbb{B}\) and the norm \(\|\cdot\|_{\mathbb{B}}\).
\end{definition}

\begin{remark}
	We note that the average integral in function spaces describes the average oscillation behavior of operators within local regions, whereas the maximal operator controls the upper bounds of these averages. We refine the control constants for Karagulyan-type bounded oscillation operators into the form of the Luxemburg norm. When \(\Phi(t)=\Psi(t)=t^r\), \(\phi(t)=1\), and \(\psi(t)=1/t\) (with \(r,\,t\geq1\)), this corresponds exactly to \cite[Definition 1.3]{Abstract2019} and \cite[Definition 1.4]{Caomingming2023} (for \(m=1\)).
\end{remark}

\begin{definition}\label{Maximal}
	Let \((X,\,\mathfrak{M},\,\mu)\) be a measure space endowed with a ball-basis \(\mathfrak{B}\). Given \(\Phi\in\mathscr{Y}\) and \(\phi\in\mathscr{G}\), the Orlicz-Morrey maximal operator is defined as
	\begin{align}\label{O-M maximal operator}
		\mathcal{M}_{\mathfrak{B},\Phi,\phi}f(x) := \sup_{B\in\mathfrak{B}:\,x\in B}\|f\|_{\Phi,\phi,B},
	\end{align}
	for \(x \in \bigcup_{B\in\mathfrak{B}}B \subset X\) and we set \(\mathcal{M}_{\mathfrak{B},\Phi,\phi}f(x)=0\) for \(x \notin \bigcup_{B\in\mathfrak{B}}B\).
\end{definition}

\begin{remark}
	
	The Orlicz-Morrey maximal operator \(\mathcal{M}_{\mathfrak{B}, \Phi, \phi}\) in (\ref{O-M maximal operator}) is clearly a sublinear operator. When \(\Phi(t)\) and \(\phi(t)\) are specified as concrete functions, we simplify the notation for the above Orlicz-Morrey maximal operator \(\mathcal{M}_{\mathfrak{B}, \Phi, \phi}\) by using the symbols given in Fig.~\ref{M-picture}.

	\begin{figure}[H]
		\centering
		\begin{tikzpicture}[
			node distance=2cm,
			box/.style={draw, rectangle, minimum width=1.5cm, minimum height=1cm},
			arrow/.style={-Stealth, thick}
			]
			\node (1) {\(\mathcal{M}_{\mathfrak{B},\Phi,\phi}\,\,\)};
			\node[right=1.8cm of 1] (2) {\(\,\,\mathcal{M}_{\mathfrak{B},\Phi}\,\,\)};
			\node[right=1.8cm of 2] (3) {\(\,\,\mathcal{M}_{\mathfrak{B},r}\,\,\)};
			\node[right=1.8cm of 3] (4) {\(\,\,M_{\mathfrak{B}}\)};
			
			\draw[arrow] (1) -- node[above] {(1)} (2);
			\draw[arrow] (2) -- node[above] {(2)} (3);
			\draw[arrow] (3) -- node[above] {(3)} (4);
		\end{tikzpicture}
		\caption{Different forms of the maximal operator \(\mathcal{M}_{\mathfrak{B},\Phi,\phi}\)\label{M-picture}}
	\end{figure}
	
	\begin{itemize*}
		\item[(1)] When \(\phi(t)=1\) (cf. \cite[Definition 1.1]{M_Phi1995}),
		\[
		\mathcal{M}_{\mathfrak{B},\Phi,\phi}:=\mathcal{M}_{\mathfrak{B},\Phi}
		=\sup_{B\in\mathfrak{B},\,x\in B} \|f\|_{\Phi,B}.
		\]
		
		\item[(2)] When \(\phi(t)=1\) and \(\Phi(t)=t^r\) (as in \cite[formula (4.1)]{Abstract2019}),
		\[
		\mathcal{M}_{\mathfrak{B},\Phi,\phi}:=\mathcal{M}_{\mathfrak{B},r}
		=\sup_{B\in\mathfrak{B},\,x\in B} 
		\left(\frac{1}{\mu(B)} \int_B |f|^r \, d\mu\right)^{1/r}.
		\]
		
		\item[(3)] When \(\phi(t)=1\) and \(\Phi(t)=t\) (following \cite[Definition 2.1.1, Definition 2.1.3]{CFA}),
		\[
		\mathcal{M}_{\mathfrak{B},\Phi,\phi}:=M_{\mathfrak{B}}
		=\sup_{B\in\mathfrak{B},\,x\in B} \frac{1}{\mu(B)} \int_B |f| \, d\mu.
		\]
	\end{itemize*}
\end{remark}

\begin{definition}\cite{Abstract2019}\label{sparse}
	A collection \(\mathcal{S} \subset \mathfrak{B} \) is called {\sffamily \(\beta\)-sparse}, with \(\beta \in (0, 1)\), if for each \( B \in \mathcal{S} \) there exists a subset \( E_B \subset B \) satisfying the following two properties:
	\begin{itemize*}
		\item \(\mu(E_B) \geq \beta \mu(B)\);
		\item the family \(\{E_B : B \in \mathcal{S}\}\) is pairwise disjoint.
	\end{itemize*}
\end{definition}

Given a sparse family \(\mathcal{S}\), \(\Phi \in \mathscr{Y}\) and \(\phi \in \mathscr{G}\), we define a linear operator by  
\begin{align}\label{sparse operator}
	\mathcal{A}_{\mathcal{S},\Phi,\phi}f(x) := \sum_{B \in \mathcal{S}} \|f\|_{\Phi,\phi,B} \cdot \mathbb{I}_B(x),
\end{align}
and refer to \(\mathcal{A}_{\mathcal{S},\Phi,\phi} \) as a {\sffamily sparse operator}. When \(\phi(t)=1\), we have the following simplified notations:
\begin{figure}[H]
	\centering
	\begin{tikzpicture}[
		node distance=2cm,
		box/.style={draw, rectangle, minimum width=1.5cm, minimum height=1cm},
		arrow/.style={-Stealth, thick}
		]
		
		\node (1) {\(\mathcal{A}_{\mathcal{S},\Phi,\phi}\,\)};
		\node[right=4cm of 1] (2) {\(\,\mathcal{A}_{\mathcal{S},r}\,\)};
		\node[right=4cm of 2] (3) {\(\,\mathcal{A}_{\mathcal{S}}\)};
		
		\draw[arrow] (1) -- node[above] {\(\phi(t)=1,\,\Phi(t)=t^r\)} (2);
		\draw[arrow] (2) -- node[above] {\(\phi(t)=1,\,\Phi(t)=t\)} (3);
	\end{tikzpicture}
	\caption{Different forms of the sparse operator \(\mathcal{A}_{\mathcal{S},\Phi,\phi}\)\label{A-picture}}
\end{figure}

Inspired by the work of Lerner and Ombrosi \cite{Lerner2020}, we introduce the key condition required in this paper, namely the \(\mathbb{W}_{\Psi, \psi}\) condition, which is defined as follows.

\begin{definition}\label{weak condition}
	Let \((X,\,\mathfrak{M},\,\mu)\) be a measure space endowed with a ball-basis \(\mathfrak{B}\), and let \((\Phi,\Psi,\phi,\psi) \in \mathscr{Y} \uplus \mathscr{G}\). Let \(\mathbb{B}\) be a Banach space and \(T\) be a \(\mathbb{B}\)-valued sublinear operator that is bounded from \(L^{(\Psi,\psi)}(X)\) to \(wL^{(\Psi,\psi)}(X)\). We say that \(T \in \mathbb{W}_{\Psi, \psi, \lambda}\) (or simply write \(T \in \mathbb{W}_{\Psi, \psi}\)) if for any balls \(A,B \in \mathfrak{B}\) with \(A \subset B\), any function \(f \in L^{(\Psi,\psi)}(X)\), and any \(0 < \lambda < 1\), the following inequality holds:
	\begin{align}\label{D1.6}
		\mu \left( \left\{x \in A : \|\mathscr{T}(f\,\mathbb{I}_B)(x)\|_\mathbb{B} > \Psi^{-1}\left( \frac{\psi(\mu(A))}{\lambda} \right) \|T\|_{L^{(\Psi,\psi)} \to wL^{(\Psi,\psi)}} \|f\|_{\Phi,\phi,A} \right\} \right)\leq \lambda \mu(A).
	\end{align}
\end{definition}

\begin{remark}
	The spaces \(L^{(\Psi,\psi)}(X)\) and \(wL^{(\Psi,\psi)}(X)\) appearing in this paper are respectively the abstract Orlicz-Morrey space endowed with a ball-basis (cf. Definition \ref{Orlicz-Morrey space}) and the abstract weak Orlicz-Morrey space endowed with a ball-basis (cf. Definition \ref{weak Orlicz-Morrey space}).
\end{remark}

Finally, we introduce the following geometric condition, which enables a pointwise domination between the sparse operator \(\mathcal{A}_{\mathcal{S}, \Phi,\phi}\) and the maximal operator \(\mathcal{M}_{\mathfrak{B}, \Phi, \phi}\).

\begin{definition}{\cite{Abstract2019}}\label{B-condition}
	Let \((X,\,\mathfrak{M},\,\mu)\) be a measure space endowed with a ball-basis \(\mathfrak{B}\). We say that the ball-basis \(\mathfrak{B}\) satisfies the  {\sffamily Besicovitch \(\mathfrak{N}\)-condition} with a constant \(\mathfrak{N} \in \mathbb{N}_+\) if for every family \(\mathfrak{A} \subset \mathfrak{B}\), there exists a subfamily \(\mathfrak{A}' \subset \mathfrak{A}\) such that
	\begin{align*}
		\bigcup_{B \in \mathfrak{A}} B = \bigcup_{B \in \mathfrak{A}'} B, \quad \text{and} \quad \sum_{B \in \mathfrak{A}'} \mathbb{I}_B(x) \leq \mathfrak{N}, \quad x \in X.
	\end{align*}
\end{definition}

\subsection{Structure of the article}\label{Section 1.4}

The organization of the subsequent material is as follows. Concrete instances of \(\Psi\)-BOOs are developed in Section \ref{Section 2}, alongside refinements of select classical theorems, to highlight the applicability of our principal results. Section \ref{Section 3} commences with a summary of established facts about the ball-basis and proceeds to examine fundamental properties of \(\Psi\)-BOOs, culminating in the proof of a sparse domination principle for these operators. A supplementary segment offers elaboration on several concepts and properties employed in Section 3. Appendix \ref{Section A1} recalls two significant function classes within Orlicz-Morrey spaces and supplies pertinent illustrative examples. In Appendix \ref{Section A2}, we first revisit the conventional Luxemburg norm and its associated properties (Propositions \ref{propA}, \ref{propB}, and \ref{propC}). The appendix then furnishes the definition of abstract (weak) Orlicz-Morrey spaces endowed with a ball-basis, demonstrating that specific choices for the functions \(\Phi(t)\) and \(\phi(t)\) recover a spectrum of known spaces, such as Orlicz spaces and generalized Morrey spaces.

\section{Specific examples}\label{Section 2}

\subsection{Orlicz-Morrey maximal operators}\label{Section 2.1}

The Orlicz-Morrey maximal operator required in this subsection has already been given in Definition \ref{Maximal}.

\begin{theorem}
	Let \((\Phi,\Psi,\phi,\psi)\in\mathscr{Y} \uplus \mathscr{G}\) and let \((X,\,\mathfrak{M},\,\mu)\) be a measure space endowed with a ball-basis \(\mathfrak{B}\). Then the operator \(\mathcal{M}_{\mathfrak{B},\Phi,\phi}\) is a \(\Psi\)-BOO with respect to \(\mathfrak{B}\).
\end{theorem}
\begin{proof}
	Choose an arbitrary ball \(B_0\in\mathfrak{B}\) and a point \(x\in B_0\). Define
	\begin{align*}
		b := \inf_{B \in \mathfrak{B}^\P} \mu(B), \qquad \text{where } \mathfrak{B}' := \{B \in \mathfrak{B}: B \cap B_0 \neq \varnothing,\ \mu(B) > \mu(B_0)\}.
	\end{align*}
	Select a ball \(B_1 \in \mathfrak{B}^\P\) such that \(b \leq \mu(B_1) < 2b \). Set \(B := B_1^\P\). By ball-hull control property we obtain
	\begin{align}\label{2.1}
		B_0 \subsetneq B_1^\P = B.
	\end{align}
	Note that for any functions \(f \in L^{(\Psi,\psi)}(X,\,\mu)\subset L^{(\Phi,\phi)}(X,\,\mu)\), and for any number \(\varepsilon\) satisfying \(0 < \varepsilon < \|f\|_{\Phi,\phi,B^\P}\), there exists a ball \(B_2 \in \mathfrak{B}\) containing \(x\) for which
	\begin{align}\label{2.2}
		\mathcal{M}_{\mathfrak{B},\Phi,\phi}(f \cdot \mathbb{I}_{B^\P})(x) = \sup_{\widetilde{B} \in \mathfrak{B}:\, \widetilde{B} \ni x} \|f \cdot \mathbb{I}_{B^\P}\|_{\Phi, \phi, \widetilde{B}} \leq \varepsilon + \|f \cdot \mathbb{I}_{B^\P}\|_{\Phi, \phi, B_2}.
	\end{align}
	\begin{itemize}[leftmargin=2.3em]
		\item Suppose \( \mu(B_2) \leq \mu(B_0) \). Then ball-hull control property together with (\ref{2.1}) gives \(B_2 \subset B_0^\P \subset B^\P\).
		Consequently,
		\begin{align}\label{2.3}
			\mathcal{M}_{\mathfrak{B},\Phi,\phi}(f \cdot \mathbb{I}_{B_0^\P})(x) \geq \|f \cdot \mathbb{I}_{B_0^\P}\|_{\Phi, \phi, B_2} = \|f\|_{\Phi, \phi, B_2}.
		\end{align}
		Combining this with (\ref{2.2}) and (\ref{2.3}) we obtain
		\begin{align*}
			\bigl|\mathcal{M}&_{\mathfrak{B},\Phi,\phi}(f\cdot\mathbb{I}_{B^\P})(x)-\mathcal{M}_{\mathfrak{B},\Phi,\phi}(f\cdot\mathbb{I}_{B_0^\P})(x) \bigr| \\
			&= \mathcal{M}_{\mathfrak{B}, \Phi, \phi}(f \cdot \mathbb{I}_{B^\P}) (x)-\mathcal{M}_{\mathfrak{B},\Phi,\phi}(f\cdot\mathbb{I}_{B_0^\P})(x) \\
			&\leq \varepsilon + \|f \cdot \mathbb{I}_{B^\P}\|_{\Phi, \phi, B_2} - \|f\|_{\Phi, \phi, B_2}
			= \varepsilon \leq \|f\|_{\Phi,\phi,B^\P}.
		\end{align*}
		
		\item Consider the alternative case \(\mu(B_2) > \mu(B_0)\). Then, by definition, \(B_2 \in \mathfrak{B}'\) and hence \(\mu(B_2) \geq b\). Moreover,
		\begin{align}\label{2.4}
			\mu(B^\P) = \mu(B_1^{\P\P}) \leq \mathscr{C}_0^2 \mu(B_1) \leq 2\mathscr{C}_0^2 b \leq 2\mathscr{C}_0^2 \mu(B_2).
		\end{align}
		Observing that \(B_2 \cap B^\P \ne \varnothing\), we therefore apply Proposition \ref{propA} in combination with (\ref{2.2}) and (\ref{2.4}) to deduce that
		\begin{align*}
			\bigl|\mathcal{M}_{\mathfrak{B},\Phi,\phi}&(f\cdot\mathbb{I}_{B^\P})(x)-\mathcal{M}_{\mathfrak{B},\Phi,\phi}(f\cdot\mathbb{I}_{B_0^\P})(x)\bigr| \\
			&\leq \mathcal{M}_{\mathfrak{B}, \Phi, \phi} (f \cdot \mathbb{I}_{B^\P}) (x) \leq \varepsilon + \|f \cdot \mathbb{I}_{B^\P}\|_{\Phi, \phi, B_2} \\
			&\leq \|f\|_{\Phi, \phi, B^\P} + \max\{1, C^{-1}_{dc}\} \max\{1, C_{ic}\} 2 \mathscr{C}_0^2 \|f\|_{\Phi,\phi,B^\P} \lesssim \|f\|_{\Phi, \phi, B^\P}.
		\end{align*}
	\end{itemize}
	Thus condition (\(\Psi\)-BOO-I) is satisfied.
	
	To proceed, fix a ball \(B \in \mathfrak{B}\), two points \(x,\,x'\in B\), and a nonzero function \(f \in L^{(\Psi,\psi)}(X,\,\mu)\subset L^{(\Phi,\phi)}(X,\,\mu)\). By the definition of the Orlicz-Morrey maximal operator,
	\begin{align}\label{2.5}
		\mathcal{M}_{\mathfrak{B},\Phi,\phi}f(x) \leq \|f\|_{\Phi,\phi,A} + \langle \|f\| \rangle_{\Phi,\phi,B},
	\end{align}
	for some ball \(A \in \mathfrak{B}\) containing \(x\).  
	\begin{itemize}[leftmargin=2.3em]
		\item If \(\mu(A) \leq \mu(B)\), ball-hull control property implies \( A \subset B^\P \), and therefore
		\begin{align}\label{2.6}
			\mathcal{M}_{\mathfrak{B},\Phi,\phi}(f\cdot\mathbb{I}_{B^\P})(x) \geq \|f\cdot\mathbb{I}_{B^\P}\|_{\Phi,\phi,A} = \|f\|_{\Phi,\phi,A} ,
		\end{align}
		Combining (\ref{2.5}) and (\ref{2.6}), we obtain
		\begin{align*}
			\bigl|\mathcal{M}_{\mathfrak{B},\Phi,\phi}f(x)-&\mathcal{M}_{\mathfrak{B},\Phi,\phi}(f\cdot\mathbb{I}_{B^\P})(x)\bigr| \\
			&=\mathcal{M}_{\mathfrak{B},\Phi,\phi}f(x)-\mathcal{M}_{\mathfrak{B},\Phi,\phi}(f\cdot\mathbb{I}_{B^\P})(x) \leq \langle \|f\|\rangle_{\Phi,\phi,B}.
		\end{align*}
		
		\item If \(\mu(A) > \mu(B)\), then \(B \subset A^\P\). Replacing \(B\) in Proposition \ref{propB} with \(A^\P\) gives
		\begin{align}\label{2.7}
			\|f\|_{\Phi,\phi,A} \lesssim \|f\|_{\Phi,\phi,A^\P} \leq \langle \|f\|\rangle_{\Phi,\phi,B}.
		\end{align}
		Using (\ref{2.5}) and (\ref{2.7}), we deduce
		\begin{align*}
			\bigl|\mathcal{M}_{\mathfrak{B},\Phi,\phi}f(x)-&\mathcal{M}_{\mathfrak{B},\Phi,\phi}(f\cdot\mathbb{I}_{B^\P})(x)\bigr| \leq \mathcal{M}_{\mathfrak{B},\Phi,\phi}f(x) \\
			&\leq \|f\|_{\Phi,\phi,A} + \langle \|f\| \rangle_{\Phi,\phi,B} \lesssim \langle \|f\|\rangle_{\Phi,\phi,B}.
		\end{align*}
	\end{itemize}
	Hence, for any \(x \in B\) we have shown
	\begin{align*}
		\left|\mathcal{M}_{\mathfrak{B},\Phi,\phi}f(x)-\mathcal{M}_{\mathfrak{B},\Phi,\phi}(f\cdot\mathbb{I}_{B^\P})(x)\right| \lesssim \langle \|f\|\rangle_{\Phi,\phi,B}.
	\end{align*}
	This estimate immediately yields
	\begin{align*}
		\bigl|\left(\mathcal{M}_{\mathfrak{B},\Phi,\phi}f-\mathcal{M}_{\mathfrak{B},\Phi,\phi}(f\cdot\mathbb{I}_{B^\P})\right)(x)-\left(\mathcal{M}_{\mathfrak{B},\Phi,\phi}f-\mathcal{M}_{\mathfrak{B},\Phi,\phi}(f\cdot\mathbb{I}_{B^\P})\right)(x')\bigr| \lesssim \langle \|f\|\rangle_{\Phi,\phi,B}.
	\end{align*}
	which verifies condition (\(\Psi\)-BOO-II). Consequently, \(\mathcal{M}_{\mathfrak{B},\Phi,\phi}\) is a \(\Psi\)-BOO with respect to \(\mathfrak{B}\).
\end{proof}

Let \((X,\,\mathfrak{M},\,\mu)\) be a measure space endowed with a ball-basis \(\mathfrak{B}\). Consider a Young function \(\Phi \in \mathscr{Y}\) and a function \(\phi \in \mathscr{G}\). For a ball \(B \in \mathfrak{B}\) and a constant \(K > 0\), we obtain
\begin{align*}
	\|K\|_{\Phi,\phi,B} = \inf\left\{\lambda > 0: \frac{1}{\mu(B) \phi( \mu(B))} \int_B \Phi \left(\frac{K}{\lambda}\right)d\mu \leq 1\right\} = \frac{K}{\Phi^{-1}(\phi( \mu(B)))}.
\end{align*}

If we now require \(\|K\|_{\Phi,\phi,B}\) to be bounded---that is, if there exist positive constants \(K_1, K_2 > 0\) and a parameter \(\lambda > 0\) such that
\begin{align}\label{2.8}
	K_1 \leq \Phi^{-1}(\lambda \phi( \mu(B))) \leq K_2,
\end{align}
then the following condition is introduced.

\begin{definition}\label{assumption2}
	For \(\Phi \in \mathscr{Y}\) and \(\phi \in \mathscr{G}\), we say that the pair \((\Phi,\phi)\) satisfies the \(E_{mb}\)-condition, written as \((\Phi,\phi) \in E_{mb}\), if inequality (\ref{2.8}) holds.
\end{definition}

According to Definition \ref{assumption1}, we obtain
\begin{align*}
	\mathcal{M}_{\mathfrak{B},\Phi,\phi}f(x) = \sup_{B\in\mathfrak{B}:\,x\in B}\|f\|_{\Phi,\phi,B} \leq \|f\|_{L^{(\Phi,\phi)}} \leq \mathscr{K} \cdot \|f\|_{L^{(\Psi,\psi)}}.
\end{align*}
Then, from Definition \ref{assumption2}, it follows that for any \(B \in \mathfrak{B}\),
\begin{align*}
	\left\|\mathcal{M}_{\mathfrak{B},\Phi,\phi}f(x)\right\|_{L^{(\Psi,\psi)}} = \sup_{B \in \mathfrak{B}} \left\|\mathcal{M}_{\mathfrak{B}, \Phi,\phi}f(x)\right\|_{\Psi,\psi,B}\leq \sup_{B \in \mathfrak{B}} \bigl\| \mathscr{K} \|f\|_{L^{(\Psi, \psi)}} \bigr\|_{\Psi,\psi,B} \lesssim \|f\|_{L^{(\Psi,\psi)}}.
\end{align*}
This implies that the operator \(\mathcal{M}_{\mathfrak{B},\Phi,\phi}: \, L^{(\Psi,\psi)}(X,\mu) \to L^{(\Psi,\psi)}(X,\mu)\) is bounded.

\subsection{\(\omega\)-Calder\'on-Zygmund operators}\label{Section 2.2}

In this section, we discuss the \(\omega\)-Calder\'on-Zygmund operator in the space \((\mathbb{R}^n,\mathcal{L}^n)\).
\begin{definition}
	Let \(\omega : [0, +\infty) \to [0, +\infty)\) be a modulus of continuity; that is, \(\omega\) is increasing, subadditive, and satisfies \(\omega(0)=0\). The modulus \(\omega\) is said to satisfy the {\sffamily Dini condition} if and only if the {\sffamily Dini norm}
	\begin{align*}
		\|\omega\|_{\rm Dini} := \int_0^1 \omega(t) \, \frac{dt}{t} < \infty.
	\end{align*}
	When this holds, we write \(\omega \in \mathrm{Dini}\).
\end{definition}

\begin{definition}
	Let \(\omega\) be a modulus of continuity. A function \(K(x,y)\) defined on \(\mathbb{R}^n \times \mathbb{R}^n \setminus \{(x,x) \in \mathbb{R}^{2n}\}\) is called an {\sffamily \(\omega\)-Calder\'on-Zygmund kernel} if there exists a constant \(C_K > 0\) for which the following estimates hold:
	\begin{itemize*}
		\item[{\rm (1)}] {\it Size condition:}
		\begin{align}\label{Kernel-1}
			\left|K(x, y)\right| \leq \frac{C_K}{\left|x-y\right|^n} \quad \text{for} \quad  x \ne y.
		\end{align}
		
		\item[{\rm (2)}] {\it Smoothness condition in the first variable:}
		\begin{align}\label{Kernel-2}
			\left|K(x, y) - K(x', y)\right| \leq \frac{C_K }{\left|x-y\right|^n}\,\omega\left(\frac{|x-x'|}{|x-y|} \right) \quad \text{for} \quad  |x-x'| \leq \frac{1}{2}|x-y|.
		\end{align}
		
		\item[{\rm (3)}] {\it Smoothness condition in the other variable:}
		\begin{align}\label{Kernel-3}
			\left|K(x, y) - K(x, y')\right| \leq \frac{C_K }{\left|x-y\right|^n}\,\omega\left(\frac{|y-y'|}{|x-y|} \right) \quad \text{for} \quad  |y-y'| \leq \frac{1}{2}|x-y|.
		\end{align}
	\end{itemize*}
	In the special case where \(\omega(t)=t^{\delta}\) for some \(\delta \in (0,1]\), the kernel \(K\) is referred to as a {\sffamily standard Calder\'on-Zygmund kernel}.
	
	A linear operator \(T: \mathcal{S} (\mathbb{R}^n) \to \mathcal{S}' (\mathbb{R}^n)\) is called an {\sffamily \(\omega\)-Calder\'on-Zygmund operator} provided that
	\begin{itemize*}
		\item[{\rm (i)\,}] \(T\) is bounded on \(L^2(\mathbb{R}^n)\).
		\item[{\rm (ii)}] there exists an \(\omega\)-Calder\'on-Zygmund kernel \(K\) such that for every \(f\in L^2_{\text{comp}}(\mathbb{R}^n)\),
		\begin{align}\label{CZ--Tf(x)}
			Tf(x)=\int_{\mathbb{R}^n} K(x,y)f(y)\,dy,\qquad x \notin {\rm supp}\,f.
		\end{align}
	\end{itemize*}
\end{definition}

\begin{remark}
	Two remarks are in order.
	\begin{itemize*}
		\item Because \(K(x,y)\) is continuous in \(y\) on \({\rm supp}\,f\) whenever \(x\notin{\rm supp}\,f\), the representation (\ref{CZ--Tf(x)}) --- initially required for \(f\in L^2_{comp}(\mathbb{R}^n)\) --- remains valid for all \(f\in L^1_{comp}(\mathbb{R}^n)\).
		
		\item It is known that the \(\omega\)-Calder\'on-Zygmund operator \(T\) is bounded from \(L^p(X,\,\rho,\,\mathfrak{M},\,\mu)\) to itself for \(1 < p < \infty\), and from \(L^1(X, \, \rho, \, \mathfrak{M}, \, \mu)\) to \(L^{1, \infty} (X, \, \rho, \, \mathfrak{M}, \, \mu)\). Here \((X,\,\rho,\,\mathfrak{M},\,\mu)\) is a space of homogeneous type endowed with the ball-basis \(\mathfrak{B}'_{\rho}\) (cf. \cite[Theorem 7.6]{Abstract2019}).
	\end{itemize*}
\end{remark}

\begin{theorem}
	Let \((\Phi,\Psi,\phi,\psi)\in\mathscr{Y} \uplus \mathscr{G}\). If \(T\) is an \(\omega\)-Calder\'on-Zygmund operator with \(\omega \in {\rm Dini}\), then \(T\) is a \(\Psi\)-BOO with respect to \(\mathfrak{B}\), with constants  
	\begin{align*}
		\mathscr{C}_1(T) \lesssim C_K \quad \text{and} \quad \mathscr{C}_2(T) \lesssim C_K\|\omega\|_{\rm Dini}.
	\end{align*}
\end{theorem}

\begin{proof}
	It is well known that Euclidean balls \(\mathfrak{B}_E\) form a ball-basis in \(\mathbb{R}^n\). Hence we choose a ball \(B = B(x_0, r_0) \in \mathfrak{B}_E\); then  
	\begin{align}\label{E ball-basis}
		B^\P = B(x_0, R_0) \in \mathfrak{B}_E,\quad\text{with} \; 2r_0 \leq R_0 \leq +\infty.
	\end{align}
	Now let \(T\) be an \(\omega\)-Calder\'on-Zygmund operator with \(\omega \in {\rm Dini}\). For a ball \(B\in\mathfrak{B}_E\), we denote by \(|B|\) its Lebesgue measure in \(\mathbb{R}^n\).
	
	First, we show that \(T\) satisfies condition (\(\Psi\)-BOO-I). Take an arbitrary ball \(A=B(x_0,r_0)\in \mathfrak{B}_E\) with \(A^\P\subsetneq \mathbb{R}^n\). From (\ref{E ball-basis}) we know that \(A^\P=B(x_0,R_0)\) with \(R_0 \geq 2r_0\). Set \(B=B(x_0,2R_0)\). Because \(A^\P\subsetneq\mathbb{R}^n\), we have \(R_0<\infty\) and  
	\begin{align}\label{2.13}
		A^\P = B(x_0,R_0) \subset B(x_0,2R_0) = B.
	\end{align} 
	Then, by the size condition (\ref{Kernel-1}), for any \(x \in A\) and \(y \in B^\P \setminus A^\P\),  
	\begin{align}\label{2.14}
		\left|K(x, y)\right| \leq \frac{C_K}{\left|x-y\right|^n} \lesssim \frac{C_K}{\omega_n\,r_0^n},
	\end{align} 
	where \(\omega_n\) denotes the volume of the unit ball \(\{ x \in \mathbb{R}^n : |x| \leq 1 \}\). Combining (\ref{2.13}) and (\ref{2.14}), we obtain for any \(x \in A\)  
	\begin{align*}
		\bigl| T(f\,\mathbb{I}_{B^\P})(x)&\,-\,T(f\,\mathbb{I}_{A^\P})(x) \bigr| \leq \int_{(B^\P) \setminus (A^\P)} |K(x, y)| |f(y)| \, dy \\
		&\lesssim \frac{C_K}{\omega_n\,r_0^n} |B^\P| \fint_{B^\P}|f(y)| = C_K \frac{R^n}{r_0^n} \|f\|_{\Phi(t)=t, \phi(t)=1, B^\P}, \quad\text{where} \; R\geq 8r_0.
	\end{align*}
	This implies that \(T\) fulfills condition (\(\Psi\)-BOO-I) with \(\mathscr{C}_1(T) \lesssim C_K\).
	
	Next, we verify that \(T\) satisfies condition (\(\Psi\)-BOO-II). For any ball \(B = B(x_0,r_0) \in \mathfrak{B}_E\) with \(B \subsetneq \mathbb{R}^n\) and for any \(x \in B\), we estimate as follows:  
	\begin{align*}
		\bigl| (T&f-T(f \cdot \mathbb{I}_{B^\P}))(x) - (Tf - T(f \cdot \mathbb{I}_{B^\P}))(x_0) \bigr| = \left|\, \int_{\mathbb{R}^n \setminus B^\P} \left( K(x, y) - K(x_0, y) \right) f(y)\, dy \,\right| \\
		&\leq \sum_{k=0}^{\infty} \int_{2^{k+1}B^\P \setminus 2^kB^\P} \left| K(x, y) - K(x_0, y) \right| \, |f(y)| \, dy \leq \sum_{k=0}^{\infty} \int_{2^{k+1}B^\P \setminus 2^kB^\P} \frac{C_K \omega\left(\frac{|x_0-x|}{|x_0-y|} \right) }{\left|x_0-y\right|^n} \, |f(y)| \, dy  \\
		&\leq \sum_{k=0}^{\infty} \frac{C_K \omega\left(2^{-k-1} \right) }{\left|2^kB^\P\right|} \int_{2^{k+1}B^\P \setminus 2^kB^\P} \, |f(y)| \, dy \leq 2^n\sum_{k=0}^{\infty} \frac{C_K \omega\left(2^{-k-1} \right) }{\left|2^{k+1} B^\P\right|} \int_{2^{k+1}B^\P} \, |f(y)| \, dy \\
		&\leq 2^n C_K \sum_{k=0}^{\infty} \omega\left(2^{-k-1} \right) \langle \|f\|\rangle_{\Phi(t)=t, \phi(t)=1, B} \lesssim C_K \|\omega\|_{\rm Dini} \langle \|f\|\rangle_{B},
	\end{align*}
	where we have used the following facts:  
	\begin{itemize*}
		\item For the ball \(B=B(x_0,r_0)\in\mathfrak{B}_E\) and any \(x \in B\), we have \(|x_0-x| < r_0\).
		
		\item If \(y \in 2^{k+1}B^\P \setminus 2^kB^\P\), then \(|x_0-y| \geq 2^kR_0 \geq 2^{k+1}r_0\) for \(k=0,1,2,...\).
		
		\item For the last estimate,
		\begin{align*}
			\sum_{k=0}^{\infty} \omega\left(2^{-k-1} \right) =  2 \sum_{k=0}^{\infty} \int_{2^{-k-1}}^{2^{-k}} \frac{\omega\left(2^{-k-1} \right)}{2^{-k}}dt \lesssim \sum_{k=0}^{\infty} \int_{2^{-k-1}}^{2^{-k}} \frac{\omega(t)}{t}dt = \int_{0}^{1} \frac{\omega(t)}{t}dt = \|\omega\|_{\rm Dini}.
		\end{align*}
	\end{itemize*}
	Consequently, condition (\(\Psi\)-BOO-II) holds with \(\mathscr{C}_2(T) \lesssim C_K\|\omega\|_{\rm Dini}\).
\end{proof}

According to Theorem \ref{main} and Theorem \ref{main2}, we can obtain the following conclusion.
\begin{theorem}
	Let \(\mathfrak{B}_E\) be a ball-basis in \(\mathbb{R}^n\) and let \(\lambda > 3\mathscr{C}_0^6\). Let \(T\) be an \(\omega\)-Calder\'on-Zygmund operator as described above and assume \(T \in \mathbb{W}_{\Psi, \psi, \lambda}\). Suppose that \((\Phi,\Psi,\phi,\psi) \in \mathscr{Y} \uplus \mathscr{G}\) and \((\Phi, \phi) \in E_{mb}\). Then the following statements hold:
	\begin{itemize*}
		\item[({\it a})] There exist two \(\frac{1}{2 \mathscr{C}_0^3}\)-sparse families \(\mathcal{S}_1, \, \mathcal{S}_2 \subset \mathfrak{B}_E\) such that for every \(f \in L^{(\Psi,\psi)}(\mathbb{R}^n) \subset L^{(\Phi,\phi)}(\mathbb{R}^n)\), every \(B \in \mathfrak{B}_E\), and \(a.e. \, x \in B\),
		\begin{align*}
			|\mathscr{T}f(x)| \lesssim \left(\mathscr{C}_1(T) + \mathscr{C}_2(T) + \|T\|_{L^{(\Psi, \psi)}(\mathbb{R}^n) \to wL^{(\Psi,\psi)}(\mathbb{R}^n)} \right) \cdot \left[ \mathcal{A}_{\mathcal{S}_1, \Phi,\phi} f(x) + \mathcal{A}_{\mathcal{S}_2, \Phi,\phi} f(x) \right].
		\end{align*}
		
		\item[({\it b})] Moreover, if \(\mathfrak{B}_E\) satisfies the Besicovitch \(\mathfrak{N}\)-condition, then
		\begin{align*}
			\|T\|_{L^{(\Psi, \psi)}(\mathbb{R}^n) \to L^{(\Psi,\psi)}(\mathbb{R}^n)} \lesssim \mathscr{K} \cdot \mathfrak{N} \cdot \left(\mathscr{C}_1(T) + \mathscr{C}_2(T) + \|T\|_{L^{(\Psi, \psi)}(\mathbb{R}^n) \to wL^{(\Psi,\psi)}(\mathbb{R}^n)} \right).
		\end{align*}
	\end{itemize*}
\end{theorem}

\subsection{Intrinsic square operators}\label{Section 2.3}

In this subsection, we fix a parameter \(0 < \alpha \leq 1\) and define \({\rm Lip}_{\alpha}\) as the collection of all Lipschitz functions \(\varphi : \mathbb{R}^{n} \to \mathbb{R}\) of order \(\alpha\) with homogeneous norm \(1\), supported in the closed unit ball \(\{ x \in \mathbb{R}^n : |x| \leq 1 \}\), and satisfying \(\int_{\mathbb{R}^{n}} \varphi(x) \, dx = 0\). For a locally integrable function \(f \in L^{1}_{{\rm loc}}(\mathbb{R}^{n})\) and a point \((y, t) \in \mathbb{R}^n \times \mathbb{R}_+ := \mathbb{R}^{n+1}_+\), we define  
\begin{align*}
	A_{\alpha}f(t, y) \;:=\; \sup_{\varphi \in {\rm Lip}_{\alpha}} \left| (\varphi_t * f)(y) \right|,
\end{align*}
where \(\varphi_t(x) = t^{-n} \varphi(x / t)\), and we call \(\varphi_t\) the kernel of \(A_{\alpha}\).

Let \(\Gamma_\beta(x) := \{(y, t) \in \mathbb{R}^{n+1}_+ : |x-y| < \beta t\}\) for \(\beta \in (0, \infty)\). We introduce the following three {\sffamily intrinsic square operators}:
\begin{align}
	\label{G 1} g_{\alpha}f(x) & := \left(\int_{0}^{\infty} \left( A_\alpha f(t, x) \right)^2\frac{dt}{t} \right)^{\frac{1}{2}},\\
	\label{G 2} g_{\alpha, \beta}f(x) & := \left( \iint_{\Gamma_\beta(x)} \left( A_\alpha f(t, z) \right)^2 \frac{dzdt}{t^{n+1}} \right)^{\frac{1}{2}},\\
	\label{G 3} g^*_{\lambda,\alpha}f(x) & := \left(\iint_{\mathbb{R}^{n+1}_+}\left(\frac{t}{t+|x-z|}\right)^{n\lambda}\left( A_\alpha f(t, z) \right)^2\frac{dzdt}{t^{n+1}} \right)^{\frac{1}{2}}, \quad \lambda > 3+\frac{2\alpha}{n}.
\end{align}

To simplify notation, we introduce the following three Banach spaces:
\begin{align}
	\label{U 1} \mathbb{B}_{g_{\alpha}} & := \left\{ F_1 : \mathbb{R}_+ \to \mathbb{R} \;\big|\; \left\|f_1\right\|_{\mathbb{B}_{g_{\alpha}}} := \left( \int_0^\infty |f_1(t)|^2 \frac{dt}{t} \right)^{\frac{1}{2}} < \infty \right\},\\
	\label{U 2} \mathbb{B}_{g_{\alpha, \beta}} & := \left\{ F_2 : \mathbb{R}^{n+1}_+ \to \mathbb{R} \;\big|\; \left\|f_2\right\|_{\mathbb{B}_{g_{\alpha, \beta}}} := \left( \int_{\mathbb{R}^{n+1}_+} |f_2(z,t)|^2 \frac{dz dt}{t^{n+1}} \right)^{\frac12} < \infty \right\},\\
	\label{U 3} \mathbb{B}_{g^*_{\lambda,\alpha}} & := \mathbb{B}_{g_{\alpha, \beta}}.
\end{align}

We proceed as follows for each of the three intrinsic square operators mentioned above.

(1) For the operator \(g_{\alpha}\): Let
\begin{align}\label{Iso-Kernel 1}
	\mathscr{P}_{g_{\alpha}}(x) := \{ \varphi_t(x) \}_{t>0},
\end{align}
where \(\varphi \in {\rm Lip}_{\alpha}\). Further, denote
\begin{align}\label{Iso-Tf 1}
	\mathscr{A}_{g_{\alpha}}f(x):=\{A_\alpha f(t, x)\}_{t>0}.
\end{align}
Here \(\mathscr{P}_{g_{\alpha}}\) is the kernel of \(\mathscr{A}_{g_{\alpha}}\). We then have \(g_{\alpha}f(x)=\left\|\mathscr{A}_{g_{\alpha}}f(x)\right\|_{\mathbb{B}_{g_{\alpha}}}\).

\vspace{6pt}

(2) For the operator \(g_{\alpha, \beta}\): Take \(P_{z,t}(x)=\mathbb{I}_{ \Gamma_\beta(0)}(z)\varphi_t(x-z)\), where \(\varphi \in {\rm Lip}_{\alpha}\) and \((z,t)\in\mathbb{R}^{n+1}_+\). Define
\begin{align}\label{Iso-Kernel 2}
	\mathscr{P}_{g_{\alpha, \beta}}(x):=\{P_{z,t}(x)\}_{(z,t)\in\mathbb{R}^{n+1}_+}
\end{align}
and 
\begin{align}\label{Iso-Tf 2}
	\mathscr{A}_{g_{\alpha, \beta}}f(x):=\{A_\alpha f(P_{z,t}, x)\}_{(z,t) \in \mathbb{R}^{n+1}_+}=\sup \left| (P_{z,t} * f)(x) \right|
\end{align}
with \(\mathscr{P}_{g_{\alpha, \beta}}\) being the kernel of \(\mathscr{A}_{g_{\alpha, \beta}}\). Then \(g_{\alpha, \beta}f(x)=\|\mathscr{A}_{g_{\alpha, \beta}}f(x)\|_{\mathbb{B}_{g_{\alpha, \beta}}}\).

\vspace{6pt}

(3) For the operator \(g^*_{\lambda,\alpha}\):  
Following a similar procedure, define
\[
Q_{z,t}(x)=\left(\frac{t}{t+|z|}\right)^{\frac{n \lambda}{2}} \varphi_t(x-z),
\quad \varphi \in {\rm Lip}_{\alpha},\,\, (z,t)\in \mathbb{R}^{n+1}_+.
\]
Set
\begin{align}\label{Iso-Kernel 3}
	\mathscr{P}_{g^*_{\lambda,\alpha}}(x) := \{Q_{z,t}(x)\}_{(z,t) \in \mathbb{R}^{n+1}_+}
\end{align}
and 
\begin{align}\label{Iso-Tf 3}
	\mathscr{A}_{g^*_{\lambda,\alpha}}f(x):=\{A_\alpha f(Q_{z,t}, x)\}_{(z,t) \in \mathbb{R}^{n+1}_+}=\sup \left| (Q_{z,t} * f)(x) \right|.
\end{align}
Consequently, \(g^*_{\lambda,\alpha}\) can be expressed as \(g^*_{\lambda,\alpha}f(x)=\|\mathscr{A}_{g^*_{\lambda,\alpha}}f(x)\|_{\mathbb{B}_{g^*_{\lambda,\alpha}}}\).

A canonical ball-basis in \((\mathbb{R}^n,\mathcal{L}^n)\) is furnished by the family \(\mathfrak{B}_Q\) of all cubes aligned with the coordinate axes, as demonstrated in remark \ref{example ball-basis}. For any \(Q \in \mathfrak{B}_Q\), the associated ball-hull is precisely the concentric fivefold dilation: \(Q^\P = 5Q\).

We begin by noting H\"older's inequality in Orlicz-Morrey spaces.
\begin{lemma}\cite[Lemma 9.2]{OM-Nakai2008}
	Given a Young function \(\Phi\), a function \(\phi \in \mathcal{G}\), and a set \(Q \in\mathfrak{B}_Q\), the following integral inequality holds:
	\begin{align*}
		\int_{Q} f(x)g(x) \, dx \leq 2|Q|\phi(|Q|)\left\| f\right\|_{\Phi,\phi,Q}\left\|g\right\|_{\widetilde{\Phi},\phi,Q},
	\end{align*}
	where \(\widetilde{\Phi}\) denotes the complementary Young function of \(\Phi\).
\end{lemma}

\begin{definition}
	Let \(\Phi \in \mathscr{Y}\) and \(\phi \in \mathscr{G}\), and let \(\widetilde{\Phi}\) be the complementary function of \(\Phi\). We say that a kernel \(\mathscr{P}\) satisfies the \(\mathbb{B}\)-valued \(\Phi\)-\(\phi\)-H\"ormander condition if
	\begin{align}\label{Hormander 1}
		\sup_{Q \in\mathfrak{B}_Q, \, x \in \frac{1}{2} Q}
		|Q|\phi(|Q|) \Bigl\|\left\| \mathscr{P}(x) \right\|_{\mathbb{B}}\Bigr\|_{\widetilde{\Phi},\phi,2Q \setminus Q} < \infty,
	\end{align}
	and 
	\begin{align}\label{Hormander 2}
		\sup_{\substack{Q\in\mathfrak{B}_Q \\ x,x' \in \frac{1}{2} Q}} \sum _{j=1}^{\infty} |2^jQ|\phi(|2^jQ|)\Bigl\| \left\| \mathscr{P}(x) - \mathscr{P}(x') \right\|_{\mathbb{B}} \Bigr\|_{ \widetilde{\Phi}, \phi, \widetilde{B}_j(Q)} < \infty.
	\end{align}
	where \(\widetilde{B}_j(Q) := 2^j Q \setminus 2^{j-1} Q\) for \(j= 1,2,\dots\). Here, \(\mathscr{P}(x)\) takes values as given in (\ref{Iso-Kernel 1}), (\ref{Iso-Kernel 2}), and (\ref{Iso-Kernel 3}), and \(\mathbb{B}\) takes values as in (\ref{U 1}), (\ref{U 2}), and (\ref{U 3}).
\end{definition}

It is worth noting that (\ref{Hormander 1}) and (\ref{Hormander 2}) are generalizations of the case \(m=1\) in \cite[Definition 2.12]{Caomingming2023}. The establishment of these two conditions will help optimize the proof process that intrinsic square operators are \(\Psi\)-BOO.

\begin{theorem}
	Suppose that \((\Phi,\Psi,\phi,\psi)\in\mathscr{Y} \uplus \mathscr{G}\). Let \(\mathscr{P}(x)\) be the intrinsic square operators satisfying the \(\mathbb{B}\)-valued \(\Phi\)-\(\phi\)-H\"ormander condition, where \(\mathfrak{G}(x)\) is taken from (\ref{G 1}), (\ref{G 2}) and (\ref{G 3}), and \(\mathbb{B}\) is taken from (\ref{U 1}), (\ref{U 2}) and (\ref{U 3}). Then \(\mathfrak{G}(x)\) is a \(\mathbb{B}\)-valued \(\Psi\)-BOO with respect to \(\mathfrak{B}_Q\).
\end{theorem}

\begin{proof}
	Let \(Q_0 \in \mathfrak{B}_Q\) with \(Q_0 \subsetneq \mathbb{R}^n\). Take a point \(x \in Q_0\) and denote \(Q = 2Q_0\). Consider \(\mathscr{A}f(x)\) as defined in (\ref{Iso-Tf 1}), (\ref{Iso-Tf 2}), and (\ref{Iso-Tf 3}). Using (\ref{Hormander 1}) and H\"older's inequality, we obtain
	\begin{align*}
		\bigl\| \mathscr{A}(f\,&\mathbb{I}_{Q^\P})(x) - \mathscr{A} (f \, \mathbb{I}_{Q_0^\P})(x) \bigr\|_{\mathbb{B}}\\
		&\leq \left\| \int_{Q^\P \setminus Q_0^\P} \mathscr{P}(x-y) f(y)dy\, \right\|_{\mathbb{B}} \leq  \int_{Q^\P \setminus Q_0^\P} \left\|\mathscr{P}(x-y)\right\|_{\mathbb{B}} |f(y)|\, dy\\
		&\leq 2|Q|\phi(|Q|)\left\| \left\|\mathscr{P}(x-y)\right\|_{\mathbb{B}}\right\|_{\widetilde{\Phi},\phi,Q^\P \setminus Q_0^\P}\left\|f\right\|_{\Phi,\phi,Q^\P \setminus Q_0^\P} \lesssim \left\|f\right\|_{\Phi,\phi,Q^\P}.
	\end{align*}
	The last inequality follows from Proposition \ref{propB}. This yields the condition (\(\Psi\)-BOO-I).
	
	Next, fix an arbitrary cube \(Q \in \mathfrak{B}_Q\) and two points \(x, x' \in Q \). Using (\ref{Hormander 2}), the properties of the class \(\mathscr{G}\), and H\"older's inequality, we have
	\begin{align*}
		\bigl\| &\left(\mathscr{A}f-\mathscr{A}(f\,\mathbb{I}_{Q^\P})\right)(x) - \left(\mathscr{A}f-\mathscr{A}(f\,\mathbb{I}_{Q^\P})\right)(x') \bigr\|_{\mathbb{B}}\\
		&\leq \left\| \int_{\mathbb{R}^n \setminus 5Q} \left(\mathscr{P}(x-y)-\mathscr{P}(x'-y)\right) f(y)\,dy\, \right\|_{\mathbb{B}}  \\
		&\leq \sum_{j=1}^{\infty} \int_{\widetilde{B}_j(Q)} \left\| \mathscr{P}(x-y) - \mathscr{P} (x'-y)\right\|_{\mathbb{B}} |f(y)|\, dy\\
		&\leq 2 \sum_{j=1}^{\infty} |\widetilde{B}_j(Q)| \phi(|\widetilde{B}_j (Q)|) \left\| \left\| \mathscr{P}(x-y) - \mathscr{P} (x'-y) \right\|_{\mathbb{B}} \right\|_{ \widetilde{\Phi}, \phi, \widetilde{B}_j(Q)}\left\|f\right\|_{\Phi,\phi,\widetilde{B}_j(Q)}  \\
		&\lesssim \sum_{j=1}^{\infty} |2^jQ|\phi(|2^jQ|)\left\| \left\| \mathscr{P}(x-y) - \mathscr{P} (x'-y)\right\|_{\mathbb{B}} \right\|_{ \widetilde{\Phi},\phi,\widetilde{B}_j(Q)}\left\|f\right\|_{\Phi,\phi,2^jQ}  \\
		&\lesssim \sup_{j \geq 0}\left\|f\right\|_{\Phi,\phi,2^jQ} \leq \langle \|f\|\rangle_{\Phi,\phi,Q}.
	\end{align*}
	This implies that the condition (\(\Psi\)-BOO-II) holds.
\end{proof}

As a consequence of Theorem \ref{main} and Theorem \ref{main2}, we have the following theorem.
\begin{theorem}
	Let \(\mathfrak{B}_E\) be a ball-basis in \(\mathbb{R}^n\) and let \(\lambda > 3\mathscr{C}_0^6\). Suppose that \((\Phi,\Psi,\phi,\psi) \in \mathscr{Y} \uplus \mathscr{G}\) and \((\Phi, \phi) \in E_{mb}\). Denote by \(\mathfrak{G}(x)\) the object defined in (\ref{G 1}), (\ref{G 2}) and (\ref{G 3}), and by \(\mathbb{B}\) the one defined in (\ref{U 1}), (\ref{U 2}) and (\ref{U 3}). We then have the following:
	
	\begin{itemize*}
		\item[({\it a})] There exist two \(\frac{1}{2 \mathscr{C}_0^3}\)-sparse families \(\mathcal{S}_1, \, \mathcal{S}_2 \subset \mathfrak{B}_E\) such that for every \(f \in L^{(\Psi,\psi)}(\mathbb{R}^n)\subset L^{(\Phi,\phi)}(\mathbb{R}^n)\), every \(B \in \mathfrak{B}_E\), and \(a.e. \, x \in B\),
		\begin{align*}
			\|\mathscr{A}f(x)\|_\mathbb{B} \lesssim \left(\mathscr{C}_1(\mathfrak{G}) + \mathscr{C}_2(\mathfrak{G}) + \|\mathfrak{G}\|_{L^{(\Psi, \psi)}(\mathbb{R}^n) \to wL^{(\Psi,\psi)}(\mathbb{R}^n)} \right) \cdot \left[ \mathcal{A}_{\mathcal{S}_1, \Phi,\phi} f(x) + \mathcal{A}_{\mathcal{S}_2, \Phi,\phi} f(x) \right],
		\end{align*}
		where \(\mathscr{A}f(x)\) is taken from (\ref{Iso-Tf 1}), (\ref{Iso-Tf 2}) and (\ref{Iso-Tf 3}).
		
		\item[({\it b})] If, in addition, \(\mathfrak{B}_E\) satisfies the Besicovitch \(\mathfrak{N}\)-condition, then
		\begin{align*}
			\|\mathfrak{G}\|_{L^{(\Psi, \psi)}(\mathbb{R}^n) \to L^{(\Psi,\psi)}(\mathbb{R}^n)} \lesssim \mathscr{K} \cdot \mathfrak{N} \cdot \left(\mathscr{C}_1(\mathfrak{G}) + \mathscr{C}_2(\mathfrak{G}) + \|\mathfrak{G}\|_{L^{(\Psi, \psi)}(\mathbb{R}^n) \to wL^{(\Psi,\psi)}(\mathbb{R}^n)} \right).
		\end{align*}
	\end{itemize*}
\end{theorem}

\subsection{Carleson-type operators}\label{Section 2.4}

Let \((X,\,\mathfrak{M},\,\mu)\) be a measure space endowed with a ball-basis \(\mathfrak{B}\). Consider a family of \(\Psi\)-BOOs \(\{T_{\alpha}\}_{\alpha \in \mathfrak{A}}\) on \((X,\, \mathfrak{M},\, \mu)\), where \(\mathfrak{A}\) is an index set. We can define a Carleson-type operator of the form
\[
T^{\mathfrak{A}}f(x):=\sup_{\alpha \in \mathfrak{A}} \left| T_{\alpha}f(x) \right|.
\]
Here, we define a Banach space as follows:  
\[
\mathbb{B}:=\left\{ F: \mathfrak{A} \to \mathbb{R} \;\big|\; \left\|f\right\|_{\mathbb{B}} := \sup_{\alpha \in \mathfrak{A}} \left|f(\alpha)\right| < \infty \right\}.
\]
Then, according to Definition \ref{BOO}, the following theorem holds trivially.

\begin{theorem}
	Let \((\Phi,\Psi,\phi,\psi)\in\mathscr{Y} \uplus \mathscr{G}\) and let \((X,\,\mathfrak{M},\,\mu)\) be a measure space endowed with a ball-basis \(\mathfrak{B}\). Suppose that \(\{T_{\alpha}\}_{\alpha \in \mathfrak{A}}\) is a family of \(\Psi\)-BOOs satisfying  
	\begin{align}\label{Carleson-type-C}
		\mathscr{C}_1^{\mathfrak{A}} := \sup_{\alpha \in \mathfrak{A}} \mathscr{C}_1 (T_{\alpha}) < \infty \quad \text{and} \quad \mathscr{C}_2^{\mathfrak{A}} := \sup_{\alpha \in \mathfrak{A}}\mathscr{C}_2(T_{\alpha}) < \infty.
	\end{align}
	Then \(T^{\mathfrak{A}}\) is a \(\mathbb{B}\)-valued \(\Psi\)-BOO with respect to \(\mathfrak{B}\), and its constants satisfy  
	\[
	\mathscr{C}_1(T^{\mathfrak{A}}) \le \mathscr{C}_1^{\mathfrak{A}} \quad \text{and} \quad \mathscr{C}_2(T^{\mathfrak{A}}) \le \mathscr{C}_2^{\mathfrak{A}} .
	\]
\end{theorem}

Combining Theorem \ref{main} and Theorem \ref{main2}, we arrive at the following result.
\begin{theorem}
	Let \((X,\,\mathfrak{M},\,\mu)\) be a measure space endowed with a ball-basis \(\mathfrak{B}\) and let \(\lambda > 3\mathscr{C}_0^6\). Suppose that \((\Phi,\Psi,\phi,\psi)\in\mathscr{Y} \uplus \mathscr{G}\) and \((\Phi, \phi) \in E_{mb}\). Let \(\{T_{\alpha}\}_{\alpha \in \mathfrak{A}}\) be a family of \(\Psi\)-BOOs satisfying (\ref{Carleson-type-C}), and \(T^{\mathfrak{A}} \in \mathbb{W}_{\Psi, \psi, \lambda}\). The following assertions are valid:
	\begin{itemize*}
		\item[({\it a})] There exist two \(\frac{1}{2 \mathscr{C}_0^3}\)-sparse families \(\mathcal{S}_1, \, \mathcal{S}_2 \subset \mathfrak{B}\) such that for every \(f \in L^{(\Psi,\psi)}(X)\subset L^{(\Phi,\phi)}(X)\), every \(B \in \mathfrak{B}\), and \(a.e. \, x \in B\),
		\begin{align*}
			|T^{\mathfrak{A}}f(x)| \lesssim \left(\mathscr{C}_1(T^{\mathfrak{A}}) + \mathscr{C}_2(T^{\mathfrak{A}}) + \|T^{\mathfrak{A}}\|_{L^{(\Psi, \psi)}(\mathbb{R}^n) \to wL^{(\Psi,\psi)}(\mathbb{R}^n)} \right) \cdot \left[ \mathcal{A}_{\mathcal{S}_1, \Phi,\phi} f(x) + \mathcal{A}_{\mathcal{S}_2, \Phi,\phi} f(x) \right].
		\end{align*}
		
		\item[({\it b})] If, furthermore, \(\mathfrak{B}\) satisfies the Besicovitch \(\mathfrak{N}\)-condition, then
		
		\begin{align*}
			\|T^{\mathfrak{A}}\|_{L^{(\Psi, \psi)}(X) \to L^{(\Psi,\psi)}(X)} \lesssim \mathscr{K} \cdot \mathfrak{N} \cdot  \left(\mathscr{C}_1(T^{\mathfrak{A}}) + \mathscr{C}_2(T^{\mathfrak{A}}) + \|T^{\mathfrak{A}}\|_{L^{(\Psi, \psi)}(\mathbb{R}^n) \to wL^{(\Psi,\psi)}(\mathbb{R}^n)} \right).
		\end{align*}
	\end{itemize*}
\end{theorem}

\section{Sparse domination of \(\Psi\)-BOOs}\label{Section 3}

\subsection{Geometric properties of ball-basis}\label{Section 3.1}

We use the notation \(a \lesssim b\) to mean \(a \leq c\cdot b\) for some constant \(c>0\) independent of all relevant parameters, and we write \(a \sim b\) when both \(a \lesssim b\) and \(b \lesssim a\) hold.

Throughout this work, \((X,\,\mathfrak{M},\,\mu)\) be a measure space endowed with a ball-basis \(\mathfrak{B}\). We shall frequently use several simple consequences of the axioms defining \(\mathfrak{B}\):

\begin{itemize*}
	\item[({\it a})] (\textbf{Ball-hull control property}) If \(A, B \in \mathfrak{B}\) satisfy \(\mu(A) \leq 2\mu(B)\) and \(A \cap B \neq \varnothing\), then \(A \subset B^\P\).
	
	\item[({\it b})] (\textbf{Ball-hull relation}) We refer to \(B^\P\) in condition B4 as the hull of the ball \(B\). For a ball \(B \in \mathfrak{B}\), set \(B^{[0]} := B\), \(B^{[1]} := B^\P\), and inductively \(B^{[n+1]} := (B^{[n]})^\P\) for \(n \geq 1\). From condition B4 it follows that \(\mu(B^{[n+1]}) \leq \mathscr{C}_0\,\mu(B^{[n]})\) and consequently \(\mu(B^{[n]}) \leq \mathscr{C}_0^n\,\mu(B)\) for all \(n \geq 0\).
	
	\item[({\it c})] (\textbf{Bounded sets}) A set \(E \subset X\) is termed {\sffamily bounded} when there exists some ball \(B \in \mathfrak{B}\) containing \(E\).
	
	\item[({\it d})] (\textbf{Almost-sure inclusion}) For measurable sets \(E, F \subset X\), we write \(E \subset F\) a.s. (almost surely) when \(\mu(E \setminus F) = 0\).
	
	\item[({\it e})] (\textbf{Density point}) A point \(x \in E\) is a {\sffamily density point} of a measurable set \(E \in \mathfrak{M}\) if for any \(\varepsilon > 0\) there exists a ball \(B \ni x\) such that \(\mu(B \cap E) > (1 - \varepsilon) \mu(B)\).
	
	\item[({\it f})] (\textbf{Density property}) The {\sffamily density property} for a measure space \((X,\,\mathfrak{M},\,\mu)\) is characterized by the condition that, given any measurable set \(E\), almost every point of \(E\) is a density point.
	
	\item[({\it g})] (\textbf{Doubling condition}) A ball-basis \(\mathfrak{B}\) satisfies the {\sffamily doubling condition} if and only if there exists a constant \(\theta > 1\) such that, given any ball \(A \in \mathfrak{B}\) with \(A^\P \subsetneq X\), we can select a ball \(B \in \mathfrak{B}\) for which both inclusions \(A \subset B\) and the measure inequality \(\mu(B) \leq \theta \mu(A)\) hold.

\end{itemize*}

The geometric properties of the ball-basis discussed below can be found in \cite{Abstract2019,Caomingming2023,Abstract2021,Abstract2023}.

\begin{lemma}\label{Lm 3.1}
	Let \((X,\,\mathfrak{M},\,\mu)\) be a measure space endowed with a ball-basis \(\mathfrak{B}\). Assume that there exist a ball \(B \in \mathfrak{B}\) and a sequence of balls \(\{G_k\}_{k \geq 1} \subset \mathfrak{B}\) such that \(G_k \cap B \neq \varnothing\) for each \(k\) and \(\lim_{k \to \infty} \mu(G_k) = r := \sup_{A \in \mathfrak{B}} \mu(A)\). Then \(X \subset \bigcup_k G_k^\P\). Moreover, for any ball \(A \in \mathfrak{B}\), there exists an integer \(k_0\) such that \(A \subset G_k\) whenever \(k \ge k_0\).
\end{lemma}

\begin{lemma}\label{Lm more}
	Let \((X,\,\mathfrak{M},\,\mu)\) be a measure space endowed with a ball-basis \(\mathfrak{B}\) that satisfies ball-hull control property. The following statements then hold:
	\begin{itemize*}
		\item[(1)] If a set \(E \subset X\) is bounded and covered by a family of balls \(\mathcal{G} \subset \mathfrak{B}\) (that is, \(E \subset \bigcup_{G \in \mathcal{G}} G\)), then one can extract a finite or infinite pairwise disjoint subfamily \(\{G_k\}_{k \in I} \subset \mathcal{G}\), where the index set \(I\) is either finite or \(I = \mathbb{N}_+\), such that \(E \subset \bigcup_{k \in I} G_k^\P\).
		
		\item[(2)] Under the additional assumption that \(\mathfrak{B}\) satisfies the density property, the following holds: for every bounded measurable set \(E\) with \(\mu(E)>0\) and every \(\varepsilon > 0\), there exists a sequence \(\{B_k\} \subset \mathfrak{B}\) satisfying
		\begin{align*}
			\mu \left( \bigcup_k B_k \setminus E \right) < \varepsilon \quad \text{and} \quad \mu \left( E \setminus \bigcup_k B_k \right) < \alpha \mu(E),
		\end{align*}
		with some admissible constant \(\alpha \in (0,1)\).
		
		\item[(3)] If, furthermore, \(\mathfrak{B}\) satisfies the density property, then for every bounded measurable set \(E \subset X\) one can find a sequence of balls \(\{B_k\}_{k \in \mathbb{N}_+} \subset \mathfrak{B}\) such that
		\begin{align*}
			E \subset \bigcup_k B_k \; a.s. \quad \text{and} \quad \sum_k \mu(B_k) \leq 2\mathscr{C}_0\,\mu(E).
		\end{align*}
		
		\item[(4)] Let \(A \in \mathfrak{B}\) and let \(\mathcal{G} \subset \mathfrak{B}\) be a pairwise disjoint family of balls such that for every \(G \in \mathcal{G}\), and there exist constants \(c_1,\, c_2 > 0\) satisfying
		\begin{align*}
			G^\P \cap A \neq \varnothing \quad \text{and} \quad 0 < c_1 \leq \mu(G) \leq c_2 <\infty.
		\end{align*}
		Then \(\mathcal{G}\) is finite, and its cardinality \(\#\mathcal{G}\) satisfies
		\begin{align*}
			\#\mathcal{G} \lesssim c_1^{-1} \max\{c_2,\; \mu(A)\}.
		\end{align*}
		(Here \(\#\mathcal{G}\) denotes the number of elements in the family \(\mathcal{G}\).)
		
		\item[(5)] The condition B3 holds if and only if the density property is satisfied.
		
	\end{itemize*}
\end{lemma}

\subsection{Properties of \(\Psi\)-BOOs}\label{Section 3.2}

For a \(\mathbb{B}\)-valued operator \(T\) satisfying \(Tf(x) = \|\mathscr{T}f(x)\|_{\mathbb{B}}\), we consider condition (\(\Psi\)-BOO-I). For any \(A, B \in \mathfrak{B}\) with \(A \subset B\) we introduce  
\begin{align*}
	\Delta_T(A,B) := \sup_{f \in L^{(\Psi,\psi)}(X)} \sup_{x\in A} \frac{ \left\| \mathscr{T}(f\,\mathbb{I}_{B^\P})(x) - \mathscr{T}(f \, \mathbb{I}_{A^\P})(x) \right\|_{\mathbb{B}}} {\|f\|_{\Phi,\phi,B^\P}}.
\end{align*}

\begin{lemma}\label{Lm 3.3}
	Let \((X,\,\mathfrak{M},\,\mu)\) be a measure space endowed with a ball-basis \(\mathfrak{B}\). Given a Young function \(\Phi \in \mathscr{Y}\) and a function \(\phi \in \mathscr{G}\), the following hold:
	
	(1) If \(A, B, C \in \mathfrak{B}\) satisfy \(A \subset B \subset C\), then \(\Delta_T(A,B) \lesssim \Delta_T(A, C)\).
	
	(2) For any \(A, B \in \mathfrak{B}\) with \(A \subset B\), we have
	\begin{align*}
		\langle \|f \, \mathbb{I}_{B^\P}\|\rangle_{\Phi,\phi,A} \lesssim \frac{\mu(B) \phi(\mu(B))}{\mu(A)\phi(\mu(A))} \|f\|_{\Phi,\phi, B^\P}.
	\end{align*}
\end{lemma}

\begin{proof}
	Now we prove part (1). Let \(g = f \, \mathbb{I}_{B^\P}\). For any \(x \in A\), we have  
	\begin{align*}
		\|\mathscr{T}(f\,\mathbb{I}_{B^\P})&(x) - \mathscr{T}(f \, \mathbb{I}_{A^\P})(x)\|_{\mathbb{B}} = \|\mathscr{T} (g\,\mathbb{I}_{C^\P})(x) - \mathscr{T}(g \, \mathbb{I}_{A^\P})(x) \|_{\mathbb{B}} \\
		&\leq \Delta_T(A,C) \|g\|_{\Phi,\phi,C^\P} = \Delta_T(A,C) \|f \, \mathbb{I}_{B^\P}\|_{\Phi,\phi,C^\P} \\
		& \leq \Delta_T(A,C) \max\{1,C_{ic}\} \|f\|_{\Phi,\phi,B^\P} \lesssim \Delta_T(A,C) \|f\|_{\Phi,\phi,B^\P}.
	\end{align*}
	The penultimate inequality follows from Proposition \ref{propC}. This yields part (1).
	
	Next we present the proof of part (2). Take a ball \(A_1 \in \mathfrak{B}\) such that \(A \subset A_1 \cap B\).
	\begin{itemize}[leftmargin=1em]
		\item {\bf {\emph Case} 1} \(\mu(A_1) \leq \mu(B^\P)\).  
		Since \(A_1 \cap B^\P \ne \varnothing\), ball-hull control property implies \(A_1 \subset B^{[2]}\). Consequently,
		\begin{align}\label{3.1}
			\|f \, \mathbb{I}_{B^\P}&\|_{\Phi,\phi,A_1} \leq \frac{\max\{1, C_{ic}\} \mu(B^{[2]})\phi(\mu(B^{[2]}))}{\mu(A_1)\phi( \mu(A_1))} \|f \, \mathbb{I}_{B^\P}\|_{\Phi,\phi,B^{[2]}} \nonumber \\
			&\lesssim \frac{\mu(B^{[2]})\phi(\mu(B^{[2]}))}{\mu(A_1)\phi( \mu(A_1))} \|f\|_{\Phi,\phi,B^\P} \lesssim \frac{\mu(B) \phi(\mu(B))}{\mu(A)\phi(\mu(A))} \|f\|_{\Phi,\phi,B^\P}.
		\end{align}
		
		\item {\bf {\emph Case} 2} \(\mu(A_1) > \mu(B^\P)\).  
		Ball-hull control property together with \(A_1 \cap B^\P \ne \varnothing\) gives \(B^\P \subset A_1^\P\). Hence,
		\begin{align}\label{3.2}
			\|f \, \mathbb{I}_{B^\P}&\|_{\Phi,\phi,A_1} \leq \frac{\max\{1, C_{ic}\} \mu(A_1^\P)\phi(\mu(A_1^\P))}{\mu(A_1)\phi( \mu(A_1))} \|f \, \mathbb{I}_{B^\P}\|_{\Phi,\phi,A_1^\P} \nonumber \\
			&\lesssim \frac{\mu(A_1^\P)\phi(\mu(A_1^\P))}{\mu(A_1)\phi( \mu(A_1))} \|f\|_{\Phi,\phi,B^\P} \lesssim \|f\|_{\Phi, \phi, B^\P}.
		\end{align}
	\end{itemize}
	The inequalities in (\ref{3.1}) and (\ref{3.2}) follow from Proposition \ref{propB}, Proposition \ref{propC}, and the fact that \(\phi \in \mathscr{G}\). Therefore, (\ref{3.1}) and (\ref{3.2}) together imply part (2).
\end{proof}

\begin{lemma}\label{Lm 3.4}
	Let \((X,\,\mathfrak{M},\,\mu)\) be a measure space endowed with a ball-basis \(\mathfrak{B}\), and let \((\Psi,\psi) \in E_{mb}\). Assume that \(T\) is a \(\mathbb{B}\)-valued linear operator satisfying the condition (\(\Psi\)-BOO-II) and that \(T \in \mathbb{W}_{\Psi, \psi}\). Then for every \(f \in L^{(\Psi, \psi)}(X)\) the following estimates hold:
	
	\begin{itemize*}
		\item[(a)] For any balls \(A,B \in \mathfrak{B}\) with \(A \subset B\),
		\begin{align*}
			\Delta_T(A,B)\lesssim\left(\mathscr{C}_2(T)+\|T\|_{L^{(\Psi, \psi)}\to wL^{(\Psi,\psi)}}\right)\frac{\mu(B)\phi(\mu(B)) }{\mu(A)\phi(\mu(A))}.
		\end{align*}
		
		\item[(b)] For any balls \(A, B, C \in \mathfrak{B}\) with \(A \subset B \subset C\),
		\begin{align*}
			\Delta_T(A,C)\lesssim\left(\mathscr{C}_2(T)+\|T\|_{L^{(\Psi, \psi)}\to wL^{(\Psi,\psi)}}+\Delta_T(A,B)\right) \frac{\mu(C)\phi(\mu(C)) }{\mu(B)\phi(\mu(B))}.
		\end{align*}
		
		\item[(c)] For any balls \(A,B \in \mathfrak{B}\) with \(A \subset B\),
		\begin{align*}
			\Delta_T(A,B^{[k]}) \lesssim \left(\mathscr{C}_2(T) + \|T\|_{ L^{(\Psi, \psi)}\to wL^{(\Psi,\psi)}}+\Delta_T(A,B)\right) \frac{\mathscr{C}_0^k \phi \left(\mathscr{C}_0^k\mu(B)\right)}{\phi(\mu(B))}, \quad k\geq 1.
		\end{align*}
		
		\item[(d)] If \(\mathfrak{B}\) also satisfies the doubling condition, then \(T\) satisfies condition (\(\Psi\)-BOO-I). Consequently, \(T\) is a \(\mathbb{B}\)-valued \(\Psi\)-BOO.
		
	\end{itemize*}
\end{lemma}

\begin{proof}
	We begin by proving part (a). Set  
	\begin{align*}
		K_0 = \Psi^{-1} (\lambda\psi(\mu(A))) \|T\|_{L^{(\Psi,\psi)} \to wL^{(\Psi, \psi)}} \|f\|_{\Phi,\phi,A}.
	\end{align*}
	Since \(A \subset B\) and \(T \in \mathbb{W}_{\Psi, \psi}\), applying estimate (\ref{D1.6}) with \(A\) and \(B^\P\) in place of \(A\) and \(B\) yields
	\begin{align*}
		\mu \left( \left\{x \in A : \|\mathscr{T}(f\, \mathbb{I}_{B^\P })(x)\|_\mathbb{B} > K_0 \right\} \right) \leq \frac{1}{\lambda} \mu(A).
	\end{align*}
	Similarly, applying (\ref{D1.6}) with \(A\) and \(A^\P\) gives
	\begin{align*}
		\mu \left( \left\{x \in A : \|\mathscr{T}(f\,\mathbb{I}_{A^\P })(x)\|_\mathbb{B} > K_0 \right\} \right) \leq \frac{1}{\lambda} \mu(A).
	\end{align*}
	Consequently,
	\begin{align*}
		\mu \left( \left\{x \in A : \|\mathscr{T}(f\, \mathbb{I}_{B^\P })(x) - \mathscr{T}(f\,\mathbb{I}_{A^\P})(x)\|_\mathbb{B} > 2 K_0 \right\} \right) \leq \frac{2}{\lambda}\mu(A).
	\end{align*}
	Now choose \(\lambda\) such that \(2/\lambda < 1\). This guarantees the existence of a point \(x' \in A\) for which
	\begin{align}\label{3.3}
		\|\mathscr{T}(f\, \mathbb{I}_{B^\P })(x') - \mathscr{T} & (f\, \mathbb{I}_{A^\P})(x')\|_\mathbb{B} \lesssim \Psi^{-1} (\lambda \psi(\mu(A))) \|T\|_{L^{(\Psi,\psi)} \to wL^{(\Psi, \psi)}} \|f\|_{\Phi,\phi,A} \nonumber \\
		&\lesssim \|T\|_{L^{(\Psi,\psi)} \to wL^{(\Psi, \psi)}} \|f\|_{\Phi,\phi,A}.
	\end{align}
	Next, for any \(x \in A\), replace \(f\) and \(B^\P\) in condition (\(\Psi\)-BOO-II) by \(f\,\mathbb{I}_{B^\P}\) and \(A^\P\). Using Lemma \ref{Lm 3.3} part (2) we obtain  
	\begin{align}\label{3.4}
		\|(\mathscr{T}(f\,&\mathbb{I}_{B^\P}) - \mathscr{T}(f\, \mathbb{I}_{A^\P}))(x) - (\mathscr{T}(f\,\mathbb{I}_{B^\P}) - \mathscr{T}(f\,\mathbb{I}_{A^\P}))(x')\|_{\mathbb{B}} \nonumber \\
		&\leq \mathscr{C}_2(T) \,\langle \|f\,\mathbb{I}_{B^\P}\| \rangle_{\Phi, \phi,A} \lesssim \mathscr{C}_2(T) \frac{\mu(B) \phi(\mu(B))}{\mu(A)\phi(\mu(A))} \|f\|_{\Phi,\phi, B^\P}.
	\end{align}
	Combining (\ref{3.3}) and (\ref{3.4}) we deduce  
	\begin{align*}
		\|\mathscr{T}(f\,\mathbb{I}_{B^\P})(x) - \mathscr{T}(f\, \mathbb{I}_{A^\P})(x) \|_{\mathbb{B}} \lesssim \frac{\mu(B) \phi(\mu(B))}{\mu(A)\phi(\mu(A))} (\mathscr{C}_2(T) + \|T\|_{L^{( \Psi, \psi)} \to wL^{(\Psi, \psi)}}) \|f\|_{\Phi,\phi, B^\P}.
	\end{align*}
	This completes the proof of part (a).
	
	Now we prove part (b). Assume \(A \subset B \subset C\) and take \(x \in A\). By part (a) we have  
	\begin{align}\label{3.5}
		\|\mathscr{T}(f\,\mathbb{I}_{C^\P})(x) - \mathscr{T}(f\, \mathbb{I}_{B^\P})(x) \|_{\mathbb{B}} \lesssim \frac{\mu(C) \phi(\mu(C))}{\mu(B)\phi(\mu(B))} (\mathscr{C}_2(T) + \|T\|_{L^{( \Psi, \psi)} \to wL^{(\Psi, \psi)}}) \|f\|_{\Phi,\phi, C^\P}.
	\end{align}
	From the definition of \(\Delta_T(A,B)\),
	\begin{align}\label{3.6}
		\|\mathscr{T}(f\,\mathbb{I}_{B^\P}) & (x) - \mathscr{T}(f\, \mathbb{I}_{A^\P})(x) \|_{\mathbb{B}} \leq \Delta_T(A,B) \|f\|_{\Phi,\phi, B^\P} \nonumber \\
		&\lesssim \Delta_T(A,B)\frac{\mu(C^\P)\phi(\mu(C^\P))}{\mu(B^\P) \phi(\mu(B^\P))} \|f\|_{\Phi,\phi,C^\P} \nonumber \\
		&\lesssim \Delta_T(A,B)\frac{\mu(C)\phi(\mu(C))}{\mu(B)\phi(\mu (B))}\|f\|_{\Phi,\phi,C^\P},
	\end{align}
	where the last two inequalities follow from Proposition \ref{propB} and the fact that \(\phi \in \mathscr{G}\). Adding (\ref{3.5}) and (\ref{3.6}) we obtain  
	\begin{align*}
		\|\mathscr{T}(f\, & \mathbb{I}_{C^\P})(x) - \mathscr{T}(f\, \mathbb{I}_{A^\P})(x) \|_{\mathbb{B}} \\
		&\lesssim (\mathscr{C}_2(T) + \|T\|_{L^{( \Psi, \psi)} \to wL^{(\Psi,\psi)}}+\Delta_T(A,B))\frac{\mu(C)\phi(\mu(C))}{\mu(B) \phi(\mu (B))} \|f\|_{\Phi,\phi, C^\P}.
	\end{align*}
	This establishes part (b).
	
	Part (c) is verified as follows. For \(k \geq 1\) we clearly have \(A \subset B \subset B^{[k]}\). Applying part (b) together with the property \(\phi \in \mathscr{G}\) gives  
	\begin{align*}
		\Delta_T(A, & B^{[k]}) \lesssim \left(\mathscr{C}_2(T) + \|T\|_{ L^{(\Psi, \psi)}\to wL^{(\Psi,\psi)}}+\Delta_T(A,B)\right) \frac{\mu(B^{[k]})\phi\left(\mu(B^{[k]})\right)}{\mu(B)\phi(\mu(B))} \\
		&\lesssim\left(\mathscr{C}_2(T) + \|T\|_{L^{(\Psi, \psi)}\to wL^{ (\Psi,\psi)}}+\Delta_T(A,B)\right)\frac{\mathscr{C}_0^k\phi\left( \mathscr{C}_0^k\mu(B)\right)}{\phi(\mu(B))}.
	\end{align*}
	
	Finally, we prove part (d). From the definition of the doubling condition, we fix a ball \(A \in \mathfrak{B}\) with \(A^\P \ne X\). Then there exist a constant \(\theta > 1\) and a ball \(B \in \mathfrak{B}\) such that \(A \subset B\) and \(\mu(B) \leq \theta \mu(A)\). Applying part (a) we obtain  
	\begin{align*}
		\Delta_T(A,B) \lesssim \left(\mathscr{C}_2(T)+\|T\|_{L^{(\Psi,\psi)} \to wL^{(\Psi,\psi)}}\right) \theta \frac{\phi(\theta \mu(A))}{\phi (\mu(A))} \lesssim \mathscr{C}_2(T)+\|T\|_{L^{(\Psi,\psi)} \to wL^{(\Psi,\psi)}}.
	\end{align*}
	This shows that \(T\) satisfies condition (\(\Psi\)-BOO-I).
\end{proof}

\begin{lemma}\label{Lm 3.5}
	Let \((X,\,\mathfrak{M},\,\mu)\) be a measure space endowed with a ball-basis \(\mathfrak{B}\). Assume that \(T\) is a \(\mathbb{B}\)-valued \(\Psi\)-BOO and that \(T \in \mathbb{W}_{\Psi, \psi}\). Define
	\begin{align*}
		\mathscr{C}(T) := \mathscr{C}_1(T) + \mathscr{C}_2(T) + \|T\|_{ L^{(\Psi, \psi)}(X) \to wL^{(\Psi,\psi)}(X)}.
	\end{align*}
	Then the following statements hold.
	
	\begin{itemize*}
		\item[(1)] For every ball \(B \in \mathfrak{B}\) there exists a ball \(\widetilde{B} \in \mathfrak{B}\) such that
		\begin{align*}
			B^{[2]} \subset \widetilde{B}, \quad \Delta_T(B^{[2]}, \widetilde{B}) \lesssim \mathscr{C}(T), \quad \text{and either} \; \widetilde{B}^{[1]}=\widetilde{B} \; \text{or} \; \mu(\widetilde{B}) \geq 2 \mu(B).
		\end{align*}
		
		\item[(2)] For every ball \(B \in \mathfrak{B}\) satisfying \(B^\P = B\), there exists \(\widetilde{B} \in \mathfrak{B}\) such that
		\begin{align*}
			B^{[2]} \subset \widetilde{B}, \quad \Delta_T(B^{[2]}, \widetilde{B}) \lesssim \mathscr{C}(T), \quad \text{and} \quad \mu(\widetilde{B}) \geq 2 \mu(B).
		\end{align*}
		
		\item[(3)] For every ball \(B \in \mathfrak{B}\) there exists a sequence \(\{B_k\}_{k \geq 0} \subset \mathfrak{B}\) with \(B_0 = B\) such that for all \(k \geq 0\),
		\begin{align*}
			X = \bigcup_k B_k, \quad B_k^{[2]} \subset B_{k+1}, \quad \text{and} \quad \Delta_T(B_k^{[2]}, B_{k+1}) \lesssim \mathscr{C}(T).
		\end{align*}
		
	\end{itemize*}
\end{lemma}

\begin{proof}
	We begin by proving part (1). Fix a ball \(B \in \mathfrak{B}\) and define  
	\begin{align*}
		\mathscr{B} := \{A \in \mathfrak{B}: B \subset A^\P\},
	\end{align*}
	Set  
	\begin{align}
		\label{3.7} & a := \sup_{A \in \mathscr{B}:\mu(A) \leq 2\mu(B)}\mu(A) \leq 2 \mu(B), \\
		\label{3.8} & b := \inf_{A \in \mathscr{B}:\mu(A) > 2\mu(B)}\mu(A) \geq 2 \mu(B).
	\end{align}
	By the definitions of \(a\) and \(b\) we can choose balls \(B_1, B_2 \in \mathscr{B}\) such that
	\begin{align}\label{3.9}
		\frac a2 < \mu(B_1) \leq a \leq 2\mu(B) \leq b \leq \mu(B_2) < 2b.
	\end{align}
	We now distinguish two cases.
	
	\begin{itemize}[leftmargin=2.2em]
		\item {\bf {\emph Case} 1} \(b > a \mathscr{C}_0^2\). Take \(\widetilde{B} := B_1^\P \supset B\). From (\ref{3.9}) we obtain  
		\begin{align}\label{3.10}
			\mu(\widetilde{B}^\P) = \mu(B_1^{[2]}) \leq \mathscr{C}_0^2 \mu(B_1) \leq a \mathscr{C}_0^2 < b.
		\end{align}
		By construction (\ref{3.7}) and (\ref{3.8}), there is no ball \(B' \in \mathscr{B}\) with \(a< \mu(B')<b\). Hence (\ref{3.10}) implies \(\mu(\widetilde{B}^\P)<a<2\mu(B_1)\). Since \(\widetilde{B}^\P \cap B_1 \ne \varnothing\), ball-hull control property forces \(\widetilde{B}^\P \subset B_1^\P = \widetilde{B} \subset \widetilde{B}^\P\), so that \(\widetilde{B}^\P = \widetilde{B}\). Consequently \(B^{[2]} \subset \widetilde{B}^{[2]} = \widetilde{B}^\P = \widetilde{B}\). Using (\ref{3.9}) again we get  
		\begin{align}\label{3.11}
			\mu((B^{[2]}) \leq \mu(\widetilde{B}) = \mu(B_1^\P) \leq \mathscr{C}_0 \mu(B_1) \leq 2\mathscr{C}_0 \mu(B) \leq 2\mathscr{C}_0 \mu(B^{[2]}).
		\end{align}
		
		\item {\bf {\emph Case} 2} \(b \leq a \mathscr{C}_0^2\).  
		Now take \(\widetilde{B} := B_2^{[3]} = (B_2^\P)^{[2]} \supset B^{[2]}\). From (\ref{3.9}) we have
		\begin{align*}
			\mu(\widetilde{B}) \geq \mu(B_2) \geq b \geq 2\mu(B),
		\end{align*}
		and  
		\begin{align}\label{3.12}
			\mu(B^{[2]}) \leq \mu(&\widetilde{B}) = \mu(B_2^{[3]}) \leq \mathscr{C}_0^3 \mu(B_2) \nonumber \\
			&\leq 2b \mathscr{C}_0^3 \leq 2a \mathscr{C}_0^5 \leq 4 \mathscr{C}_0^5 \mu(B) \leq 4 \mathscr{C}_0^5 \mu(B^{[2]}).
		\end{align}
		
	\end{itemize}
	Both (\ref{3.11}) and (\ref{3.12}) give \(\mu(\widetilde{B}) \sim \mu(B^{[2]})\). Applying Lemma \ref{Lm 3.4} part (a) we obtain  
	\begin{align*}
		\Delta_T(B^{[2]},\widetilde{B}) \lesssim \left(\mathscr{C}_2(T) +\ |T\|_{L^{(\Psi, \psi)}\to wL^{(\Psi,\psi)}}\right) \frac{\mu(\widetilde{B})\phi(\mu(\widetilde{B}))}{\mu(B^{[2]})\phi(\mu(B^{[2]}))} \lesssim \mathscr{C}(T).
	\end{align*}
	where the last inequality follows from the equivalence \(\mu(\widetilde{B}) \sim \mu(B^{[2]})\) together with the properties of \(\phi\). This completes the proof of part (1).
	
	Now, we prove part (2). Fix a ball \(B \in \mathfrak{B}\) satisfying \(B^\P = B\). By condition (\(\Psi\)-BOO-I), there exists a ball \(\widetilde{B} \in \mathfrak{B}\) such that \(B \subsetneq \widetilde{B}\) and \(\Delta_T(B, \widetilde{B}) \lesssim \mathscr{C}_1(T)\). Using the identity \(B^\P = B\) twice we obtain  
	\begin{align*}
		B^{[2]} = B \subsetneq \widetilde{B} \quad \text{and} \quad \Delta_T(B^{[2]}, \widetilde{B}) = \Delta_T(B, \widetilde{B}) \lesssim \mathscr{C}_1(T).
	\end{align*}
	Assume for contradiction that \(\mu(\widetilde{B}) \leq 2 \mu(B)\). Then by ball-hull control property and the fact that \(\widetilde{B} \cap B \ne \varnothing\), we would have \(\widetilde{B} \subset B^\P = B\), contradicting the strict inclusion \(B \subsetneq \widetilde{B}\). Therefore \(\mu(\widetilde{B}) \geq 2 \mu(B)\), which completes the proof of part (2).
	
	Part (3) follows by an inductive construction based on parts (1) and (2). Set \(B_0 := B\). Suppose the ball \(B_k\) has already been chosen; using either part (1) or part (2) we select \(B_{k+1} \in \mathfrak{B}\) (\(k \geq 0\)) such that  
	\begin{align*}
		B_k^{[2]} \subset B_{k+1}, \quad \Delta_T(B_k^{[2]}, B_{k+1}) \lesssim \mathscr{C}(T), \quad \text{and} \quad \mu(B_{k+1}) \geq 2 \mu(B_k), \quad k \geq 0.
	\end{align*}
	Thus we obtain a sequence \(\{B_k\}_{k \geq 0}\subset\mathfrak{B}\). Lemma \ref{Lm 3.1} then yields \(X = \bigcup_{k \geq 0} B_k\), and the desired properties hold for all \(k \geq 0\). This completes the proof of the Lemma \ref{Lm 3.5}.
\end{proof}

\begin{remark}
	For convenience, we note that hereafter the notation \(\mathscr{C}(T)\) will always refer to the constant introduced in Lemma \ref{Lm 3.5}; that is,
	\[
	\mathscr{C}(T) := \mathscr{C}_1(T) + \mathscr{C}_2(T) + \|T\|_{L^{(\Psi, \psi)}(X) \to wL^{(\Psi,\psi)}(X)}.
	\]
\end{remark}

\begin{lemma}\label{Lm 3.6}
	Let \((X,\,\mathfrak{M},\,\mu)\) be a measure space endowed with a ball-basis \(\mathfrak{B}\), and assume that \(T\) is a \(\Psi\)-BOO with \(T \in \mathbb{W}_{\Psi, \psi}\). Set \(\lambda \geq 3 \mathscr{C}_0^4\), let \(F \subset X\) be a measurable set, and choose a ball \(A \in \mathfrak{B}\) such that
	\begin{align*}
		F \cap A \ne \varnothing \quad \text{and} \quad \mu(F) \leq \lambda^{-1}\mu(A).
	\end{align*} 
	Then there exists a family \(\mathcal{G} \subset \mathfrak{B}\) satisfying the following properties:
	
	\begin{itemize*}
		\item[(a)] For every \(G \in \mathcal{G}\), \(\displaystyle F \cap A^\P \cap G \ne \varnothing\),
		
		\vspace{6pt}
		
		\item[(b)] \(\displaystyle F \cap A^\P \subset \bigcup_{G \in \mathcal{G}} G \; a.s.\),
		
		\vspace{6pt}
		
		\item[(c)] \(\displaystyle \mu\left(\bigcup_{G \in \mathcal{G}} G^\P\right) \leq 3\lambda^{-1}\mathscr{C}_0^2 \mu(A)\),
		
		\vspace{6pt}
		
		\item[(d)] For each \(G \in \mathcal{G}\) one can find a ball \(\widetilde{G} \in \mathfrak{B}\) such that  
		\begin{align*}
			\widetilde{G} \not\subset F, \quad G^{[2]} \subset \widetilde{G} \subset A^\P, \quad \text{and} \quad \Delta_T(G^{[2]},\widetilde{G}) \lesssim \mathscr{C}(T),
		\end{align*}
		with constants independent of \(F\), \(A\), \(\lambda\) and \(\mathcal{G}\).
	\end{itemize*}
\end{lemma}

\begin{proof}
	We begin by showing parts (a) and (b). Set \(E := F \cap A^\P\) and apply Lemma \ref{Lm more} part (3) to obtain a family \(\mathfrak{B}' \subset \mathfrak{B}\) such that for every \(B \in \mathfrak{B}'\),
	\begin{align}\label{3.13}
		E \cap B \ne \varnothing, \quad E \subset \bigcup_{B \in \mathfrak{B}'}B \; a.s., \quad \text{and} \quad \sum_{B \in \mathfrak{B}'}\mu(B) \leq 2 \mathscr{C}_0 \mu(E).
	\end{align}
	Fix \(B \in \mathfrak{B}'\). By Lemma \ref{Lm 3.5} part (3) we can find a sequence \(\{B_k\}_{k \geq 0} \subset \mathfrak{B}\) with \(B_0 = B\) satisfying  
	\begin{align}\label{3.14}
		X = \bigcup_{k \geq 0} B_k, \quad B_k^{[2]} \subset B_{k+1}, \quad \text{and} \quad \Delta_T(B_k^{[2]}, B_{k+1}) \lesssim \mathscr{C}(T), \quad k \geq 0.
	\end{align}
	We now construct \(\mathcal{G}\). For each \(B \in \mathfrak{B}'\) let \(k \geq 0\) be the smallest index for which \(B_{k+1}^\P \not\subset F\) and take \(G = B_k\). Denote by \(\mathcal{G}\) the collection of all such balls \(G\). Thus, for every \(G \in \mathcal{G}\) there exists \(B \in \mathfrak{B}'\) such that either \(G = B\) or, for some \(k \geq 1\), \(G = B_k \supset B^{[2]}\). Using the first two relations in (\ref{3.13}) we obtain  
	\begin{align*}
		F \cap A^\P \cap G = E \cap G \ne \varnothing \quad \text{and} \quad  F \cap A^\P = E \subset \bigcup_{G \in \mathcal{G}}G \; a.s.,
	\end{align*}
	which establishes parts (a) and (b).
	
	Next we prove part (c). For the family \(\mathcal{G}\) constructed above, observe that the condition \(G \not\subset F\) leads to the following two alternatives:
	\begin{itemize*}
		\item \(G = B_0\);
		\item or, for some \(k \geq 1\), \(G = B_k\) implies \(G^\P = B_k^\P \subset F\).
	\end{itemize*}
	Consequently,
	\begin{align}\label{3.15}
		\mu\left(\bigcup_{G \in\mathcal{G}}G^\P\right) = \mu & \left(\bigcup_{G \in \mathcal{G}: G^\P \subset F} G^\P\right) + \mu\left(\bigcup_{G \in\mathcal{G}: G^\P \not\subset F} G^\P \right)  \nonumber  \\
		& \leq \mu(F)+ \mu\left( \bigcup_{B \in \mathfrak{B}'} B^\P\right) \leq \mu(F) + \sum_{B \in \mathfrak{B}'}\mu(B^\P)  \nonumber  \\
		&\leq \mu(F) + \mathscr{C}_0 \sum_{B \in \mathfrak{B}'}\mu(B) \leq \mu(F) + \mathscr{C}_0 \cdot 2 \mathscr{C}_0 \mu(E)  \nonumber  \\
		&\leq (1 + 2 \mathscr{C}_0^2) \mu(F) \leq (1 + 2\mathscr{C}_0^2)\lambda^{-1}\mu(A)\leq 3 \lambda^{-1} \mathscr{C}_0^2 \mu(A)
	\end{align}
	which gives part (c).
	
	Finally we verify part (d). Suppose \(G = B_k \in \mathcal{G}\) for some \(k \geq 0\). Define  
	\begin{align*}
		\widetilde{G} := 
		\left\{ 
		\begin{array}{ll}
			A^\P,  & \mu(B_{k+1}^\P) > \mu(A), \\
			B_{k+1}^\P,  & \mu(B_{k+1}^\P) \leq \mu(A). \\
		\end{array}
		\right.
	\end{align*}
	By the assumptions we have \(\mu(A^\P)\geq\mu(A)\geq \lambda\mu(F)\geq \mu(F)\), which implies \(A^\P \not\subset F\). Consequently, for every \(G \in \mathcal{G}\), it follows that \(\widetilde{G} \not\subset F\). Moreover, from (\ref{3.15}) we obtain \(\mu(G^{[2]}) \leq \mathscr{C}_0^2\mu(G) \leq \mathscr{C}_0^2 \cdot 3 \lambda^{-1} \mathscr{C}_0^2 \mu(A) \leq \mu(A)\). Since \(G \cap A \ne \varnothing\), ball-hull control property yields \(G^{[2]} \subset A^\P\).
	\begin{itemize}[leftmargin=2.2em]
		\item If \(\mu(B_{k+1}^\P) > \mu(A)\), then \(G^{[2]} \subset A^\P = \widetilde{G} \subset B_{k+1}^{[2]}\). Applying Lemma \ref{Lm 3.3} part (1) and Lemma \ref{Lm 3.4} part (c) we get  
		\begin{align}\label{3.16}
			\Delta_T(G^{[2]},\widetilde{G}) \lesssim \Delta_T(B_k^{[2]}, B_{k+1}^{[2]}) \lesssim \mathscr{C}(T) + \Delta_T(B_k^{[2]}, B_{k+1}) \lesssim \mathscr{C}(T).
		\end{align}
		
		\item If \(\mu(B_{k+1}^\P) \leq \mu(A)\), then the conditions \(\mu(B_{k+1}^\P) \leq \mu(A)\) and \(B_{k+1}^\P \cap A \ne \varnothing\) together with ball-hull control property imply \(G^{[2]} = B_k^{[2]} \subset B_{k+1} \subset B^\P_{k+1} = \widetilde{G} \subset A^\P\). Again using Lemma \ref{Lm 3.3} part (1) and Lemma \ref{Lm 3.4} part (c) we obtain  
		\begin{align}\label{3.17}
			\Delta_T(G^{[2]},\widetilde{G}) = \Delta_T(B_k^{[2]},B^\P_{k+1}) \lesssim \mathscr{C}(T) + \Delta_T(B_k^{[2]}, B_{k+1}) \lesssim \mathscr{C}(T).
		\end{align}
	\end{itemize}
	The last estimates in (\ref{3.16}) and (\ref{3.17}) follow from the bound in (\ref{3.14}). This completes the proof of part (d) and hence of the lemma.
\end{proof}

Let \(T\) be a \(\mathbb{B}\)-valued operator such that \(Tf(x) = \|\mathscr{T}f(x)\|_{\mathbb{B}}\). The truncation operator is given by  
\begin{align*}
	T^*f(x) := \sup_{B \in \mathfrak{B} : x \in B} \left\| \mathscr{T}f(x) - \mathscr{T}(f \, \mathbb{I}_{B^\P})(x) \right\|_{\mathbb{B}}.
\end{align*}

\begin{theorem}\label{Th 3.7}
	Let \((X,\,\mathfrak{M},\,\mu)\) be a measure space endowed with a ball-basis \(\mathfrak{B}\), and let \((\Psi,\psi) \in E_{mb}\). Assume that \(T\) is a \(\mathbb{B}\)-valued linear operator satisfying the condition (\(\Psi\)-BOO-II) and that \(T \in \mathbb{W}_{\Psi, \psi}\). Then for every \(f \in L^{(\Psi, \psi)}(X)\), we have
	
	\begin{align*}
		\|T^*\|_{L^{(\Psi, \psi)}(X) \to wL^{(\Psi, \psi)}(X)} \lesssim \mathscr{C}_2(T) + \|T\|_{L^{(\Psi, \psi)}(X) \to wL^{(\Psi, \psi)}(X)}.
	\end{align*}
\end{theorem}

\begin{proof}
	Fix a ball \(B \in \mathfrak{B}\)) and a number \(\lambda > 0\). Define \(E:=\{x \in B: T^*f(x)> \lambda\}\), and assume for the moment that \(E\) is bounded. For each \(x \in E\) there exists a ball \(B_x = B(x) \in \mathfrak{B}\) such that \(\left\| \mathscr{T}f(x) - \mathscr{T}(f \, \mathbb{I}_{B_x^\P})(x) \right\|_{\mathbb{B}} > \lambda\). Notice that \(E \subset \bigcup_{x \in E}B_x\) and that \(E\) is bounded. By Lemma \ref{Lm more} part (1) we can extract a pairwise disjoint subfamily \(\{B_k\}_{k \geq 0}\) with \(E \subset \bigcup_k B_k^\P\). For a parameter \(\kappa > 1\) set  
	\begin{align*}
		\widetilde{B}_k = \left\{x \in B_k: |T(f\,\mathbb{I}_{B_k^\P})(x)|< \Psi^{-1} (\kappa \, \psi(\mu(A))) \|T\|_{L^{(\Psi, \psi)} \to wL^{(\Psi, \psi)}}\langle \|f\|\rangle_{\Phi,\phi,B_k} \right\}.
	\end{align*}
	Since \(T \in \mathbb{W}_{\Psi, \psi}\), we have \(\mu(X \setminus \widetilde{B}_k) \leq \kappa^{-1} \mu(B_k)\). Consequently,
	\begin{align}\label{3.18}
		\mu(\widetilde{B}_k) \geq \mu(B_k) - \mu(X\setminus\widetilde{B}_k) \geq (1-\kappa^{-1}) \mu(B_k).
	\end{align}
	Now fix an index \(k\) and write \(B_k = B(x_k)\). Take a point \(x \in \widetilde{B}_k \setminus \{\mathcal{M}_{\mathfrak{B},\Phi,\phi}f > \delta \lambda\}\), where  
	\begin{align}\label{3.19}
		\delta := \frac{1}{4\Psi^{-1} (\kappa \, \psi(\mu(A))) (\mathscr{C}_2 (T) + \|T\|_{L^{(\Psi, \psi)} \to wL^{(\Psi, \psi)}})}.
	\end{align}
	From the definition of \(\widetilde{B}_k\) and (\ref{3.19}) we obtain  
	\begin{align}\label{3.20}
		\|\mathscr{T}(f \, \mathbb{I}_{B_k^\P})(x)\|_{\mathbb{B}} = |T(f\, \mathbb{I}_{B_k^\P}) & (x)|< \Psi^{-1} (\kappa \, \psi(\mu(A))) \|T\|_{L^{(\Psi, \psi)} \to wL^{(\Psi, \psi)}} \langle \|f\|\rangle_{\Phi,\phi,B_k} \nonumber \\
		&\leq \Psi^{-1} (\kappa \, \psi(\mu(A))) \|T\|_{L^{(\Psi, \psi)} \to wL^{(\Psi, \psi)}} \mathcal{M}_{\mathfrak{B},\Phi,\phi}f(x) \nonumber \\
		&\leq \Psi^{-1} (\kappa \, \psi(\mu(A))) \|T\|_{L^{(\Psi, \psi)} \to wL^{(\Psi, \psi)}} \delta \lambda \leq \frac{\lambda}{4}.
	\end{align}
	On the other hand, condition (\(\Psi\)-BOO-II) together with (\ref{3.19}) (for a suitable choice of \(\kappa\)) gives  
	\begin{align}\label{3.21}
		\mathscr{H}(x) := & \|(\mathscr{T}f - \mathscr{T}(f \,\mathbb{I}_{B_k^\P}))(x) - (\mathscr{T}f - \mathscr{T}(f\, \mathbb{I}_{B_k^\P}))(x_k) \|_{\mathbb{B}}  \nonumber \\
		&\leq \mathscr{C}_2(T) \, \langle \|f\|\rangle_{\Phi,\phi,B} \leq \mathscr{C}_2(T) \mathcal{M}_{\mathfrak{B},\Phi,\phi}f(x) \leq \mathscr{C}_2(T) \delta \lambda \leq \frac{\lambda}{4}.
	\end{align}
	Combining (\ref{3.20}) and (\ref{3.21}) we get  
	\begin{align*}
		\lambda \leq \|\mathscr{T} & f(x_k) - \mathscr{T} (f \, \mathbb{I}_{B_k^\P}) (x_k)\|_{\mathbb{B}}  \nonumber \\
		&\leq \| \mathscr{T}f(x) \|_{\mathbb{B}} + \|\mathscr{T}(f \, \mathbb{I}_{B_k^\P})(x) \|_{\mathbb{B}} + \mathscr{H}(x) \leq \|\mathscr{T}f(x)\|_{\mathbb{B}} + \frac{\lambda}{2},
	\end{align*}
	which implies \(\|\mathscr{T}f(x)\|_{\mathbb{B}} = |Tf(x)| \geq \lambda / 2\). Therefore, we can partition \(\widetilde{B}_k\) into two pieces:
	\begin{align}\label{3.22}
		\bigcup_k \widetilde{B}_k \subset \{x \in B:|Tf(x)| \geq \lambda/2\} \cup \{x \in B:\mathcal{M}_{\mathfrak{B},\Phi,\phi}f(x) > \delta \lambda\}.
	\end{align}
	From (\ref{3.18}) and (\ref{3.22}) we deduce  
	\begin{align*}
		\mu(E) & \leq \sum_k \mu(B_k^\P) \lesssim \sum_k \mu(B_k) \lesssim \sum_k \mu(\widetilde{B}_k)  \\
		&\leq \mu(\{x \in B:|Tf(x)| \geq \lambda/2\}) + \mu(\{x \in B: \mathcal{M}_{\mathfrak{B}, \Phi, \phi}f(x) > \delta \lambda\})  \\
		&\leq \frac{\mu(B)\psi(\mu(B))}{\Psi\left(\frac{\lambda/2} {\|T\|_{L^{(\Psi, \psi)} \to wL^{(\Psi, \psi)}}\|f\|_{\Psi,\psi, B}}\right)} + \frac{\mu(B)\psi(\mu(B))}{\Psi\left(\frac{\delta \lambda} {\|\mathcal{M}_{\mathfrak{B}, \Phi, \phi}\|_{L^{(\Psi, \psi)} \to wL^{(\Psi, \psi)}}\|f\|_{\Psi,\psi, B}}\right)}  \\
		&\lesssim \frac{\mu(B)\psi(\mu(B))}{\Psi\left(\frac{\lambda} {(\mathscr{C}_2(T) + \|T\|_{L^{(\Psi, \psi)} \to wL^{(\Psi, \psi)}}) \|f\|_{\Psi,\psi, B}}\right)}.
	\end{align*}
	Thus  
	\begin{align*}
		\Psi\left(\frac{\lambda} {(\mathscr{C}_2(T) + \|T\|_{L^{(\Psi, \psi)} \to wL^{(\Psi, \psi)}}) \|f\|_{\Psi,\psi, B}}\right)\mu(\{x \in B: T^*f(x)> \lambda\}) \lesssim \mu(B)\psi(\mu(B)).
	\end{align*}
	This inequality implies that 
	\begin{align*}
		\sup_{\lambda > 0}\lambda \, \mu\left(\left\{x \in B: \Psi\left( \frac{|T^*f(x)|}{(\mathscr{C}_2(T)+\|T\|_{L^{(\Psi, \psi)} \to wL^{ (\Psi, \psi)}})\|f\|_{\Psi,\psi, B}}\right) > \lambda\right\}\right) \lesssim \mu(B)\psi(\mu(B)).
	\end{align*}
	Consequently, it follows that
	\begin{align*}
		\|T^*f\|_{w,\Psi,\psi,B} \lesssim (\mathscr{C}_2(T) + \|T\|_{L^{ (\Psi, \psi)} \to wL^{(\Psi, \psi)}}) \|f\|_{\Psi,\psi, B}.
	\end{align*}
	Since the ball \(B\) was arbitrary, the above estimate implies that
	\begin{align*}
		\|T^*f\|_{wL^{(\Psi,\psi)}} \lesssim (\mathscr{C}_2(T) + \|T\|_{L^{ (\Psi, \psi)} \to wL^{(\Psi, \psi)}}) \|f\|_{L^{(\Psi,\psi)}}.
	\end{align*}
	Therefore, the proof is complete.
\end{proof}

Moreover, for the subsequent analysis, we need to introduce the following operator:
\begin{align*}
	\Upsilon f(x) := \max\left\{|Tf(x)|, \, T^*f(x), \, \mathscr{C}(T) \mathcal{M}_{\mathfrak{B}, \Phi, \phi} f(x) \right\}.
\end{align*}
Hence, it is immediate that
\begin{align}\label{3.23}
	\|\Upsilon\| := \|\Upsilon\|_{L^{ (\Psi, \psi)} \to wL^{(\Psi, \psi)}} \lesssim \mathscr{C}(T)
\end{align}
For a measure space \((X,\,\mathfrak{M},\,\mu)\) endowed with a ball-basis \(\mathfrak{B}\), we fix in the sequel a ball \(B_0 \in \mathfrak{B}\) and \(f \in L^{(\Phi, \phi)} (X, \mu)\) satisfying \(\|f\|_{L^{(\Phi, \phi)} (B_0, \mu)} \geq \frac{1}{2} \|f\|_{L^{(\Phi, \phi)} (X, \mu)}\) and
\begin{align}\label{3.24}
	{\rm supp}\,f \subset B_0^{[3]}, 
\end{align}
which entails no loss of generality for our analysis. This is because
\begin{align*}
	\left\| f \cdot \mathbf{1}_{B_0^{[3]}} \right\|_{L^{(\Phi, \phi)} (B_0, \mu)} & = \sup_{B \subset B_0 \in \mathfrak{B}} \left\| f \cdot \mathbf{1}_{B_0^{[3]}} \right\|_{\Phi,\phi,B} = \sup_{B \subset B_0 \in \mathfrak{B}} \left\| f \right\|_{\Phi,\phi,B} = \|f\|_{L^{(\Phi, \phi)} (B_0, \mu)} \\
	& \geq \frac{1}{2} \|f\|_{L^{(\Phi, \phi)} (X, \mu)} \geq \frac{1}{2} \left\| f \cdot \mathbf{1}_{B_0^{[3]}} \right\|_{L^{(\Phi, \phi)} (X, \mu)}.
\end{align*}

\begin{lemma}\label{Lm 3.8}
	Let \((X,\,\mathfrak{M},\,\mu)\) be a measure space endowed with a ball-basis \(\mathfrak{B}\). Let \(\Psi,\Phi\in\mathscr{Y}\) and \(\psi, \phi \in \mathscr{G}\). Suppose \(T\) is a \(\mathbb{B}\)-valued \(\Psi\)-BOO and \(T \in \mathbb{W}_{\Psi, \psi}\). Take \(\lambda \geq 3 \mathscr{C}_0^4\). For the ball \(B_0\) in (\ref{3.24}), there exists a family \(\mathcal{G} = \mathcal{G}(B_0) \subset \mathfrak{B}\) satisfying the following conditions:
	
	\begin{itemize}[leftmargin=1.5em]
		\item[({\it a})] For any \(A \subset \mathcal{G}\), there exists \(\mathcal{F}(A) \subset \mathcal{G}\) such that:
		
		\begin{itemize*}
			\item[({\it a}1)] \(A^\P \cap B \ne \varnothing\) for all \(B \in \mathcal{F}(A)\),
			
			\vspace{5pt}
			
			\item[({\it a}2)] \(\mu\left(\bigcup_{B \in \mathcal{F}(A)}B^\P\right) \leq 3 \lambda^{-1} \mathscr{C}_0^2 \mu(A)\),
			
			\vspace{5pt}
			
			\item[({\it a}3)] \(\Upsilon(f \, \mathbb{I}_{A^{[3]}})(x) \lesssim \mathscr{C}(T) \Psi^{-1}(\lambda \phi(\mu(A))) \|f\|_{\Phi, \phi, A^{[3]}}, \quad a.e. \; x \in A^\P \setminus \bigcup_{B \in \mathcal{F}(A)}B\).
		\end{itemize*}
		
		\vspace{5pt}
		
		\item[({\it b})] For any \(B \in \mathcal{F}(A)\), there exist \(\widetilde{B} \in \mathfrak{B}\) and \(\xi \in \widetilde{B}\), which yields the following result:
		\begin{itemize*}
			
			\vspace{5pt}
			
			\item[({\it b}1)] \(B^{[2]} \subset \widetilde{B} \subset A^\P\),
			
			\vspace{5pt}
			
			\item[({\it b}2)] \(\Upsilon(f \, \mathbb{I}_{A^{[3]}})(\xi) \lesssim \mathscr{C}(T) \Psi^{-1}(\lambda \phi(\mu(A))) \|f\|_{\Phi, \phi, A^{[3]}}\),
			
			\vspace{5pt}
			
			\item[({\it b}3)] \(\|\mathscr{T}(f\,\mathbb{I}_{\widetilde{B}^ \P})(x) - \mathscr{T}(f\,\mathbb{I}_{B^{[3]}})(x)\|_{\mathbb{B}} \lesssim \mathscr{C}(T) \Psi^{-1}(\lambda \phi(\mu(A))) \|f\|_{ \Phi,\phi,A^{[3]}}, \quad\text{for all} \; x \in B^{[2]}\).
		\end{itemize*} 
	\end{itemize}
\end{lemma}

\begin{proof}
	We first prove part ({\it a}). For this, we need to use Lemma \ref{Lm 3.6}, which, however, applies under the condition that \(\mu(F(B_0)) \leq \lambda^{-1} \mu(B_0)\). Consequently, for the number \(\lambda \mathscr{C}_0^3 \geq 3 \mathscr{C}_0^7 > 1\) chosen above, we define the following set
	\begin{align}\label{3.25}
		F(B_0) := \left\{x \in B_0^{[3]}: \Upsilon(f \, \mathbb{I}_{B_0^{[3]}})(x) > \Psi^{-1}\left( \lambda \mathscr{C}_0^3 \psi \left(\mu(B_0^{[3]}) \right) \right) \|\Upsilon\| \|f\|_{\Phi,\phi,B_0^{[3]}} \right\}.
	\end{align}
	Then we obtain
	\begin{align*}
		F(B_0) \leq \frac{\mu(B_0^{[3]})}{\lambda \mathscr{C}_0^3} \leq \frac{\mu(B_0)}{\lambda}.
	\end{align*}
	Applying Lemma \ref{Lm 3.6} with \(A = B_0\) and \(F = F(B_0)\), we obtain a family \(\mathcal{F}(B_0) \subset \mathfrak{B}\) such that
	\begin{align}
		\label{3.26} &F(B_0) \cap B_0^\P \cap B \ne \varnothing, \quad \forall \, B \in \mathcal{F}(B_0),  \\
		\label{3.27} &F(B_0) \cap B_0^\P \subset \bigcup_{B \in \mathcal{F}(B_0)} B \; a.s.,  \\
		\label{3.28} &\mu\left(\bigcup_{B \in \mathcal{F}(B_0)} B^\P\right) \leq 3\lambda^{-1}\mathscr{C}_0^2 \mu(B_0),
	\end{align}
	and for each \(B \in \mathcal{F}(B_0)\) one can find a ball \(\widetilde{B} \in \mathfrak{B}\) such that  
	\begin{align}\label{3.29}
		\widetilde{B} \not\subset F(B_0), \quad B^{[2]} \subset \widetilde{B} \subset B_0^\P, \quad \text{and} \quad \Delta_T(B^{[2]},\widetilde{B}) \lesssim \mathscr{C}(T).
	\end{align}
	We note that (\ref{3.27}) implies  \(\mu\left(F(B_0) \cap B_0^\P \setminus \bigcup_{B \in \mathcal{F}(B_0)} B\right) = 0\).
	This shows that
	\begin{align*}
		F(B_0) \cap B_0^\P \setminus \bigcup_{B \in \mathcal{F}(B_0)} B = F(B_0) \cap \left(B_0^\P \setminus \bigcup_{B \in \mathcal{F}(B_0)} B\right) = \varnothing.
	\end{align*}
	Therefore, for almost every \(x \in B_0^\P \setminus \bigcup_{B \in \mathcal{F}(B_0)} B\), we deduce from (\ref{3.23}) and (\ref{3.25}) that
	\begin{align*}
		\Upsilon(f \, \mathbb{I}_{B_0^{[3]}})(x) \lesssim  \mathscr{C} (T)\Psi^{-1}\left( \lambda \psi \left(\mu(B_0) \right) \right) \|f\|_{\Phi,\phi,B_0^{[3]}}.
	\end{align*}
	This completes the proof of part ({\it a}) for the case \(A = B_0\).
	
	Next, we demonstrate that part ({\it b}) holds. The second estimate in (\ref{3.29}) implies that part ({\it b}1) is satisfied for \(A = B_0\). The first estimate in (\ref{3.29}) indicates that there exists a point \(\xi \in \widetilde{B} \setminus F(B_0)\), and consequently,
	\begin{align}\label{3.30}
		\Upsilon(f \, \mathbb{I}_{B_0^{[3]}})(\xi) \lesssim  \mathscr{C} (T)\Psi^{-1}\left( \lambda \psi \left(\mu(B_0) \right) \right) \|f\|_{\Phi,\phi,B_0^{[3]}}.
	\end{align}
	This shows that part ({\it b}2) holds for \(A = B_0\). Furthermore, using (\ref{3.29}) and (\ref{3.30}), we obtain for any \(x \in B^{[2]}\) that
	\begin{align*}
		\|\mathscr{T}(f\,\mathbb{I}_{\widetilde{B}^ \P}) & (x) - \mathscr{T} (f\,\mathbb{I}_{B^{[3]}})(x)\|_{\mathbb{B}} = \|\mathscr{T}(f \, \mathbb{I}_{B_0^{[3]} \cap \widetilde{B}^\P})(x) - \mathscr{T}(f \, \mathbb{I}_{B_0^{[3]} \cap B^{[3]}})(x)\|_{\mathbb{B}}  \\
		&\leq \Delta_T(B^{[2]},\widetilde{B}) \|f \, \mathbb{I}_{B_0^{[3]}}\|_{\Phi,\phi,\widetilde{B}^\P} \leq \mathscr{C}(T) \mathcal{M}_{\mathfrak{B}, \Phi, \phi} (f \, \mathbb{I}_{B_0^{[3]}})  \\
		&\leq \Upsilon(f \, \mathbb{I}_{B_0^{[3]}})(\xi) \lesssim  \mathscr{C} (T)\Psi^{-1}\left( \lambda \psi \left(\mu(B_0) \right) \right) \|f\|_{\Phi,\phi,B_0^{[3]}}.
	\end{align*}
	This implies that part ({\it b}3) is also satisfied for \(A = B_0\).
	
	Finally, we show that such a family \(\mathcal{G} \subset \mathfrak{B}\) exists. Replacing the ball \(B_0 \in \mathcal{F}(B_0)\) in the arguments above with an arbitrary \(A \in \mathcal{F}(B_0)\), and repeating the procedures for part ({\it a}) and part ({\it b}), we obtain a family \(\mathcal{F}(A)\) that satisfies part ({\it a}1) through part ({\it b}3). We formalize this iterative process by defining
	\begin{align}\label{3.31}
		\mathcal{F}_0(B_0) := \{B_0\}, \quad \mathcal{F}_1(B_0) := \mathcal{F}(B_0), \quad \text{and} \quad \mathcal{F}_{k+1}(B_0) := \bigcup_{B \in \mathcal{F}_k(B_0)} \mathcal{F}(B), \;\; k \geq 1.
	\end{align}
	We now set
	\begin{align}\label{3.32}
		\mathcal{G} = \mathcal{G}(B_0) := \bigcup_{k \geq 0} \mathcal{F}_k (B_0).
	\end{align}
	Given that every family \(\mathcal{F}_k (B_0)\) (where \(k \geq 1\)) satisfies part ({\it a}1) through part ({\it b}3), we conclude that the resulting family \(\mathcal{G} = \mathcal{G}(B_0) \subset \mathfrak{B}\) also satisfies parts ({\it a}1) through ({\it b}3). Hence, the lemma is proved.
\end{proof}

For the convenience of subsequent exposition, we introduce the following notation:
\begin{itemize*}
	\item The set \(\mathcal{F}(A)\) in Lemma \ref{Lm 3.8} denotes the stopping time balls associated with \(A\).
	
	\item The family \(\mathcal{F}_k(A)\) in (\ref{3.31}) denotes the collection of the \(k\)-th generation stopping time balls associated with \(A\).
	
	\item The set \(\mathcal{G}\) in (\ref{3.32}) denotes the collection of all generation balls.
\end{itemize*}
Given \(A, B \in \mathcal{G}\) and \(k \geq 0\), if \(B \in \mathcal{F}_k(A)\), we write \(\mathcal{F}^k(B) = A\). In which case we say that \(A\) is the \(k\)-th generation ancestor of \(B\), or that \(B\) is one of the \(k\)-th generation descendants of \(A\). Repeated application of Lemma \ref{Lm 3.8} part ({\it a}2) yields
\begin{align}\label{3.33}
	\mu\left(\bigcup_{B \in \mathcal{F}_k(A)}B \right) \leq (3 \lambda^{-1} \mathscr{C}_0^2)^k \mu(A), \quad \text{for all} \; k \geq 1.
\end{align}

\begin{lemma}\label{Lm 3.9}
	Let \((X,\,\mathfrak{M},\,\mu)\) be a measure space endowed with a ball-basis \(\mathfrak{B}\) and \(\lambda > 3\mathscr{C}_0^6\). Fix \(B_0 \in \mathfrak{B}\) satisfying (\ref{3.24}). Then the following properties hold:
	\begin{itemize}[leftmargin=1.5em]
		\item[(1)] When \(\lambda\) is sufficiently large, there exist two types of sparse families within the collection \(\mathcal{G}\) from Lemma \ref{Lm 3.8}, namely: a \(2^{-1}\)-sparse family and a \(2^{-1}\mathscr{C}_0^{-3}\)-sparse family.
		
		\item[(2)] For any ball \(B\) belonging to the \(2^{-1}\)-sparse family, there exists a sequence \(\{B_j\}_{j=0}^k\) that is a subset of the \(2^{-1}\)-sparse family, such that \(B_{j+1} \in \mathcal{F}(B_j)\) for \(j=0,1,...,k-1\) and \(B_k = B\). There exist balls \(\widetilde{B}_j\) and points \(\xi_j \in \widetilde{B}_{j+1}\) such that the following estimates hold:
		\begin{align}
			\label{3.34} &B_{j+1}^{[2]} \subset \widetilde{B}_{j+1} \subset B_j^\P,  \\
			\label{3.35} &\Upsilon(f \, \mathbb{I}_{B_j^{[3]}})(\xi_j) \lesssim \mathscr{C}(T) \Psi^{-1}(\lambda \phi(\mu(B_j))) \|f\|_{\Phi, \phi, B_j^{[3]}},  \\
			\label{3.36} &\left\|\mathscr{T}(f\,\mathbb{I}_{\widetilde{B}_{j+1}^ \P})(x) - \mathscr{T}(f\, \mathbb{I}_{B_{j+1}^{[3]}})(x)\right \|_{\mathbb{B}} \lesssim \mathscr{C}(T) \Psi^{-1}(\lambda \phi(\mu(B_j))) \|f\|_{ \Phi,\phi,B_j^{[3]}}, \;\; \text{for all} \; x \in B_0.
		\end{align}
	\end{itemize}
\end{lemma}

\begin{proof}
	The collection \(\mathcal{G}\) from Lemma \ref{Lm 3.8} implies that parts ({\it a}1) through ({\it b}3) hold. We shall prove part (1) by partitioning a certain subset of \(\mathcal{G}\) in the following way.
	
	Let \(B \in \mathfrak{B}\) and \(\mathscr{R} = \mathscr{C}_0^2\). We denote \(r(B) = [\log_\mathscr{R}\mu(B)]\), where \([x]\) represents the greatest integer not exceeding \(x\). We then define the following collections:
	\begin{align*}
		\mathcal{G}^{k_0} := \{B_0\}, \; \mathcal{G}^k := \{B \in \mathcal{G}: r(B) = k\} = \{B \in \mathcal{G}: \mathscr{R}^k \leq \mu(B) < \mathscr{R}^{k+1}\}, \; k < k_0 = r(B_0).
	\end{align*}
	We shall now search for sparse families within \(\bigcup_{k \leq k_0}\mathcal{G}^k \subset \mathcal{G}\). First, we fix a ball \(B' \in \mathcal{G}\). We need to remove certain balls from \(\mathcal{G}\). The set of removed balls is denoted by \(\mathfrak{B}_{R} \subset \mathcal{G}\). Here, \(\mathfrak{B}_R\) consists of those balls \(B \in \mathcal{G}\) that satisfy
	\begin{align}\label{3.37}
		B^{[2]} \cap B' \ne \varnothing \quad \text{and} \quad r(B) 
		+ 2 \leq r(B') \leq r(\mathcal{F}^1(B)) - 2.
	\end{align}
	Subsequently, we set
	\begin{align*}
		\mathcal{H}^{k_0} := \mathcal{G}^{k_0} =\{B_0\} \quad \text{and} \quad \mathcal{H}^k := \mathcal{G}^k \setminus \bigcup_{B \in \mathfrak{B}_R} \mathcal{G}(B), \quad \text{for} \; k < k_0.
	\end{align*}
	Note that by Lemma \ref{Lm 3.8} part ({\it b}1), it is evident that \(\bigcup_{B \in \mathcal{H}^k}B \subset B_0^\P\) for all \(k \leq k_0\). According to Lemma \ref{Lm 3.1} part (a), there exists a pairwise disjoint subfamily \(\mathcal{E}^k \subset \mathcal{H}^k\) such that
	\begin{align}\label{3.38}
		\bigcup_{B \in \mathcal{H}^k}B \subset \bigcup_{B \in \mathcal{E}^k} B^\P.
	\end{align}
	We then denote
	\begin{align}\label{3.39}
		\mathcal{E} := \bigcup_{k \leq k_0}\mathcal{E}^k \subset \bigcup_{k \leq k_0}\mathcal{H}^k \subset \bigcup_{k \leq k_0}\mathcal{G}^k \subset \mathcal{G}.
	\end{align}
	The collection \(\mathcal{E}\) can be partitioned into
	\begin{align}\label{3.40}
		\mathcal{E}_1 := \{B \in \mathcal{E}: r(B) \text{ is odd}\} \quad \text{and} \quad \mathcal{E}_2 := \{B \in \mathcal{E}: r(B) \text{ is even}\}.
	\end{align}
	It is not difficult to see that this divides \(\mathcal{E}\) into two parts satisfying \(\mathcal{E}_1 \cap \mathcal{E}_2 = \varnothing\) and \(\mathcal{E} = \mathcal{E}_1 \cup \mathcal{E}_2\). From Lemma \ref{Lm 3.10}, we know that both \(\mathcal{E}_1\) and \(\mathcal{E}_2\) are \(1/2\)-sparse. We then construct the following sets:
	\begin{align}\label{3.41}
		\mathcal{S}_i := \{B^{[3]}: B \in \mathcal{E}_i\} \subset \mathcal{G}, \quad \text{for} \; i=1,2.
	\end{align}
	One can readily observe that both \(\mathcal{S}_1\) and \(\mathcal{S}_2\) are \(2^{-1} \mathscr{C}_0^{-3}\)-sparse. Thus, we have found two types of sparse families within \(\mathcal{G}\), which establishes the validity of part (1).
	
	Next, we prove part (2). Let \(B \in \mathcal{E}\). According to Lemma \ref{Lm 3.11}:
	\begin{align}\label{3.42}
		\Upsilon(f \, \mathbb{I}_{B^{[3]}})(x) \lesssim \mathscr{C}(T) \Psi^{-1}(\lambda \phi(\mu(B))) \|f\|_{\Phi, \phi, B^{[3]}}, \quad x \in \left(B^\P \setminus \bigcup_{G \in \mathcal{E}: r(G) < r(B)} G^\P\right) \setminus F_0,
	\end{align}
	where \(F_0\) is a set of measure zero. We also denote
	\begin{align*}
		\mu(F_1) := \mu\left(\bigcap_{k \leq k_0} \bigcup_{G \in \mathcal{E}: r(G) < k} G^\P\right) = 0.
	\end{align*}
	Fix a point \(x \in B_0 \setminus (F_0 \cup F_1)\). Then there exists a ball \(B \in \mathcal{E}\) such that \(x \in B^\P \setminus \bigcup_{G \in \mathcal{E}: r(G) < r(B)}G^\P\). Considering the construction of \(\mathcal{G}\) and the fact that \(B \in \mathcal{E} \subset \mathcal{G}\), we can find a sequence \(\{B_j\}_{j=0}^k \subset \mathcal{E}\) satisfying the following conditions:
	\begin{itemize*}
		\item \(B_{j+1} \in \mathcal{F}(B_j)\) for \(j = 0, 1, ..., k-1\),
		\item \(B_k = B\).
	\end{itemize*}
	Consequently, from Lemma \ref{Lm 3.8} parts ({\it b}1) through ({\it b}3), we obtain the existence of balls \(\widetilde{B}_j \in \mathfrak{B}\) and points \(\xi_j \in \widetilde{B}_{j+1}\) such that the following estimates hold:
	\begin{align*}
		&B_{j+1}^{[2]} \subset \widetilde{B}_{j+1} \subset B_j^\P,  \\
		&\Upsilon(f \, \mathbb{I}_{B_j^{[3]}})(\xi_j) \lesssim \mathscr{C}(T) \Psi^{-1}(\lambda \phi(\mu(B_j))) \|f\|_{\Phi, \phi, B_j^{[3]}},  \\
		&\left\|\mathscr{T}(f\,\mathbb{I}_{\widetilde{B}_{j+1}^ \P})(t) - \mathscr{T}(f\, \mathbb{I}_{B_{j+1}^{[3]}})(t)\right \|_{\mathbb{B}} \lesssim \mathscr{C}(T) \Psi^{-1}(\lambda \phi(\mu(B_j))) \|f\|_{ \Phi,\phi,B_j^{[3]}}, \;\; \text{for all} \; t \in B_{j+1}^{[2]}.
	\end{align*}
	The first two estimates above imply (\ref{3.34}) and (\ref{3.35}) respectively. We note that \(x \in B^\P = B_k^\P \subset B_{j+1}^{{2}}\). Therefore, by taking \(t = x\), we establish the validity of (\ref{3.36}). Hence, part (2) holds, which completes the proof of the lemma.
\end{proof}

\subsection{Sparse sets and further properties of \(\Psi\)-BOOs}\label{Section 3.3}

In this subsection, we aim to provide some auxiliary lemmas required for Lemma \ref{Lm 3.9}.

\begin{lemma}\label{Lm 3.10}
	When \(\lambda > 3\mathscr{C}_0^6\), the collections \(\mathcal{E}_1\) and \(\mathcal{E}_2\) defined in (\ref{3.40}) are \(1/2\)-sparse.
\end{lemma}

\begin{proof}
	For notational convenience, and only within the proof of this lemma, we denote by \(\mathcal{E}\) either \(\mathcal{E}_1\) or \(\mathcal{E}_2\). Therefore, it suffices to show that \(\mathcal{E}\) is \(1/2\)-sparse. That is, we need to prove that for every \(Q \in \mathcal{E}\), there exists a subset \(E_Q\) of \(Q\) satisfying
	\begin{align}\label{3.43}
		\{E_Q: Q \in \mathcal{E}\} \text{ is pairwise disjoint}, \quad \text{and} \quad \mu(E_Q) \geq \frac{1}{2} \mu(Q).
	\end{align}
	
	We now proceed to establish (\ref{3.43}). For any \(Q \in \mathcal{E}\), we define
	\begin{align}\label{3.44}
		E_Q := Q \setminus \bigcup_{H \in \mathcal{E}: H \cap Q \neq \varnothing, r(H) \leq r(Q) - 2} H.
	\end{align}
	For any \(Q_1,Q_2 \in \mathcal{E}\), it necessarily follows that
	\begin{align}\label{3.45}
		E_{Q_1} \cap E_{Q_2} = \varnothing.
	\end{align}
	At this point, we fix the aforementioned \(Q_1\) and \(Q_2\), which satisfy \(Q_1 \cap Q_2 \neq \varnothing\) and \(r(Q_2) \leq r(Q_1) - 2\). Then we have \(Q_2^{[2]} \cap Q_1 \neq \varnothing\), and in combination with (\ref{3.37}) we obtain
	\begin{align}\label{3.46}
		r(Q_2) + 2 \leq r(Q_1) \leq r(\mathcal{F}^1(Q_2)) - 2.
	\end{align}
	Lemma \ref{Lm 3.8} part ({\it a}2) implies that
	\begin{align}\label{3.47}
		r(\mathcal{F}^1(Q_2)) = [\log_\mathscr{R}\mu(\mathcal{F}^1(Q_2))] \leq [\log_\mathscr{R}3 \lambda^{-1} \mathscr{C}_0^2 \mu(Q_2)] = -2 + r(Q_2).
	\end{align}
	From (\ref{3.46}) and (\ref{3.47}) we get
	\begin{align}\label{3.48}
		r(\mathcal{F}^1(Q_2)) + 2 \leq r(Q_1) \leq r(\mathcal{F}^1(Q_2)) - 2.
	\end{align}
	We denote by \(\mathfrak{P}_{Q_1}\) the collection of all sets satisfying (\ref{3.48}):
	\begin{align*}
		\mathfrak{P}_{Q_1} := \{ P \in \mathcal{E} : P^{[2]} \cap Q_1 \neq \varnothing, |r(P) - r(Q_1)| \leq 2 \}.
	\end{align*}
	Thus, \(\mathcal{F}^1(Q_2) \in \mathfrak{P}_{Q_1}\). According to Lemma \ref{Lm more} part (4), we have \(\#\mathfrak{P}_{Q_1} \lesssim 1\). Furthermore, for any \(P \in \mathfrak{P}_{Q_1}\), it holds that \(|\log_\mathscr{R} \mu(P) - \log_\mathscr{R} \mu(Q)| \leq 3\), which indicates
	\begin{align}\label{3.49}
		\mathscr{C}_0^{-6} \mu(Q) = \mathscr{R}^{-3} \mu(Q) \leq \mu(P) \leq \mathscr{R}^3 \mu(Q) = \mathscr{C}_0^6 \mu(Q).
	\end{align}
	Consequently, we obtain
	\begin{align}\label{3.50}
		\mu(Q \setminus E_Q & ) = \mu\left(\bigcup_{H \in \mathcal{E}: H \cap Q \neq \varnothing, r(H) \leq r(Q) - 2} H \right) \leq \mu\left( \bigcup_{P \in \mathfrak{P}_{Q_1}} \bigcup_{H \in \mathcal{G}(P)} H \right)  \nonumber \\
		&\leq \sum_{P \in \mathfrak{P}_{Q_1}} \sum_{k=1}^\infty \mu\left( \bigcup_{H \in \mathcal{F}_k(P)} H\right) \leq \sum_{P \in \mathfrak{P}_{Q_1}} \mu(P) \cdot \sum_{k=1}^\infty (3 \lambda^{-1} \mathscr{C}_0^2)^k  \nonumber \\
		&\lesssim \mu(P) \sum_{k=1}^\infty \lambda^{-k} \leq \frac{\mu(P) }{\lambda-1} \leq \frac{\mathscr{C}_0^6}{3\mathscr{C}_0^6-1} \mu(Q) \leq \frac{1}{2} \mu(Q).
	\end{align}
	In deriving (\ref{3.50}), we have used (\ref{3.33}) and (\ref{3.49}).
	
	Finally, we justify the validity of (\ref{3.45}). For \(Q_1,Q_2 \in \mathcal{E}\), one of the following two cases must hold:
	\begin{align*}
		Q_1 \cap Q_2 = \varnothing \quad \text{or} \quad Q_1 \cap Q_2 \ne \varnothing.
	\end{align*}
	\begin{itemize}[leftmargin=1.5em]
		\item If \(Q_1 \cap Q_2 = \varnothing\), it is evident that \(E_{Q_1} \cap E_{Q_2} = \varnothing\).
		
		\item If \(Q_1 \cap Q_2 \neq \varnothing\), since each family \(\mathcal{E}^k\) is pairwise disjoint, we must have \(r(Q_1) \ne r(Q_2)\). Then, assuming \(r(Q_2) \leq r(Q_1) - 2\), it follows from (\ref{3.44}) that \(E_{Q_1} \cap Q_2 = \varnothing\). Moreover, because \(E_{Q_2} \subset Q_2\), we conclude that \(E_{Q_1} \cap E_{Q_2} = \varnothing\).
	\end{itemize}
	Hence, (\ref{3.45}) holds. The results (\ref{3.45}) and (\ref{3.50}) together imply (\ref{3.43}). This completes the proof of this part.
\end{proof}

\begin{lemma}\label{Lm 3.11}
	For any \(B \in \mathcal{E}\), there exists a set \(F_0\) of measure zero such that the following estimate holds:
	\begin{align*}
		\Upsilon(f \, \mathbb{I}_{B^{[3]}})(x) \lesssim \mathscr{C}(T) \Psi^{-1}(\lambda \phi(\mu(B))) \|f\|_{\Phi, \phi, B^{[3]}}, \quad x \in \left(B^\P \setminus \bigcup_{G \in \mathcal{E}: r(G) < r(B)} G^\P\right) \setminus F_0.
	\end{align*}
\end{lemma}

\begin{proof}		
	Observe that we fix \(B \in \mathcal{E}\). If we assert that
	\begin{align}\label{3.51}
		\bigcup_{G \in \mathcal{F}(B)} G \subset \bigcup_{G \in \mathcal{E}: r(G) < r(B)} G^\P =: H.
	\end{align}
	then it follows that \(x \in B^\P \setminus \bigcup_{G \in \mathcal{E}: r(G) < r(B)} G^\P \subset  B^\P \setminus \bigcup_{G \in \mathcal{F}(B)} G\). Applying Lemma \ref{Lm 3.8} part ({\it a}3) then completes the proof of the lemma.
	
	We now demonstrate that (\ref{3.51}) holds. This is equivalent to
	\begin{align}\label{3.52}
		G \subset H \quad \text{for all} \; G \in \mathcal{F}(B) \setminus \mathcal{E}.
	\end{align}
	Fix \(G \in \mathcal{F}(B) \setminus \mathcal{E}\). It is clear that we have \(r(G) < r(B)\). According to (\ref{3.39}), the set \(G\) must satisfy one of the following conditions, namely:
	\begin{align*}
		\text{either} \quad G \in \bigcup_{k \leq k_0}(\mathcal{H}^k \setminus \mathcal{E}^k) \quad \text{or} \quad G \in \bigcup_{k \leq k_0}(\mathcal{G}^k \setminus \mathcal{H}^k).
	\end{align*}
	\begin{itemize}[leftmargin=1.5em]
		\item {\bf {\emph Case} 1} Suppose \(G \in \bigcup_{k \leq k_0}(\mathcal{H}^k \setminus \mathcal{E}^k)\). Then there exists a unique \(k \leq k_0\) and a ball \(\widetilde{B}_k \in \mathcal{H}^k\) such that, by virtue of (\ref{3.38}), we have
		\begin{align}\label{3.53}
			G = \widetilde{B}_k \subset \bigcup_{\widetilde{B}_k \in \mathcal{H}^k} \widetilde{B}_k \subset \bigcup_{B_k \in \mathcal{E}^k} B_k^\P \subset H.
		\end{align}
		Furthermore, for each \(B_k \in \mathcal{E}^k\), it follows that \(r(B_k) = k = r(G) < r(B)\).
		
		\item {\bf {\emph Case} 2} Suppose \(G \in \bigcup_{k \leq k_0}(\mathcal{G}^k \setminus \mathcal{H}^k)\). From the construction of \(\mathcal{H}^k\), it follows that \(\bigcup_{k \leq k_0}(\mathcal{G}^k \setminus \mathcal{H}^k) = (\bigcup_{k \leq k_0}\mathcal{G}^k) \cap (\bigcup_{B \in \mathfrak{B}_R}\mathcal{G}(B))\). Using (\ref{3.37}), there exists a ball \(\widetilde{B} \in \mathcal{G} \subset \mathcal{E}\) such that \(G^{[2]} \cap \widetilde{B} \ne \varnothing\) and
		\begin{align}\label{3.54}
			r(G) + 2 \leq r(\widetilde{B}) \leq r(\mathcal{F}^1(G)) - 2 = r(B) - 2.
		\end{align}
		We note that (\ref{3.54}) implies
		\begin{align*}
			\log_\mathscr{R}\mu(G) -1 +2 \leq [\log_\mathscr{R}\mu(G)] + 2 = r(G) + 2 \leq r(\widetilde{B}) = [\log_\mathscr{R}\mu (\widetilde{B})] \leq \log_\mathscr{R}\mu(\widetilde{B}),
		\end{align*}
	\end{itemize}
	which consequently gives \(\mathscr{R}\mu(G) \leq \mu(\widetilde{B})\), a relation equivalent to \(\mu(G^{[2]}) \leq \mathscr{C}_0^2\mu(G) \leq \mu(\widetilde{B})\). Applying ball-hull control property then yields
	\begin{align}\label{3.55}
		G \subset G^{[2]} \subset \widetilde{B}^\P \subset H.
	\end{align}
	The relations (\ref{3.53}) and (\ref{3.55}) together demonstrate that (\ref{3.51}) holds, which completes the proof of the lemma.
\end{proof}

\subsection{Proof of Theorem \ref{main}}\label{Section 3.4}

In this subsection, we aim to prove Theorem \ref{main}.

\begin{proof}[Proof of Theorem \ref{main}]
	For any \(B_0 \in \mathfrak{B}\) satisfying (\ref{3.24}), by the arbitrariness of \(B_0\), inequality (\ref{ine main}) is equivalent to 
	\begin{align}\label{3.56}
		\|\mathscr{T}(f \, \mathbb{I}_{B_0^{[3]}})(x)\|_\mathbb{B} \lesssim \mathscr{C}(T) \cdot \left[\mathcal{A}_{\mathcal{S}_1, \Phi,\phi} f(x) + \mathcal{A}_{\mathcal{S}_2, \Phi,\phi} f(x)\right].
	\end{align}
	Here, both \(\mathcal{S}_1\) and \(\mathcal{S}_2\) are \(\frac{1}{2 \mathscr{C}_0^3}\)-sparse, and \(\mathcal{S}_1 \cup \mathcal{S}_2 = \mathcal{S}\).  
	
	Next, we prove (\ref{3.56}). For the above \(B_0\) and for \(\lambda > 3\mathscr{C}_0^6\), according to Lemma \ref{Lm 3.9}, we obtain the following conclusions:  
	\begin{itemize*}
		\item[(1)] There exist two \(\frac{1}{2 \mathscr{C}_0^3}\)-sparse families \(\mathcal{S}_1\) and \(\mathcal{S}_2\), which satisfy \(\mathcal{S}_1 \cup \mathcal{S}_2 = \mathcal{S}\),  
		
		\item[(2)] One can find a sequence \(\{B_j\}_{j=0}^k\) such that (\ref{3.34}) through (\ref{3.36}) hold, where \(B_k = B\) and \(j = 0, 1, ..., k-1\).  
		
	\end{itemize*}
	Then, we claim that for the above sequence \(\{B_j\} \subset \mathcal{S}\), the following inequality holds:  
	\begin{align}\label{3.57}
		\|\mathscr{T}(f \, \mathbb{I}_{B_j^{[3]}})(x) - \mathscr{T}(f \, \mathbb{I}_{B_{j+1}^{[3]}})(x)\|_\mathbb{B} \lesssim \mathscr{C}(T) \|f\|_{\Phi, \phi, B_j^{[3]}}.
	\end{align}
	Furthermore, combining (\ref{3.42}) with the condition \((\Phi, \phi) \in E_{mb}\) yields
	\begin{align}\label{3.58}
		\|\mathscr{T}(f \, \mathbb{I}_{B^{[3]}})(x)\|_\mathbb{B} \leq \Upsilon (f \, \mathbb{I}_{B^{[3]}})(x) \leq \mathscr{C}(T) \|f\|_{\Phi, \phi, B^{[3]}}.
	\end{align}
	Note that \(B_k = B\). Then, (\ref{3.57}) and (\ref{3.58}) imply
	\begin{align*}
		\|\mathscr{T}(f \, \mathbb{I}_{B_0^{[3]}})&(x)\|_\mathbb{B} = \sum_{j=0}^{k-1} \|\mathscr{T}(f \, \mathbb{I}_{B_j^{[3]}})(x) - \mathscr{T}(f \, \mathbb{I}_{B_{j+1}^{[3]}})(x)\|_\mathbb{B} + \|\mathscr{T}(f \, \mathbb{I}_{B_k^{[3]}})(x)\|_\mathbb{B} \nonumber \\
		&\lesssim \mathscr{C}(T) \sum_{j=0}^k \|f\|_{\Phi, \phi, B_j^{[3]}} \leq \mathscr{C}(T) \sum_{B \in \mathcal{S}_1 \cup \mathcal{S}_2} \|f\|_{\Phi, \phi, B^{[3]}} \mathbb{I}_{B^{[3]}} (x)  \nonumber \\
		&\leq \mathscr{C}(T) \left[\mathcal{A}_{\mathcal{S}_1, \Phi,\phi} f(x) + \mathcal{A}_{\mathcal{S}_2, \Phi,\phi} f(x)\right].
	\end{align*}
	This shows that (\ref{3.56}) holds, thereby completing the proof of the theorem.
\end{proof}

\begin{remark}
	Here we explain why (\ref{3.57}) holds. Let us define
	
	\begin{align*}
		&\mathscr{Z}_1 = \left\|\left(\mathscr{T}(f \, \mathbb{I}_{B_j^{[3]}}) - \mathscr{T}(f \, \mathbb{I}_{\widetilde{B}_{j+1}^\P})\right)(x) - \left(\mathscr{T}(f \, \mathbb{I}_{B_j^{[3]}}) - \mathscr{T}(f \, \mathbb{I}_{\widetilde{B}_{j+1}^\P})\right)(\xi_j)\right\|_\mathbb{B},  \\
		&\mathscr{Z}_2 = \left\|\mathscr{T}(f \, \mathbb{I}_{B_j^{[3]}}) (\xi_j) - \mathscr{T}(f \, \mathbb{I}_{\widetilde{B}_{j+1}^\P}) (\xi_j) \right\|_\mathbb{B},  \\
		&\mathscr{Z}_3 = \left\|\mathscr{T}(f \, \mathbb{I}_{\widetilde{B}_{j+1}^\P})(x) -  \mathscr{T}(f \, \mathbb{I}_{B_{j+1}^{[3]}})(x) \right\|_\mathbb{B}.
	\end{align*}
	Then,
	\begin{align}\label{3.59}
		\|\mathscr{T}(f \, \mathbb{I}_{B_j^{[3]}})(x) - \mathscr{T}(f \, \mathbb{I}_{B_{j+1}^{[3]}})(x)\|_\mathbb{B} \leq \mathscr{Z}_1 + \mathscr{Z}_2 + \mathscr{Z}_3.
	\end{align}
	In the following, we estimate \(\mathscr{Z}_1\), \(\mathscr{Z}_2\), and \(\mathscr{Z}_3\) separately.
	\begin{itemize*}
		\item The relation (\ref{3.34}) shows that \(\widetilde{B}_{j+1} \subset \widetilde{B}_{j+1}^\P \subset B_j^{[2]} \subset B_j^{[3]}\). Then, for \(\xi_j \in \widetilde{B}_{j+1}\), substituting \(f \, \mathbb{I}_{B_j^{[3]}}\) and \(\widetilde{B}_{j+1}\) for \(f\) and \(B\) in condition (\(\Psi\)-BOO-II), we obtain
		\begin{align}\label{3.60}
			\mathscr{Z}_1 \leq \mathscr{C}_2(T) & \, \langle \|f \, \mathbb{I}_{B_j^{[3]}}\|\rangle_{\Phi,\phi,\widetilde{B}_{j+1}} \leq \mathscr{C}(T) \mathcal{M}_{\widetilde{B}_{j+1},\Phi,\phi}(f \, \mathbb{I}_{B_j^{[3]}})(\xi_j)  \nonumber \\
			&\leq \Upsilon (f \, \mathbb{I}_{B_j^{[3]}})(\xi_j) \lesssim \mathscr{C}(T) \|f\|_{\Phi, \phi, B_j^{[3]}}.
		\end{align}
		The last estimate uses (\ref{3.35}) and the fact that \((\Phi, \phi) \in E_{mb}\).
		
		\item The relation (\ref{3.34}) also shows that \(\widetilde{B}_{j+1} \subset \widetilde{B}_{j+1}^\P \subset B_j^{[3]}\). Then, for \(\xi_j \in \widetilde{B}_{j+1}\), substituting \(f \, \mathbb{I}_{B_j^{[3]}}\) and \(\widetilde{B}_{j+1}\) for \(f\) and \(B\) in the definition of \(T^*\), we have
		\begin{align}\label{3.61}
			\mathscr{Z}_2 = \left\|\mathscr{T}(f \, \mathbb{I}_{B_j^{[3]}}) (\xi_j) - \mathscr{T}(f \, \mathbb{I}_{B_j^{[3]}} \, \mathbb{I}_{\widetilde{B}_{j+1}^\P}) (\xi_j) \right\|_\mathbb{B} \leq  T^*(f & \, \mathbb{I}_{B_j^{[3]}})(\xi_j)  \nonumber \\
			\leq \Upsilon (f \, \mathbb{I}_{B_j^{[3]}})(\xi_j) \lesssim \mathscr{C}(T) \|f\|_{\Phi, \phi, B_j^{[3]}}.&
		\end{align}
		Again, the last estimate uses (\ref{3.35}) and the fact that \((\Phi, \phi) \in E_{mb}\).
		
		\item Using (\ref{3.36}) and the fact that \((\Phi, \phi) \in E_{mb}\), we obtain
		\begin{align}\label{3.62}
			\mathscr{Z}_3 = \left\|\mathscr{T}(f \, \mathbb{I}_{\widetilde{B}_{j+1}^\P})(x) -  \mathscr{T}(f \, \mathbb{I}_{B_{j+1}^{[3]}})(x) \right\|_\mathbb{B} \lesssim \mathscr{C}(T) \|f\|_{\Phi, \phi, B_j^{[3]}}.
		\end{align}
	\end{itemize*}
	Therefore, combining (\ref{3.59}), (\ref{3.60}), (\ref{3.61}), and (\ref{3.62}) implies that (\ref{3.57}) holds.
\end{remark}

\begin{appendices}

\section{Two important classes of functions}\label{Section A1}~~

{\bf (1) \(\mathscr{Y}\)-class functions}

We term a function \(\Phi : [0, +\infty] \to [0, +\infty]\) a {\sffamily Young function} provided that it meets the following conditions:
\begin{itemize}
	\item \(\Phi\) is convex.
	\item \(\Phi\) is left-continuous.
	\item \(\displaystyle \lim\limits_{t \to 0^+}\Phi(t) = \Phi(0) = 0\) and \(\displaystyle \lim\limits_{t\to +\infty} \Phi(t) = \Phi(+ \infty) = +\infty\).
\end{itemize}

Let \(\mathscr{Y}\) denote the collection of all Young functions, that is,
\begin{align}\label{Y function}
	\mathscr{Y}=\left\{\Phi:\,[0,\,+\infty]\to[0,\,+\infty]\,\big| \, 0 < \Phi(t) < +\infty, \, 0 < r < +\infty \right\}.
\end{align}

\begin{remark}
	If \(\Phi\in\mathscr{Y}\), then \(\Phi\) is absolutely continuous on any closed subinterval of \([0,\,+\infty)\) and \(\Phi:\,[0,\, +\infty) \to [0,\,+\infty)\) is a bijection.
\end{remark}

For a Young function \(\Phi\) and for \(0 \leq t \leq +\infty\), we define  
\begin{align*}
	\Phi^{-1}(t)=\inf\{s \geq 0:\,\Phi(s)>t\} \qquad (\inf \varnothing = +\infty).
\end{align*}

\begin{remark}
	Regarding the inverse function \(\Phi^{-1}\), we have the following observations.
	\begin{itemize}
		\item \(\Phi\in\mathscr{Y}\) implies that \(\Phi^{-1}\) is the inverse of \(\Phi\) in the usual sense.
		\item For every \(0 \leq t < +\infty\), the following double inequality holds:
		\[
		\Phi\left(\Phi^{-1}(t)\right) \leq t \leq \Phi^{-1}\left(\Phi(t)\right).
		\]
	\end{itemize}
\end{remark}

For a Young function \(\Phi\), one defines its {\sffamily complementary function} \(\widetilde{\Phi}\) via the formula
\begin{align*}
	\widetilde{\Phi}(s)= 
	\left\{ 
	\begin{array}{ll}
		\sup\limits_{t \in [0,\,+\infty)}\{st-\Phi(t)\}, ~~ & s \in [0,\,+\infty), \\
		+\infty,  & s=+\infty. \\
	\end{array}
	\right.
\end{align*}
It is clear that \(\widetilde{\Phi}\) is also a Young function and satisfies \(\widetilde{\widetilde{\Phi}}=\Phi\).

\begin{example}
	The following are examples concerning the complementary function \(\widetilde{\Phi}\).
	
	\begin{itemize}
		\item[\(\blacksquare\)] \(\Phi(t)=(1+t)\log(1+t)-t \; \Longrightarrow \; \widetilde{\Phi}(t)=e^t-t-1\).
		
		\vspace{6pt}
		
		\item[\(\blacksquare\)] \(\displaystyle\Phi(t)=\frac{1}{2}t^2\; \Longrightarrow\; \widetilde{\Phi}(t)=\frac{1}{2}t^2\).
		
		\vspace{6pt}
		
		\item[\(\blacksquare\)] \(\Phi(t)=\max\{0,t-1\} \; \Longrightarrow \; \widetilde{\Phi}(t)= 
		\left\{ 
		\begin{array}{ll}
			t, ~~ & 0 \leq t \leq 1, \\
			+\infty,  & t>1. \\
		\end{array}
		\right.\)
	\end{itemize}
\end{example}

{\bf (2) \(\mathscr{G}\)-class functions}

Regarding the concepts of almost increasing (abbreviated as \(a.\,ic.\)), almost decreasing (abbreviated as \(a.\,dc.\)), and doubling condition.

\begin{itemize}
	\item A function \(\varphi:\,(0,\,+\infty) \to (0,\,+\infty)\) is called {\sffamily almost increasing} (resp. {\sffamily almost decreasing}) if there exists a constant \(C_{ic} > 0\) (resp. \(C_{dc} > 0\)) such that
	\[
	\varphi(r) \leq C_{ic} \varphi(s) \qquad
	\left(\text{resp. } \varphi(r)\geq C_{dc} \varphi(s) \right) \qquad 
	\text{for } \; r \leq s.
	\]
	
	\item A function \(\varphi:\,(0,\,+\infty) \to(0,\,+\infty)\) is said to satisfy the {\sffamily doubling condition} if there exists a constant \(C>0\) such that
	\[
	C^{-1} \leq \frac{\varphi(r)}{\varphi(s)} \leq C \qquad 
	\text{for } \; \frac{1}{2} \leq \frac{r}{s} \leq 2.
	\]
	
	\item For functions \(\varphi,\,\vartheta: \,(0,\,+\infty)\to(0,\,+\infty)\) we write \(\varphi\sim\vartheta\) if there exists a constant \(C > 0\) such that
	\[
	C^{-1} \vartheta(r) \leq \varphi(r) \leq C \vartheta(r) \qquad 
	\text{for all } \; r > 0.
	\]
\end{itemize}

Let \(\mathscr{G}\) denote the set of all functions \(\varphi:\, (0,\, +\infty) \to (0,\,+\infty)\); that is,
\begin{align}\label{G function}
	\mathscr{G} = \left\{ \varphi:\,\varphi(r) \,\, a.\,dc., \, r\varphi(r) \,\, a.\,ic.\right\}.
\end{align}

\begin{remark}
	For \(\varphi \in \mathscr{G}\), we have the following assertions.
	\begin{itemize}
		\item[\(\blacksquare\)] \(\varphi\in\mathscr{G}\) implies that \(\varphi\) satisfies the doubling condition.
		
		\item[\(\blacksquare\)] If there exists \(\vartheta:\,(0,\,+\infty)\to(0,\,+\infty)\) and a function \(\varphi \in \mathscr{G}\) such that \(\varphi\sim\vartheta\), then \(\vartheta \in \mathscr{G}\).
	\end{itemize}
\end{remark}

\section{Orlicz-Morrey spaces and related function spaces}\label{Section A2}~~

Let \((X,\,\mathfrak{M},\,\mu)\) be a measure space endowed with a ball-basis \(\mathfrak{B}\). Given Young functions \(\Phi\in\mathscr{Y}\), \(\phi\in\mathscr{G}\) and a ball \(B\in\mathfrak{B}\), we introduce the following three types of Luxemburg norms.

\begin{itemize}
	\item Standard Luxemburg norm: 
	\begin{align*}
		\|f\|_{\Phi} = \inf\left\{\lambda > 0 : \int_{X} \Phi\left(\frac{|f(x)|}{\lambda}\right) d\mu(x) \leq 1\right\}.
	\end{align*}
	
	\item Normalized Luxemburg norm: 
	\begin{align*}
		\|f\|_{\Phi,\,B} = \inf\left\{\lambda > 0 : \frac{1}{\mu(B)} \int_{B} \Phi\left(\frac{|f(x)|}{\lambda}\right) d\mu(x) \leq 1\right\}.
	\end{align*}
	
	\item Generalized Luxemburg norm: 
	\begin{align*}
		\|f\|_{\Phi,\phi,B} = \inf\left\{\lambda > 0 : \frac{1}{\mu(B)\phi(\mu(B))} \int_{B} \Phi\left(\frac{|f(x)|}{\lambda}\right) d\mu(x) \leq 1\right\}.
	\end{align*}
\end{itemize}

\begin{example}
	The following are several examples:
	\begin{itemize}
		\item \(\Phi(t)=t^p\,(p>1)\,\Longrightarrow\,\|f\|_{\Phi}=\|f\|_{L^p},\,\|f\|_{\Phi,\,B}=\|f\|_{L^p,B}\). 
		\begin{itemize}
			\item \(\phi(t)=t^{-1}\,\Longrightarrow\,\|f\|_{\Phi,\phi,B}=\|f\|_{L^p(B)}\). 
			\item \(\phi(t)=1\,\Longrightarrow\,\|f\|_{\Phi,\phi,B}=\|f\|_{L^p,B}\). 
		\end{itemize}
		
		\item \(\Phi(t)=t\log^r(e+t)\,(r>0),\,\phi(t)=1\,\Longrightarrow\,\|f\|_{\Phi}=\|f\|_{L(\log L)^r},\,\|f\|_{\Phi,B}=\|f\|_{\Phi,\phi,B}=\|f\|_{L(\log L)^r,B}\). 
		
		\item \(\Phi(t)=e^{t^r}-1\,(r>0),\,\phi(t)=1\,\Longrightarrow\,\|f\|_{\Phi}=\|f\|_{\exp L^r},\,\|f\|_{\Phi,B}=\|f\|_{\Phi,\phi,B}=\|f\|_{\exp L^r,B}\). 
	\end{itemize}
\end{example}

Generalized Luxemburg norm possesses the following property.
\begin{prop}\label{propA}
	Let \((X,\,\mathfrak{M},\,\mu)\) be a measure space endowed with a ball-basis \(\mathfrak{B}\). Assume that for a given Young function \(\Phi \in \mathscr{Y}\), a function \(\phi\in\mathscr{G}\), and balls \(A,\,B\in\mathfrak{B}\), we have \(A \cap B \ne \varnothing\) and \(\mu(B) \leq C_{A,B}\mu(A)\) for some constant \(C_{A,B} \geq 1\). Then we have
	\begin{align*}
		\|f\cdot\mathbb{I}_B\|_{\Phi,\phi,A} \leq \max\{1,C^{-1}_{dc}\}\max\{1,C_{ic}\}C_{A,B}\|f\|_{\Phi,\phi,B}.
	\end{align*}
\end{prop}

\begin{proof}
	Let \(K=\|f\|_{\Phi,\phi,B}\). By the definition of the generalized Luxemburg norm
	\begin{align*}
		\int_{B} \Phi\left(\frac{|f(x)|}{K}\right) d\mu(x) \leq \mu(B)\phi(\mu(B)).
	\end{align*}
	Since \(\Phi\) is a Young function and we set \(C = \max\{1,C^{-1}_{dc}\}\max\{1,C_{ic}\}C_{A,B} \geq 1\). We next consider two cases:
	\begin{itemize}[leftmargin=1.5em]
		\item If \(\mu(A)\leq\mu(B)\), then \(\phi(\mu(A))\geq C_{dc} \phi(\mu(B))\). Consequently,
		\begin{align*}
			\frac{1}{\mu(A)\phi(\mu(A))}&\int_{A} \Phi\left(\frac{\left|f\cdot\mathbb{I}_B(x)\right|}{CK}\right) d\mu(x) \leq \frac{1}{\mu(A)\phi(\mu(A))}\cdot\frac{1}{C}\int_{A \cap B} \Phi\left(\frac{|f(x)|}{K}\right) d\mu(x)\\
			&\leq\frac{\mu(B)\phi(\mu(B))}{C\mu(A) \phi(\mu(A))} \leq \frac{\mu(B)}{\mu(A)C_{A,B}}\frac{1}{C_{dc}} \frac{1}{\max\{1,C^{-1}_{dc}\}} \leq 1.
		\end{align*}
		
		\item If \(\mu(B)\leq\mu(A)\), then \(\mu(B)\phi(\mu(B))\leq C_{ic} \mu(A)\phi(\mu(A))\). Hence,
		\begin{align*}
			\frac{1}{\mu(A)\phi(\mu(A))}&\int_{A} \Phi\left(\frac{\left|f\cdot\mathbb{I}_B(x)\right|}{CK}\right) d\mu(x) \leq \frac{1}{\mu(A)\phi(\mu(A))}\cdot\frac{1}{C}\int_{A \cap B} \Phi\left(\frac{|f(x)|}{K}\right) d\mu(x)\\
			&\leq\frac{\mu(B)\phi(\mu(B))}{C\mu(A) \phi(\mu(A))} \leq \frac{C_{ic}}{\max\{1,C_{ic}\}} \leq 1.
		\end{align*}
	\end{itemize}
	Therefore, we obtain \(\|f\cdot\mathbb{I}_B\|_{\Phi,\phi,A}\leq C\|f\|_{\Phi,\phi,B}\).
\end{proof}

\begin{prop}\label{propB}
	Let \((X,\,\mathfrak{M},\,\mu)\) be a measure space endowed with a ball-basis \(\mathfrak{B}\). For a Young function \(\Phi \in \mathscr{Y}\) and a function \(\phi\in\mathscr{G}\), suppose that balls \(A,\,B\in\mathfrak{B}\) satisfy \(A \subset B\). Then
	\begin{align*}
		\|f\|_{\Phi,\phi,A} \leq \frac{\max\{1,C_{ic}\} \mu(B)\phi(\mu(B))}{\mu(A)\phi(\mu(A))} \|f\|_{\Phi,\phi,B}.
	\end{align*}
\end{prop}

\begin{proof}
	Set \(K=\|f\|_{\Phi,\phi,B}\). From the definition of the generalized Luxemburg norm
	\begin{align*}
		\int_{B} \Phi\left(\frac{|f(x)|}{K}\right) d\mu(x) \leq \mu(B)\phi(\mu(B)).
	\end{align*}
	The inclusion \(A \subset B\) implies \(\mu(A)\leq\mu(B)\), and consequently
	\begin{align*}
		\mu(A)\phi(\mu(A)) \leq C_{ic} \mu(B)\phi(\mu(B)\leq \max\{1,C_{ic}\} \mu(B)\phi(\mu(B)).
	\end{align*}
	Define
	\begin{align*}
		C=\frac{\max\{1,C_{ic}\} \mu(B)\phi(\mu(B))}{\mu(A)\phi(\mu(A))} \geq 1.
	\end{align*}
	It follows that
	\begin{align*}
		\frac{1}{\mu(A)\phi(\mu(A))}&\int_{A} \Phi\left(\frac{\left|f(x)\right|}{CK}\right) d\mu(x) \leq \frac{1}{C\mu(A)\phi(\mu(A))}\int_{B} \Phi\left(\frac{|f(x)|}{K}\right) d\mu(x) \leq 1.
	\end{align*}
	Hence \(\|f\|_{\Phi,\phi,A} \leq C\|f\|_{\Phi,\phi,B}\), completing the proof.
\end{proof}

\begin{prop}\label{propC}
	Let \((X,\,\mathfrak{M},\,\mu)\) be a measure space endowed with a ball-basis \(\mathfrak{B}\). Given a Young function \(\Phi \in \mathscr{Y}\) and a function \(\phi\in\mathscr{G}\), assume that two balls \(A,\,B\in\mathfrak{B}\) satisfy \(A \subset B\). Then
	\begin{align*}
		\|f \, \mathbb{I}_A\|_{\Phi,\phi,B} \leq \max\{1,C_{ic}\} \|f\|_{\Phi,\phi,A}.
	\end{align*}
\end{prop}

\begin{proof}
	Set \(K = \|f\|_{\Phi,\phi,A}\). Because \(\mu(A) \leq \mu(B)\) and \(\phi \in \mathscr{G}\), we have \(\mu(A)\phi(\mu(A)) \leq C_{ic} \mu(B)\phi(\mu(B))\). Consequently,
	\begin{align*}
		\frac{1}{\mu(B)\phi(\mu(B))} &\int_B \Phi\left(\frac{f \, \mathbb{I}_A(x)}{K\max\{1,C_{ic}\}}\right)d\mu(x) \leq \frac{1}{\mu(B) \phi(\mu(B))} \frac{1}{\max\{1,C_{ic}\}}\int_A \Phi\left( \frac fK \right)d\mu \\
		& \leq \frac{\mu(A)\phi(\mu(A))}{\mu(B)\phi(\mu(B))} \frac{1}{ \max\{1,C_{ic}\}} \leq 1.
	\end{align*}
	Hence the result follows.
\end{proof}

The framework we have been working with is that of an abstract Orlicz-Morrey space endowed with a ball-basis. We give its precise definition below.

\begin{definition}\label{Orlicz-Morrey space}
	Let \((X,\,\mathfrak{M},\,\mu)\) be a measure space endowed with a ball-basis \(\mathfrak{B}\), and let \(\Phi \in \mathscr{Y}\) be a Young function and \(\phi \in \mathscr{G}\). The {\sffamily abstract Orlicz-Morrey space endowed with a ball-basis} is defined as  
	\begin{align*}
		L^{(\Phi,\phi)}(X) = \left\{ f \in L_{{\rm loc}}^1(X) : \|f\|_{L^{(\Phi,\phi)}(X)} < +\infty \right\} =: L^{(\Phi,\phi)},
	\end{align*}
	\begin{align*}
		\|f\|_{L^{(\Phi,\phi)}(X)} = \sup_{B \in \mathfrak{B}}\|f\|_{\Phi,\phi,B} =: \|f\|_{L^{(\Phi,\phi)}}.
	\end{align*}
\end{definition}

It is not difficult to verify that \(\|\cdot\|_{L^{(\Phi,\phi)}}\) indeed defines a norm and that \(L^{(\Phi,\phi)}\) is a Banach space. Furthermore, this framework unifies several important spaces: the abstract Lebesgue spaces, abstract Morrey spaces, abstract generalized Morrey spaces, and abstract Orlicz spaces all emerge as special cases of the ball-basis endowed abstract Orlicz-Morrey space. The inclusion relationships among these abstract function spaces are as follows.

\begin{figure}[H]
	\centering
	\begin{tikzpicture}[
		node distance=2cm,
		box/.style={draw, rectangle, minimum width=1.5cm, minimum height=1cm},
		arrow/.style={-Stealth, thick}
		]
		\node (1) {\(L^{(\Phi,\phi)}(X)\)};
		\node[below right=1.5cm of 1] (3) {\(L^{(q,\phi)}(X)\)};
		\node[right=1.5cm of 3] (4) {\(\mathcal{M}^p_q(X)\)};
		
		\coordinate (midpoint) at ($(3)!0.5!(4)$);
		\node[above=3cm of midpoint] (2) {\(L^\Phi(X)\)};
		
		\node[above right=1.5cm of 4] (5) {\(L^p(X)\) and \(L^\infty(X)\)};

		\draw[arrow] (1) -- node[above left] {\(\phi_1(t)\)} (2);
		\draw[arrow] (1) -- node[below left] {\(\Phi_3(t)\)} (3);
		\draw[arrow] (3) -- node[above] {\(\phi_2(t)\)} (4);
		\draw[arrow] (2) -- node[above right] {\(\Phi_1(t)\) and \(\Phi_2(t)\)} (5);
		\draw[arrow] (4) -- node[below right] {Values of \(p\) and \(q\)} (5);
		
	\end{tikzpicture}
	\caption{Relationships among various abstract function spaces endowed with a ball-basis}
\end{figure}

\begin{itemize}[leftmargin=2em]
	\item[{\large\(\clubsuit\)}]\,\,\,{\bf Abstract Orlicz spaces \(L^{\Phi}(X)\)} 
\end{itemize}

\begin{itemize}
	\item For \(\phi_1(t)=1/t\), we have \(L^{(\Phi,\phi_1)} = L^{\Phi}\).
\end{itemize}

\begin{definition}\label{Orlicz space}
	Let \((X,\,\mathfrak{M},\,\mu)\) be a measure space endowed with a ball-basis \(\mathfrak{B}\). Given a Young function \(\Phi\), define the {\sffamily abstract Orlicz space endowed with a ball-basis} as 
	\begin{align*}
		L^\Phi(X) = \left\{ f \in L_{{\rm loc}}^1(X) : \, \text{there exists} \, k > 0 \, \text{such that} \, \int_X \Phi(k|f(x)|) \, d\mu(x) < +\infty \right\} =: L^\Phi.
	\end{align*}
	The norm on \(L^\Phi\) is given by \(\|f\|_{L^{\Phi}} =\|f\|_{\Phi} \), where \(\|f\|_\Phi\) denotes the standard Luxemburg norm. Furthermore, \(L^\Phi\) is a Banach space. One readily verifies that \(\int_X\Phi\left(|f(x)|/\|f\|_{\Phi}\right) d\mu(x) \leq 1 \).
\end{definition}

Let \((X,\,\mathfrak{M},\,\mu)\) be a measure space endowed with a ball-basis \(\mathfrak{B}\). For the above Orlicz space.

\begin{itemize}[leftmargin=4em]
	\item[{\large\(\spadesuit\)}]\,\,\,{\bf Abstract Lebesgue Spaces \(L^p(X)\) and \(L^\infty(X)\)} 
\end{itemize}

\begin{itemize}
	\item If \(\Phi_1(t)=t^p\) with \(1 \leq p < +\infty\), then
	\begin{align*}
		L^p(X)=\left\{f:\,\|f\|_{L^p(X)}=\left(\int_X f^p d\mu\right)^{\frac{1}{p}} < +\infty \right\} =: L^p.
	\end{align*}
	
	\vspace{6pt}
	
	\item If \(\Phi_2(t)=\left\{ 
	\begin{array}{lc}
		0, ~~ & 0 \leq t \leq 1, \\
		+\infty,  & t>1. \\
	\end{array}
	\right.\), then 
	\begin{align*}
		L^\infty(X)=\left\{f:\,\|f\|_{L^\infty(X)}:=\esssup \limits_{x \in X}|f(x)| < +\infty\right\} =: L^\infty.
	\end{align*}
\end{itemize}

\begin{remark}
	For \(\phi \sim 1\), the Orlicz-Morrey space \(L^{(\Phi,\phi)}(X)\) is norm-equivalent to \(L^{\infty}(X)\).
\end{remark}

\begin{itemize}[leftmargin=2em]
	\item[{\large\(\clubsuit\)}]\,\,\,{\bf Abstract generalized Morrey spaces \(L^{(q,\phi)}(X)\)} 
\end{itemize}

\begin{itemize}
	\item If \(\Phi\) takes the form \(\Phi_3(t)=t^q\) for \(1 \leq q < +\infty\), then the identity \(L^{(\Phi_3,\phi)} = L^{(q,\phi)}\) holds.
\end{itemize}

\begin{definition}\label{G Morrey space}
	Let \((X,\,\mathfrak{M},\,\mu)\) be a measure space endowed with a ball-basis \(\mathfrak{B}\). For a fixed function \(\phi\in\mathscr{G}\), the {\sffamily abstract generalized Morrey space endowed with a ball-basis} is defined by  
	\begin{align*}
		L^{(q,\phi)}(X) = \left\{ f\in L^q_{{\rm loc}}(X) :\, \sup_{B \in \mathfrak{B}} \, \left(\frac{1}{\mu(B)\phi(\mu(B))} \int_B |f|^q d\mu\right)^{\frac{1}{q}} < +\infty \right\} =: L^{(q,\phi)}.
	\end{align*}
	Its norm is given by
	\begin{align*}
		\|f\|_{L^{(q,\phi)}}= \sup_{B\in\mathfrak{B}} \,\left(\frac{1}{\mu(B)\phi(\mu(B))} \int_B |f|^q d\mu\right)^{\frac{1}{q}}.
	\end{align*}
	Moreover, \(L^{(q,\phi)}\) is a Banach space.
\end{definition}

\begin{itemize}[leftmargin=4em]
	\item[{\large\(\spadesuit\)}]\,\,\,{\bf Abstract Morrey spaces \(\mathcal{M}^p_q(X)\)} 
\end{itemize}

\begin{itemize}
	\item \(L^{(q,\phi)}\) coincides with the Morrey space \(\mathcal{M}^p_q\) when \(\phi_2(t)=t^{-q/p}\) and \(1 \leq q \leq p < +\infty\).
\end{itemize}

\begin{definition}\cite{Shan2026}\label{Morrey space}
	Let \((X,\,\mathfrak{M},\,\mu)\) be a measure space endowed with a ball-basis \(\mathfrak{B}\). Define the {\sffamily abstract Morrey space endowed with a ball-basis} as  
	\begin{align*}
		\mathcal{M}^p_q(X) = \left\{ f\in L^q_{{\rm loc}}(X) : \, \sup_{B\in\mathfrak{B}} \, \mu(B)^{\frac{1}{p}-\frac{1}{q}}\, \left(\int_B |f|^q d\mu \right)^\frac{1}{q} < +\infty \right\} =: \mathcal{M}^p_q.
	\end{align*}
	It is equipped with the norm  
	\begin{align*}
		\|f\|_{\mathcal{M}^p_q} = \sup_{B\in\mathfrak{B}} \, \mu(B)^{\frac{1}{p}-\frac{1}{q}}\, \left(\int_B |f|^q d\mu \right)^\frac{1}{q}.
	\end{align*}
	Moreover, \(\mathcal{M}_{q}^{p}\) is a Banach space.
\end{definition}

\begin{itemize}[leftmargin=6em]
	\item[{\large\(\bigstar\)}]\,\,\,{\bf Abstract Lebesgue spaces} 
\end{itemize}

\begin{itemize}
	\item In the case \(1 \leq q = p < +\infty\), \(\mathcal{M}^p_q\) coincides with \(L^p\). 
	
	\item For parameters satisfying \(1 \leq q \leq p = +\infty\), \(\mathcal{M}^p_q\) is identified as the space \(L^\infty\).
\end{itemize}

Let \((X,\,\mathfrak{M},\,\mu)\) be a measure space endowed with a ball-basis \(\mathfrak{B}\). For a measurable set \(E \subset X\), a measurable function \(f\), and a number \(t > 0\), we denote  
\begin{align*}
	d(E,f,t)=\mu(\{x \in E:|f(x)|>t\}).
\end{align*}
When \(E=X\), we simply write \(d(f,t)\) instead of \(d(X,f,t)\).

Given a Young function \(\Phi \in \mathscr{Y}\), a function \(\phi \in \mathscr{G}\), and a ball \(B \in \mathfrak{B}\), we define  
\begin{align*}
	\|f\|_{w,\Phi,\phi,B} = \inf\left\{\lambda > 0 : \sup_{t>0} \frac{t d(B, \Phi(|f|/\lambda), t)}{\mu(B)\phi(\mu(B))} \leq 1\right\}.
\end{align*}
A direct verification shows that \(\|f\|_{w,\Phi,\phi,B} \leq \|f\|_{\Phi,\phi,B}\) and that
\begin{align*}
	\sup_{t>0}t d(E, \Phi(|f|), t) = \sup_{t>0}t d(E, f, \Phi^{-1}(t)) = \sup_{t>0}\Phi(t) d(E, f, t).
\end{align*}

\begin{definition}\label{weak Orlicz-Morrey space}
	Let \((X,\,\mathfrak{M},\,\mu)\) be a measure space endowed with a ball-basis \(\mathfrak{B}\), and let \(\Phi \in \mathscr{Y}\) and \(\phi \in \mathscr{G}\). The {\sffamily abstract weak Orlicz-Morrey space endowed with a ball-basis} is defined by
	\begin{align*}
		wL^{(\Phi,\phi)}(X) = \left\{ f \in L^{1}_{\rm loc}(X) : \|f\|_{wL^{(\Phi,\phi)}} < +\infty \right\} =: wL^{(\Phi,\phi)},
	\end{align*}
	where
	\begin{align*}
		\|f\|_{wL^{(\Phi,\phi)}}:= \sup_{B \in \mathfrak{B}}\|f\|_{w, \Phi, \phi, B}.
	\end{align*}
\end{definition}
It can be readily verified that \(\|\cdot\|_{wL^{(\Phi,\phi)}}\) is a {\sffamily quasi-norm}. With this quasi-norm, \(wL^{(\Phi,\phi)}(X)\) becomes a {\sffamily quasi-Banach space}. In particular, the following quasi-triangle inequality holds:
\begin{align*}
	\|f+g\|_{wL^{(\Phi,\phi)}} \leq 2 \left(\|f\|_{wL^{(\Phi,\phi)}} + \|g\|_{wL^{(\Phi,\phi)}}\right), \qquad \forall f,g \in wL^{(\Phi, \phi)}(X).
\end{align*}

\end{appendices}

\bigskip

\noindent{\bf\Large Declarations}

\noindent{\bf Acknowledgements } 
The authors would like to thank the editors and reviewers for careful reading and valuable comments, which lead to the improvement of this paper.

\medskip
\noindent{\bf Mathematics Subject Classification(2020)}
Primary 42B20; Secondary 42B25; 42B35; 46E30.

\medskip
\noindent{\bf Data, materials, and code availability}
Our manuscript has no associated data, materials, and code.

\medskip 
\noindent{\bf Funding information} The research was supported by the National Natural Science Foundation of China (No. 12461021) and the Research Innovation Program for Postgraduates of Xinjiang Uygur Autonomous Region (No. XJ2026G070).

\medskip
\noindent{\bf Contributions} All authors participated in the conception and design of the study and reviewed the manuscript.

\medskip
\noindent{\bf Conflict of interest} The authors declare that they have no conflict of interest.

\medskip
\noindent{\bf Author contributions}
All authors participated in the conception and design of the study and reviewed the manuscript.

\medskip
\noindent{\bf Consent for publication}
The authors agree to the publication.

\medskip
\noindent{\bf Ethical statement}
{\bf (1) Conflict of Interest:} The authors declare that they have no known competing financial interests or personal relationships that could have appeared to influence the work reported in this paper. {\bf (2) Data Availability:} No new data were generated or analyzed in support of this research. {\bf (3) Authorship and Originality:} This manuscript is the authors' original work and has not been published or submitted elsewhere. All authors have read and approved the final version and agree to be accountable for all aspects of the work.

\end{document}